\documentclass[10pt]{amsart}
\usepackage{amssymb, enumitem}
\usepackage[all]{xy}\usepackage{tikz} 
\usepackage{hyperref}
\usepackage{aliascnt}
\usepackage{verbatim}
\usepackage[english]{babel}
\usepackage{enumitem}
\usepackage{mathtools}

\def\today{\number\day\space\ifcase\month\or   January\or February\or
   March\or April\or May\or June\or   July\or August\or September\or
   October\or November\or December\fi\   \number\year}

\theoremstyle{definition}
\newtheorem{lma}{Lemma}[section]

\newaliascnt{thmCt}{lma}
\newtheorem{thm}[thmCt]{Theorem}
\aliascntresetthe{thmCt}

\newaliascnt{corCt}{lma}
\newtheorem{cor}[corCt]{Corollary}
\aliascntresetthe{corCt}

\newaliascnt{propCt}{lma}
\newtheorem{prop}[propCt]{Proposition}
\aliascntresetthe{propCt}

\newtheorem*{thm*}{Theorem}
\newtheorem*{cor*}{Corollary}
\newtheorem*{prop*}{Proposition}

\newcounter{theoremintro}
\newtheorem{thmintro}[theoremintro]{Theorem}

\newaliascnt{pgrCt}{lma}

\aliascntresetthe{pgrCt}

\newaliascnt{dfCt}{lma}
\newtheorem{df}[dfCt]{Definition}
\aliascntresetthe{dfCt}

\newaliascnt{remCt}{lma}
\newtheorem{rem}[remCt]{Remark}
\aliascntresetthe{remCt}

\newaliascnt{remsCt}{lma}

\aliascntresetthe{remsCt}

\newaliascnt{egCt}{lma}

\aliascntresetthe{egCt}

\newaliascnt{egsCt}{lma}

\aliascntresetthe{egsCt}

\newaliascnt{qstCt}{lma}

\aliascntresetthe{qstCt}

\newaliascnt{pbmCt}{lma}
\newtheorem{pbm}[pbmCt]{Problem}
\aliascntresetthe{pbmCt}

\newaliascnt{notaCt}{lma}
\newtheorem{nota}[notaCt]{Notation}
\aliascntresetthe{notaCt}

\theoremstyle{theorem}
\newtheorem{claim}{Claim}[thmCt]

\newcommand{\beq}{\begin{equation}}
\newcommand{\eeq}{\end{equation}}
\newcommand{\beqa}{\begin{eqnarray*}}
\newcommand{\eeqa}{\end{eqnarray*}}
\newcommand{\bal}{\begin{align*}}
\newcommand{\eal}{\end{align*}}
\newcommand{\bi}{\begin{itemize}}
\newcommand{\ei}{\end{itemize}}
\newcommand{\be}{\begin{enumerate}}
\newcommand{\ee}{\end{enumerate}}

\newcommand{\ep}{\varepsilon}

\newcommand{\Z}{{\mathbb{Z}}}
\newcommand{\C}{{\mathbb{C}}}
\newcommand{\N}{{\mathbb{N}}}

\newcommand{\M}{{\mathcal{M}}}

\newcommand{\R}{{\mathcal{R}}}

\newcommand{\B}{{\mathcal{B}}}
\newcommand{\U}{{\mathcal{U}}}

\newcommand{\id}{{\mathrm{id}}}

\newcommand{\spec}{{\mathrm{sp}}}

\newcommand{\diag}{{\mathrm{diag}}}
\newcommand{\supp}{{\mathrm{supp}}}

\newcommand{\Aut}{{\mathrm{Aut}}}

\newcommand{\Ad}{{\mathrm{Ad}}}

\newcommand{\dimRok}{{\mathrm{dim}_\mathrm{Rok}}}

\newcommand{\ca}{$C^*$-algebra}
\newcommand{\cas}{$C^*$-algebras}
\newcommand{\uca}{unital $C^*$-algebra}

\numberwithin{equation}{section}

\newcommand{\I}{\infty}

\title[Str. outer actions of am. groups on $\mathcal{Z}$-stable nuc.
$C^*$-algebras]{Strongly outer actions of amenable groups on $\mathcal{Z}$-stable nuclear
$C^*$-algebras}

\date{\today}

\thanks{The first named author was partially supported
by the Deutsche Forschungsgemeinschaft (SFB 878), and by a Postdoctoral Research Fellowship
from the Humboldt Foundation. This work was supported by GIF grant 1137/2011,
 by Israel Science Foundation Grant 476/16 and by the Fields Institute. The third author was supported
 by the European Union’s Horizon 2020 Marie Sklodowska-Curie grant No 891709.}

\author[Eusebio Gardella]{Eusebio Gardella}
\address{Eusebio Gardella
Mathematisches Institut, Fachbereich Mathematik und Informatik der
Universit\"at M\"unster, Einsteinstrasse 62, 48149 M\"unster, Germany.}
\email{gardella@uni-muenster.de}
\urladdr{www.math.uni-muenster.de/u/gardella/}

\author[Ilan Hirshberg]{Ilan Hirshberg}
\address{Ilan Hirshberg
Department of Mathematics, Ben-Gurion University of the Negev, Be’er
Sheva, Israel.}
\email{ilan@math.bgu.ac.il}
\urladdr{www.math.bgu.ac.il/~ilan/}

\author[Andrea Vaccaro]{Andrea Vaccaro}
\address{Andrea Vaccaro
Department of Mathematics, Universit\'e de Paris,
Paris, France.}
\email[]{vaccaro@imj-prg.fr}
\urladdr{https://sites.google.com/view/avaccaro/home}

\begin{document}
\begin{abstract}
Let $A$ be a separable, unital, simple, $\mathcal{Z}$-stable, nuclear 
$C^*$-algebra,
and let $\alpha\colon G\to \mathrm{Aut}(A)$ be an action of a discrete, countable, 
amenable group. Suppose that the orbits of the action of $G$ on $T(A)$ are
finite and that their cardinality is bounded. We show that the following are 
equivalent:
\begin{enumerate}\item $\alpha$ is strongly outer;
	\item $\alpha\otimes\mathrm{id}_{\mathcal{Z}}$ has the weak tracial Rokhlin 
	property.\end{enumerate}
If $G$ is moreover residually finite, the above conditions are also equivalent 
to
\begin{enumerate}\item[(3)] $\alpha\otimes\mathrm{id}_{\mathcal{Z}}$ has finite 
Rokhlin dimension (in fact, at most 2).\end{enumerate}

If $\partial_eT(A)$ is furthermore compact, has finite covering dimension, and the orbit space
$\partial_eT(A)/G$ is Hausdorff,
we generalize results by Matui and Sato to show that
$\alpha$ is cocycle conjugate to $\alpha\otimes\mathrm{id}_{\mathcal{Z}}$, even
if $\alpha$ is not strongly outer. In 
particular, in this case the equivalences above
hold for $\alpha$ in place of $\alpha\otimes\mathrm{id}_{\mathcal{Z}}$.
In the course of the proof, we develop equivariant versions of complemented 
partitions of unity and uniform 
property $\Gamma$ as technical tools of independent interest.
\end{abstract}
\maketitle

\tableofcontents

\section{Introduction}
\renewcommand*{\thetheoremintro}{\Alph{theoremintro}}

This paper concerns the structure of group actions by discrete, countable, a\-me\-na\-ble groups 
on separable, simple, unital, nuclear, $\mathcal{Z}$-stable $C^*$-algebras.
One of the main themes of research in operator algebra theory in the past 
several decades has been the Elliott program to classify simple, nuclear, 
separable $C^*$-algebras by $K$-theoretic data. The Elliott program is  
now essentially complete: simple nuclear separable 
$C^*$-algebras which satisfy the Universal Coefficient Theorem and absorb the
Jiang-Su algebra $\mathcal{Z}$
(\cite{jiangsu}) are classified via the Elliott invariant. Moreover, this
classification cannot be extended to the non-$\mathcal{Z}$-stable case without 
enlarging the invariant and without significant new ideas. The property of 
$\mathcal{Z}$-stability is a regularity condition for simple $C^*$-algebras, 
analogous to the McDuff property for type II$_1$ factors. We refer the 
reader to \cite{WinterICM} for a recent survey and further references
concerning the classification program and Toms-Winter regularity, as a 
detailed exposition of these topics is beyond the scope of this paper.

The analysis of group actions on operator algebras is a natural and important line of research which has been studied 
intensively 
for $C^*$-algebras as well as in von Neumann algebra theory. It is 
closely related to the classification program discussed above,
particularly via the crossed product construction. Specifically,
given the role of $\mathcal{Z}$-stability in the Elliott program (or the related 
regularity properties in the Toms-Winter conjecture), it is important to 
understand how robust the class of simple, nuclear, separable,
$\mathcal{Z}$-stable $C^*$-algebras is, in particular
with respect to standard constructions such as crossed 
products. The following is an important open problem in this 
context.

\begin{pbm}
Let $A$ be a separable, simple, nuclear, $\mathcal{Z}$-stable 
$C^*$-algebra, and let $\alpha 
\colon G \to \Aut(A)$ be an action of a discrete countable amenable group $G$. 
Find conditions that ensure that $A\rtimes_\alpha G$ is also separable, 
simple, nuclear, and $\mathcal{Z}$-stable.
\end{pbm}

In the above setting, nuclearity and separability are always guaranteed
since $G$ is amenable and countable. 
By a celebrated result of Kishimoto \cite{Kis_simplicity}, $A\rtimes_\alpha G$
is simple whenever $\alpha_g$ is outer for all $g \in G 
\setminus \{1\}$.
The main task, then, is to find general conditions that entail preservation of 
$\mathcal{Z}$-stability.

The Rokhlin property and its various generalizations form a collection of 
regularity conditions for group actions on \ca s,
whose roots stem from the Rokhlin Lemma in Ergodic Theory. This result was first
extended to 
von Neumann algebras, starting with Connes' work for the case of a single 
automorphism (\cite[Theorem 1.2.5]{connes}; see also \cite[Chapter XVII, Lemma 2.3]{take3}). Connes'
result was later generalized by Ocneanu
in \cite[Section 6.1]{ocneanu}, and used to prove that all outer 
actions of a discrete countable amenable group $G$ on the hyperfinite II$_1$ factor are 
cocycle conjugate.
In the context of actions on \ca s, early
works include the studies of cyclic and finite group actions on UHF-algebras by
Herman and Jones (\cite{hermjon1,hermjon2}), and Herman and Ocneanu (\cite{herocn}), and later
for automorphisms on UHF and A$\mathbb{T}$-algebras by Kishimoto 
(\cite{kishimoto,kishimoto:I,kishimoto2}). A connection with permanence of 
$\mathcal{Z}$-stability was made in \cite{hw07}, where it was shown that for 
actions of the integers, the reals or finite groups, $\mathcal{Z}$-stability is 
preserved when passing to the crossed product, provided the action has the 
Rokhlin property.

Although the Rokhlin property is relatively common for actions of the integers, 
there are significant $K$-theoretic obstructions
for finite group actions (and hence actions of groups 
which have torsion). This was studied in depth by Izumi \cite{Izu_finiteI_2004,
	Izu_finiteII_2004} and spurred additional work \cite{San_crossed_2015, GarSan_equivariant_2016}.
Attempts to circumvent impediments of this sort led Phillips to 
introduce the
\emph{tracial Rokhlin property} \cite{phil:tracial},
where the projections in the Rokhlin property are assumed to have a 
leftover which
is small in trace.
Among other applications,
the tracial Rokhlin property has been used in
\cite{EchLucPhiWal_structure_2010} to study fixed point
algebras of the irrational rotation algebra $A_\theta$ under certain canonical 
actions of finite cyclic groups.

The tracial Rokhlin property does not bypass the most obvious obstruction to
admitting Rokhlin actions: the existence of nontrivial projections. The need to 
study
weaker versions of these properties led to two further generalizations. The 
first
one, called the
\emph{weak tracial Rokhlin property}, which replaces projections with positive 
elements, has been considered in \cite{HirOro_tracially_2013, 
satoRoh,matuisato:I,matuisato:II,wang,GarHirSan_rokhlin_2017}.

A different approach was taken in a paper by the second author, Winter, and 
Zacharias \cite{hwz},
who introduced the notion of \emph{Rokhlin dimension}. In this formulation,
the partition of unity appearing in the Rokhlin property is replaced by a 
multi-tower partition
of unity consisting of positive contractions, the elements of each tower being 
indexed by the group elements and
permuted by the group action. Rokhlin dimension zero then corresponds to the 
Rokhlin property, but the extra
flexibility makes finiteness of the Rokhlin dimension a much more 
common feature. This notion has primarily been used as a tool to show that various 
structural
properties of interest (such as $\mathcal{Z}$-stability or finite nuclear dimension) pass from an algebra to the crossed product. 
Rokhlin dimension has been extended and studied for actions of various 
classes of groups; the generalization which is pertinent for this paper is in 
work of Szab{\'o}, Wu, and Zacharias for residually finite groups 
(\cite{swz}). We refer the reader to 
\cite{Gar_rokhlin_2017,HirPhi_rokhlin_2015,HirSzaWinWu_rokhlin_2017,GarKalLup_rokhlin_19} for further
generalizations.

This work focuses on actions on simple \ca s, and aims to improve upon related 
works by Matui-Sato (\cite{matuisato:I,matuisato:II}) and by Liao (\cite{liao,liaoZm}). We study the relationships 
between strong outerness, the weak tracial Rokhlin property and
finite Rokhlin dimension by showing that they are equivalent in many cases of 
interest.
More specifically, we obtain the following main results.

\begin{thmintro} \label{thmB}
	Let $A$ be a separable, simple, nuclear, unital, stably finite, 
	\ca, let $G$ be a countable, discrete, amenable group,
	and let $\alpha\colon G \to \Aut(A)$ be an action. 
	Suppose that the orbits of the action
	induced by $\alpha$ on $T(A)$ are
finite and that their cardinality is uniformly bounded. 
	Then the following are equivalent:
	\be\item \label{B:1} $\alpha$ is strongly outer.
	\item \label{B:2} $\alpha\otimes\id_{\mathcal{Z}}$ has the weak tracial Rokhlin property.
	\ee
When $G$ is residually finite, then the above are also equivalent to:
\be\setcounter{enumi}{2}
	\item \label{B:3} $\alpha\otimes\id_{\mathcal{Z}}$ has finite 
	Rokhlin dimension.
	\item \label{B:4} $\alpha\otimes\id_{\mathcal{Z}}$ has  
	Rokhlin dimension at most 2.
	\ee
\end{thmintro}
This theorem is restated as \autoref{thm:mainRokDim} and proved in 
\autoref{s6}. The definitions of strong outerness, the weak tracial Rokhlin 
property, and Rokhlin dimension are provided in \autoref{s.preliminaries}. 

The first precursor of this result is Theorem 5.5 in 
\cite{EchLucPhiWal_structure_2010}, where (1) $\Leftrightarrow$ (2) 
is shown under the 
additional
assumptions that $A$ has tracial rank zero, $A$ has a unique tracial state, and 
$G$ is finite. Matui and Sato proved \eqref{rokdim:1} $\Leftrightarrow$ 
\eqref{rokdim:2} in the case that $A$ is nuclear and has finitely many 
extreme 
tracial states, and the group
$G$ is elementary amenable (\cite[Theorem~3.6]{matuisato:II}).
This was later extended by Wang to all discrete countable amenable groups 
in
\cite[Theorem 3.8]{wang}. Here, we remove all the smallness assumptions on the 
extreme points of $T(A)$, although we still have some non-trivial requirements 
on the size of the orbits
for the induced action on $T(A)$. This is made possible by developing an 
equivariant version of complemented partitions of unity from \cite{cpou}; see more on this below.

The equivalence \eqref{rokdim:1} $\Leftrightarrow$ \eqref{rokdim:3} of Theorem 
\ref{thmB}
generalizes Liao's work in \cite{liao, liaoZm}, where a 
similar result is proved for $\Z^m$-actions, under the additional 
assumptions that $T(A)$ is a Bauer simplex with finite dimensional extreme 
boundary, and that the group acts trivially in $T(A)$.
Our method of proof differs significantly from Liao's
(other than in the use
of complemented partitions of unity to remove the topological assumptions on 
$\partial_e T(A)$), in that we 
obtain the Rokhlin towers by embedding suitable model actions
on dimension drop algebras into the central sequence algebra of $A$.
The advantage of our approach is that it does not 
require any restrictions on the group: in particular, we are able to treat
groups with torsion, as well as groups that are not finitely generated. (The 
application of property (SI) in~\cite[Theorem 6.4]{liao} makes essential use of 
the fact that $\Z$ has no torsion, and the same applies for the generalization 
to $\Z^m$ in \cite{liaoZm}.)

Our next main theorem improves upon 
results of Sato from \cite{sato19} (which in turn generalizes results from 
\cite{matuisato:II}). The main theorem in \cite{sato19} requires for the group to
act trivially on the trace space, whereas we can weaken this assumption 
to a condition on the orbits analogous to the one in Theorem \autoref{thmB}.

\begin{thmintro}[\autoref{thm:equiZstab}] \label{thmC}
Let $A$ be a separable, simple, nuclear, stably finite, $\mathcal{Z}$-stable, infinite-dimensional, \uca.
Let $G$ be a countable, discrete, amenable group,  
and let $\alpha\colon G\to\Aut(A)$ be an action.
Suppose that $\partial_eT(A)$ is compact, that $\dim(\partial_eT(A))<\I$,
that the orbits of the induced action of $G$ on $\partial_eT(A)$ are
finite with uniformly bounded cardinality, and that the orbit space $\partial_eT(A)/G$
is Hausdorff. Then $\alpha$ is cocycle conjugate to
$\alpha\otimes\id_{\mathcal{Z}}$.
\end{thmintro}

The conditions on the trace space are met, for instance, when the $G$-action induced
by $\alpha$ on $\partial_eT(A)$ factors through a finite group action.
We do not know whether the restriction on the topology of the trace space or the way in which the group acts 
on it can be relaxed.
Some progress has been made in the recent paper
\cite{wou}, where the methods here developed are combined with arguments from topological dynamics to show equivariant $\mathcal{Z}$-stability for
actions of $\mathbb{Z}$, assuming that 
$\partial_e T(A)$ is compact and finite-dimensional.

The arguments in this paper make an essential use of an equivariant version of 
uniform property $\Gamma$ and complemented partitions of unity (CPoU). The 
non-equivariant versions are methods introduced in \cite{cpou} 
in order to prove one of the remaining implications of the Toms-Winter 
conjecture, and were further developed in \cite{uniformGamma}. Roughly 
speaking, complemented partitions of unity provide a technique for globalizing
properties which occur fiber-wise in
the von Neumann algebras associated to the GNS representations of the traces, to 
properties holding uniformly over all traces; we adapt this method to allow 
gluing dynamical properties. As 
this is of independent interest, we record here what is proved in this respect,
which can be thought of as a dynamical analogue of \cite[Theorem 
4.6]{uniformGamma}.
The definitions of the terms in the theorem below appear in the relevant parts 
of the paper. 
\begin{thmintro}\label{thmA}
	Let $A$ be a separable, simple, nuclear, unital, stably finite, \ca\
	with no finite-dimensional
	quotients. Let $G$ be a countable, discrete, amenable group
	and let $\alpha\colon G \to \text{Aut}(A)$ be an action such that the 
	induced action
	on $T(A)$ has finite orbits bounded in size by a uniform constant $M > 0$.
	Then the following are equivalent
	\begin{enumerate}
		\item \label{gamma.cpou:1} $(A,\alpha)$ has uniform property $\Gamma$.
		\item \label{gamma.cpou:2} $(A, \alpha)$ has complemented partitions of 
		unity with constant $M$.
		\item \label{gamma.cpou:3} For every $n \in \N$ there is a unital 
		embedding of the matrix algebra $M_n \to (A^\U \cap A')^{\alpha^\U}$.
	\end{enumerate}
	If $A$ is also $\mathcal{Z}$-stable and simple, then the 
	previous conditions are equivalent to
	\begin{enumerate}
		\setcounter{enumi}{3}
		\item \label{gamma.cpou:4} $(A,\alpha)$ is cocycle conjugate to $(A 
		\otimes \mathcal{Z}, \alpha \otimes \text{id}_\mathcal{Z})$.
	\end{enumerate}
\end{thmintro}
Theorem \ref{thmA} is the combination of 
\autoref{thm:mainCPoU-parts-1-3},
which covers the equivalence of the first three conditions, and 
\autoref{thm:mainCPoU}, which adds the last condition.

The paper is organized as follows. In \autoref{s3} and \autoref{s4} we modify 
ideas from \cite{cpou} in order to develop equivariant analogs of uniform 
property $\Gamma$ and complemented partitions of unity. The main goal in 
\autoref{s5} is to use techniques from the previous two sections to obtain that, 
under our assumptions, we can equivariantly embed model actions on the 
hyperfinite II$_1$-factor into the central sequence algebra of uniform-tracial ultrapowers. 
\autoref{s.model} is a technical section, devoted to constructing model actions 
of finite groups which have Rokhlin-type towers on dimension drop algebras; 
those are needed in order to lift actions from the uniform-tracial central 
sequence algebra to the norm central sequence algebra. \autoref{s6} contains the 
proofs of Theorem \ref{thmB} and Theorem \ref{thmA}. The last section provides 
a proof of Theorem \ref{thmC}.

\section{Preliminaries} \label{s.preliminaries}

\subsection{Trace norms.}\label{ss2.1}
For a unital \ca\ $A$, the \emph{trace space} of $A$, denoted $T(A)$, is the compact
subspace of the dual of $A$ (endowed with the weak* toplogy) consisting of all
states $\tau$ of $A$ such that $\tau(ab) = \tau(ba)$ for all $a,b \in A$.

We say that a trace $\tau \in T(A)$ is \emph{faithful} if $\tau(a^*a) > 0$ for all
$a \in A\setminus\{0\}$. Given a trace $\tau \in T(A)$, the 
associated \emph{trace seminorm} $\| \cdot \|_{2,\tau}$ on $A$ is given by
\[
\| a \|_{2,\tau} = \tau(a^*a)^{1/2},
\]
for all $a\in A$. For a closed subset 
$T \subseteq T(A)$, we set
$\| \cdot \|_{2,T} = \sup_{\tau \in T} \| \cdot \|_{2,\tau}$,
which is a seminorm on $A$.
We use the abbreviation $\| \cdot \|_{2,u}$ for $\| \cdot \|_{2, T(A)}$. Notice that
$\| \cdot \|_{2, T}$ is in fact a norm if $T$ contains at least one faithful trace. This is always
the case, for instance, when $A$ is simple and admits a trace.

Let $G$ be a discrete group, and let 
$\alpha\colon G\to\Aut(A)$ be an action. Then $\alpha$ naturally
induces an action $\alpha^*$ 
of $G$ on $T(A)$ by affine homeomorphisms\footnote{The action $\alpha^*$ is the restriction of the dual of $\alpha$ to $T(A)\subseteq A^*$, hence the notation.},
given by $\alpha^*_g(\tau)=\tau\circ\alpha_{g^{-1}}$ for 
all $g\in G$ and all $\tau\in T(A)$. 
We say that a trace
$\tau\in T(A)$ is $\alpha$-invariant if $\tau\circ\alpha_g=\tau$
for all $g\in G$, and we denote by $T(A)^\alpha\subseteq T(A)$ the 
space of $\alpha$-invariant traces. Note that $T(A)^\alpha$ is always non-empty
if $G$ is amenable (\cite[Theorem 1.3.1]{runde}). Given $\tau \in T(A)$ such that the orbit
$G \cdot \tau$ is finite, we set
\[
\tau^\alpha := \frac{1}{| G \cdot \tau |} \sum_{\sigma \in G \cdot \tau} \sigma.
\]
If $T\subseteq T(A)$ is $G$-invariant, then $\alpha_g$ is
isometric with respect to $\|\cdot\|_{2,T}$ for all $g\in G$. 

\subsection{Ultrapowers.} 
\label{prelim:ultrapowers}
Let $A$ be a \ca, let $G$ be a discrete group, and let 
$\alpha\colon G\to\Aut(A)$ be an action. 
We denote by $\ell^\I(A)$ the \ca\ of all 
bounded sequences in $A$ with the supremum norm, endowed with 
the $G$-action given by pointwise application of $\alpha$.
For a free ultrafilter $\mathcal{U}$ on $\N$ (which we will fix 
throughout),
set 
\[c_\U(A) = \{(a_n) \in \ell^\infty(A) \colon \lim_{n \to \U} \| a_n \| = 
0\}.\] 
This is a closed, two-sided $G$-invariant ideal in $\ell^\I(A)$. 
We define
the \emph{norm ultrapower} of $A$ to be the quotient
$A_\mathcal{U}= \ell^\infty(A)/ c_\U(A)$, and denote by $\alpha_\U\colon 
G\to\Aut(A_\U)$ the induced action. We denote by 
$\pi_A\colon \ell^\I(A)\to A_\U$ the equivariant quotient map.

Given a closed subset $T\subseteq T(A)$,
we set
\[
J_T= \{(a_n)_{n \in \N} \in A_\mathcal{U}\colon \lim_{n \to \mathcal{U}} \| a_n \|_{2,T} = 0\}.
\]
Then $J_T$ is a closed, two-sided, ideal in $A_\U$, and it is 
$G$-invariant if $T$ is. If $\tau \in T(A)$, we abbreviate $J_{\{\tau\}}$ 
to $J_\tau$.
For $T=T(A)$, we abbreviate $J_{T(A)}$ to $J_A$, and call it
the \emph{trace kernel ideal}. The associated quotient 
$$
A^\mathcal{U} = A_\mathcal{U} / J_A
$$
is called the \emph{uniform tracial ultrapower} of $A$. We
denote by $\alpha^\U\colon G\to\Aut(A^\U)$ the induced action, 
and by $\kappa_A\colon A_\U\to A^\U$ the equivariant quotient map.
We abbreviate $\kappa_A$ to $\kappa$ whenever the algebra $A$ is 
clear from the context.
Given a subset $S \subseteq A_\U$, the commutant of $S$ in $A_\U$
is denoted by $A_\U \cap S'$, and we use similar notation for 
subsets of $A^\U$.
The following useful fact will be used repeatedly;
see \cite[Proposition 4.5, Proposition 4.6]{kirchror}:

\begin{lma}\label{lma:SurjRelComm}
Let $A$ be a separable \ca, let $G$ be a discrete group, let $\alpha\colon
G\to\Aut(A)$ be an action, let $S\subseteq A_\U$ be a separable $G$-invariant 
subset, and set $\overline{S}=\kappa(S)$. 
Then $\kappa$ restricts to a surjective, 
equivariant map 
\[\kappa\colon (A_\U\cap S',\alpha_\U)\to (A^\U\cap \overline{S}',\alpha^\U).
\]
\end{lma}

Like the norm ultrapower, the uniform tracial ultrapower of a \ca\ $A$ satisfies
countable saturation properties that allow us, via reindexing and diagonal arguments, to derive
exact statements in $A^\U$ from approximations in $\| \cdot \|_{2, T_\U(A)}$ or in $\| \cdot \|_{2,T(A)}$. For future reference, we isolate this in the 
following remark.

\begin{rem} \label{saturation}

The notion of saturation is a fundamental and classical concept from model theory,
which has also been formalized for \cas\ (see \cite[Section 4.3]{modeltheory}).
Among operator algebraists, all instances of saturation in the context of ultrapowers
are usually reduced to an application of a technical lemma known as
Kirchberg's $\varepsilon$-test \cite[Lemma A.1]{kir}.
We also refer to this technical tool when invoking `countable saturation' in our proofs.
\end{rem}

Given a sequence $(\tau_n)_{n\in\N}$ in $T(A)$, there is a trace 
$\tau\in T(A_\U)$ given by 
$\tau(a)=\lim_{n\to\U}\tau_n(a_n)$ whenever $(a_n)_{n\in\N}\in \ell^\I(A)$
is a representing sequence for $a$. 
We call traces of this form \emph{limit traces},
and denote by $T_\U(A)$ the set of all limit traces on $A_\U$.
Since any limit trace vanishes on $J_A$, 
with a slight abuse of notation we also regard 
the elements in $T_\U(A)$ as traces over $A^\U$.
We denote by
$T^\alpha_\mathcal{U}(A)$ the set of all traces in $T_\mathcal{U}(A)$
which arise from sequences of traces in $T(A)^\alpha$. (This set 
should not be confused with the set $T_\U(A)^{\alpha_{\mathcal{U}}}$ of $G$-invariant elements of $T_\mathcal{U}(A)$, 
which may a-priori be larger.)

A straightforward computation shows that for $(a_n)_{n\in\N}$
in $\ell^\I(A)$ with corresponding class
$a\in A_\U$, we have
$\| a \|_{2, T_\U(A)} = 0 $ if and only if 
$\lim_{n \to \U} \| a_n \|_{2, u} = 0$.
In particular, this shows that 
\[
J_A = \{ a \in A_\U\colon \| a \|_{2, T_\U(A)} = 0 \}.
\]

The ideals $J_{T(A)}$ and $J_{T(A)^\alpha}$ are not equal in 
general, even if $G$ is amenable. (Take, for example, $A=C(S^1)$ and $G=\Z$ acting on it via irrational 
rotations.) While the tools that we develop in this work are suitable
for studying the quotient $A_\U/J_{T(A)^\alpha}$, the kind of
conclusions that we are interested in refer to the quotient $A_\U/J_A=A^\U$. In general, it is not clear how to transfer information
from one to the other. The assumptions
in our main results concerning the size of the orbits of $\alpha^*$
are used to do this.

\begin{prop} \label{prop:ideal}
Let $A$ be a \ca\ such that $T(A)$ is nonempty, 
let $G$ be a discrete group, and let 
$\alpha\colon G\to\Aut(A)$ be an action. 
Suppose that the cardinality of the 
orbits of $\alpha^*$ is uniformly bounded.
Then $\|\cdot \|_{2, T(A)^\alpha}$ and $\|\cdot \|_{2,u}$ are equivalent, and in particular $J_A = J_{T(A)^\alpha}$.
\end{prop}
\begin{proof}
Given $\tau\in T(A)$, set
$
\tau^\alpha = \frac{1}{|G\cdot \tau|} \sum_{\sigma \in G \cdot \tau} \sigma$.
One readily checks that $\tau^\alpha \in T(A)^\alpha$, and that
$\tau(a) \le |G\cdot \tau| \tau^\alpha(a)$ for all $a \in A_+$.
Let $M>0$ be a uniform bound for the orbits of $\alpha^*$. Given
for $a \in A$, we have
\[
\| a \|_{2, T(A)^\alpha} \le \| a \|_{2,u} \le M^{1/2} \| a \|_{2, T(A)^\alpha}.\qedhere
\]
\end{proof}

Given a \ca\ $B$ and a $G$-action $\gamma \colon G \to \text{Aut}(B)$, we write
$B^\gamma$ for the fixed point algebra of $\gamma$.
The following simple lemma follows from a straightforward reindexation 
argument, which we omit.

\begin{lma}
	\label{lma:reindexation-uniform}
	Let $A$ and $B$ be a separable unital \ca, let $G$ be a countable 
	discrete group, and let 
	$\alpha\colon G\to\Aut(A)$ be an action. Suppose there exists a unital 
	homomorphism $\varphi \colon B \to  (A_\U\cap A')^{\alpha_\U}$. Then for 
	any separable subset $S \subseteq (A_\U)^{\alpha_\U}$ there exists a unital 
	homomorphism $\psi \colon B \to (A_\U\cap S')^{\alpha_\U}$.
\end{lma}

\subsection{$W^*$-ultrapowers.}
We will also need to use tracial 
ultrapowers for von Neumann algebras, so 
we recall this notion as well.
Let $(\M, \tau)$ be a tracial von Neumann algebra, and 
set 
\[c_{\mathcal{U}, \tau}=\{(a_n)_{n\in\N} \in  \ell^\infty( \mathcal{M})\colon 
\lim_{n \to \U} \| a_n \|_{2, \tau} = 0\}.\]
We denote by
$\mathcal{M}^\mathcal{U}$ the quotient
$\mathcal{M}^\mathcal{U} = \ell^\infty( \mathcal{M}) / c_{\mathcal{U}, \tau}$,
and call it the \emph{$W^*$-ultrapower} of $(\M,\tau)$. 

There is unfortunately a notational conflict, since the notation 
$\M^\U$ could mean both the $W^*$-ultrapower of $(\M,\tau)$, or 
the uniform tracial ultrapower of $\M$ regarded as a \ca. 
Both notations are by now well established in the literature, 
and we will always make it clear which one we are referring to. In practice, little confusion should arise since we will never consider
the uniform tracial ultrapower of a von Neumann algebra.

 Denote by
$\pi_\tau$ the GNS representation associated to $\tau$. Most tracial von 
Neumann algebras we deal with in this 
note are of the form
$\pi_\tau(A)''$ for some separable,
simple, unital \ca\ $A$ and some trace $\tau \in T(A)$.The canonical and 
implicit
choice of faithful, normal trace on $\pi_\tau(A)''$ is
(the unique tracial extension of) $\tau$. We will often use the notation $\mathcal{M}_\tau$ to abbreviate
$\pi_\tau(A)''$.

\subsection{Strong outerness}

Let $A$ be a $C^*$-algebra, and let $\tau \in T(A)$ be a trace. Note that if $\theta \in 
\Aut(A)$ and $\tau$ is left invariant by $\theta$, then $\theta$ extends uniquely 
to a trace-preserving automorphism of $\pi_\tau(A)''$,
which we denote by $\theta^\tau$.

\begin{df}
Let $A$ be a simple, unital \ca\ with nonempty trace space, 
and let $\theta\in\Aut(A)$ be an automorphism. 
We say that $\theta$ is \emph{strongly outer} if $\theta^\tau$ is outer
for every $\tau\in T(A)$ satisfying $\tau\circ\theta=\tau$.

An action $\alpha\colon G \to \Aut(A)$ of a discrete group $G$ on $A$
is said to be 
\emph{strongly outer} if $\alpha_g$ is strongly outer
for all $g\in G\setminus\{1\}$.
\end{df}


A strongly outer automorphism is clearly outer. 
The reverse implication is, however,
false. For example, let $A = \bigotimes_{n=1}^{\infty} M_{2^n}$ and let 
$\theta$ be the approximately inner order 2 automorphism of $A$ given as the 
limit of $\Ad (u_k)$, where $u_k = v_1 \otimes v_2 \otimes \ldots \otimes v_k 
\otimes 1_{\bigotimes_{n=k+1}^{\infty} M_{2^n}}$ and $v_j = 
\diag(-1,1,1,1,\ldots,1)$. 
One can 
check that the sequence $u_k$ is Cauchy in the trace norm, and therefore 
$\theta^\tau$ is inner, although $\theta$ is not inner (see \cite{longo}).

\subsection{Rokhlin dimension}

The notion of Rokhlin dimension was first defined for actions of 
$\mathbb{Z}$ and finite groups
in \cite{hwz}, and later extended to residually finite groups in \cite{swz}.
A group $G$ is \emph{residually finite} if for every $g \in G \setminus \{ e 
\}$ there is a normal subgroup $H \le G$ 
of finite index such that $g \notin H$.

\begin{df}[{\cite[Definition 4.4]{swz}}] \label{def:frd}
	Let $G$ be a countable, discrete, residually finite group, let
	$A$ be a separable, unital \ca, and let $\alpha\colon G \to \Aut(A)$ be an 
	action. Given $d \in \N$, we say that $\alpha$
	has \emph{Rokhlin dimension at most $d$}, written $\dimRok(\alpha) \leq d$, 
	if
	for any normal subgroup $H \le G$ of finite index, there
	are positive contractions $f_{\overline g}^{(j)} \in A_\mathcal{U} \cap 
	A'$, 
	for $j = 0, \ldots, d$ and $\overline g \in G/H$,
	such that:
	\begin{enumerate}
		\item $(\alpha_\U)_g(f_{\overline h}^{(j)}) = f^{(j)}_{\overline{gh}}$ 
		for all $j = 0, \ldots, d$ and $g \in G$ and
		$\overline h \in G/H$,
		\item $f^{(j)}_{\overline g} f^{(j)}_{\overline h} = 0$ for all $j = 0, 
		\ldots, d$ and $\overline g,\overline h \in G/H$ with $\overline g \not= 
		\overline h$,
		\item $\sum_{j=0}^d \sum_{\overline g \in G/H} f_{\overline g}^{(j)} 
		=1$.
	\end{enumerate}
\end{df}

There is a related notion, called \emph{Rokhlin dimension with commuting 
	towers}, where the
elements $f_{\overline{g}}^{(j)}$ are assumed to moreover
pairwise commute (see
\cite[Definiton 2.3.b]{hwz} and \cite[Definition 9.2]{swz}). We will not deal 
with this notion here.

We record here an equivalent definition of Rokhlin dimension, which uses 
approximations
instead of ultrapowers.

\begin{prop}[{\cite[Propostion 4.5]{swz}}] \label{rmk:EquivalenceDimRok}
	Using the notation from \autoref{def:frd}, we have $\dimRok(\alpha)\leq 
	d$ if and only if for any normal subgroup $H \leq G$ of finite index, for 
	any finite subset $G_0 \subseteq G$, for every $\ep>0$ and for every finite
	subset $F\subseteq A$, there are positive contractions
	$f^{(j)}_{\overline{g}} \in A$, for $j=0,\ldots,d$ and for $\overline{g}\in 
	G/H$, satisfying:
	\begin{enumerate}[label=(\alph*)]
		\item \label{equiv:a} $\big\| 
		\alpha_{g}(f^{(j)}_{\overline{h}})-f^{(j)}_{\overline{gh}}\big\|<\ep$ 
		for all $j=0,\ldots,d$, for all $g\in G_0$ and for all $\overline{h}\in 
		G/H$,
		\item \label{equiv:b} 
		$\big\|f^{(j)}_{\overline{g}}f^{(j)}_{\overline{h}}\big\|<\ep$ for 
		all $j=0,\ldots,d$ and for all $\overline{g},\overline{h}\in G/H$ with 
		$\overline{g}\neq \overline{h}$,
		\item \label{equiv:c} $\big\|1-\sum_{j=0}^d\sum_{\overline{g}\in G/H} 
		f^{(j)}_{\overline{g}}\big\|<\ep$,
		\item \label{equiv:d} 
		$\big\|af^{(j)}_{\overline{g}}-f^{(j)}_{\overline{g}}a\big\|<\ep$ 
		for all $a\in F$, for all $\overline{g}\in G/H$ and for all 
		$j=0,\ldots,d$.
	\end{enumerate}
\end{prop}

\subsection{The weak tracial Rokhlin property}

We begin by recalling some terminology.

\begin{df}
Let $G$ be a discrete group, and let $\delta>0$. 
Given finite subsets $K,S\subseteq G$, we say that $S$ is 
\emph{$(K,\delta)$-invariant} if
\[
\Big | S \cap \bigcap_{g \in K} gS\Big | \ge (1-\delta) | S |.
\]
\item
	
\end{df}

The
existence of $(K, \delta)$-invariant subsets of $G$ for every finite set $K$ 
and $\delta >0$ is
F\o lner's characterization of amenability.



The next definition goes back to Ocneanu's 
notion of the Rokhlin property for actions of amenable
groups on von Neumann algebras and specifically on the 
hyperfinite II$_1$-factor; see \cite[Chapter 6]{ocneanu}.

\begin{df}\label{df:wtRp}
	Let $G$ be a discrete, amenable group, let $A$ be a simple, separable \uca, 
	and let $\alpha\colon G\to\Aut(A)$ be an
	action. We say that $\alpha$ has the 
	\emph{weak tracial Rokhlin property}
	if for any finite subset $K \subseteq G$ and any $\delta > 0$, there are $n\in\N$, 
	$(K,\delta)$-invariant finite
	subsets $S_1, \dots, S_n \subseteq G$, and positive contractions $f_{\ell,g} 
	\in A_\U \cap A'$ for $\ell = 1,\dots,
	n$ and $g \in S_\ell$, such that
	\begin{enumerate}
		\item \label{wtRp:1} $(\alpha_{\U})_{gh^{-1}}(f_{\ell, h}) = f_{\ell, 
		g}$ for all all $\ell = 1, \dots, n$ and all $g,h \in S_\ell$ ,
		\item \label{wtRp:2} $f_{\ell,g} f_{k,h} = 0$ for all $\ell, k = 1, 
		\dots, n$, $g \in S_\ell$, $k \in S_h$, whenever $(\ell,
		g) \not = (k,h)$,
		\item \label{wtRp:3} $1 - \sum_{\ell = 1}^n \sum_{g \in S_\ell} 
		f_{\ell, g} \in J_A$,
		\item \label{wtRp:4} For $\tau\in T_\U(A)$, for $\ell=1,\ldots,n$ and 
		for $g\in S_\ell$, the value of $\tau(f_{\ell,g})$ is
		independent of $\tau$ and $g$, and is positive.
	\end{enumerate}
\end{df}

\autoref{df:wtRp} is inspired by Wang's \cite[Propostion 2.4]{wang},
except for item \eqref{wtRp:4}, which is inspired by Matui and Sato's 
\cite[Definition
2.5.3]{matuisato:II}. In particular, our definition of the weak
tracial Rokhlin property extends that of Matui-Sato 
to groups that are not necessarily monotileable.

\begin{rem}
Condition \eqref{wtRp:4} in \autoref{df:wtRp}
was used in a predecessor of this paper
(\cite{GarHir_strongly_2018}) to prove that actions with the
weak tracial Rokhlin property have equivariant property (SI) whenever the underlying algebra has property (SI), 
which is needed in the proof of implication
(\ref{rokdim:1}) $\Rightarrow$ (\ref{rokdim:3}) of Theorem \autoref{thmA}. In 
this paper, we rely on the more general results from \cite{equiSI}; 
thus
item \eqref{wtRp:4} above is no longer used to prove the other implications. 
Nevertheless, we
carry out the proof of (\ref{rokdim:1}) $\Rightarrow$ 
(\ref{rokdim:2}) in Theorem \autoref{thmA} so as to obtain Rokhlin towers also
satisfying condition \eqref{wtRp:4},
as we believe that such a stronger condition might prove to be useful in 
future 
applications.
\end{rem}

\section{Equivariant uniform property $\Gamma$} \label{s3}
In the theory of von Neumann algebras, property $\Gamma$ was originally introduced by 
Murray and von Neumann (\cite{mvnIV}) in order prove the existence of non-hyperfinite II$_1$-factors.
A II$_1$-factor $\mathcal{M}$ with trace $\tau$
has property $\Gamma$ if its central sequence algebra $\mathcal{M}^\U
\cap \mathcal{M}'$ is non-trivial. 
Dixmier later showed that in this case property $\Gamma$ is equivalent to the requirement that the II$_1$-factor 
$\mathcal{M}^\U
\cap \mathcal{M}'$ is \emph{diffuse} (\cite{dix}), that is, for every $n \in 
\N$ there are orthogonal projections
$p_1, \dots, p_n \in \mathcal{M}^\U \cap \mathcal{M}'$ such that $\tau_{\mathcal{M}^\U}(p_i) = 1/n$.
This latter formulation of property $\Gamma$ inspired an analogous definition
for uniform tracial ultrapowers, which has been recently introduced in \cite[Definition 2.1]{cpou}
and systematically studied in \cite{uniformGamma}.

In the next two sections we borrow some of the main ideas in \cite{cpou} and \cite{uniformGamma},
and  adapt them to the equivariant
setting. We work with actions of discrete countable groups; for some 
statements, the group will be assumed to be amenable as well.
We start with the definition of uniform property $\Gamma$ for actions of countable
discrete groups on unital separable \cas.

\begin{df}\label{df:EqPropGamma}
Let $G$ be a countable, discrete group, let
$A$ be a unital, separable \ca\ with non-empty trace space, and let $\alpha\colon G\to \Aut(A)$ be an action. We say that 
$(A,\alpha)$ has \emph{uniform property
$\Gamma$} if for every $n \in \N$ and every
$\| \cdot \|_{2, T_\U(A)}$-separable subset 
$S\subseteq A^\mathcal{U}$,
there are projections $p_1, \ldots, p_n \in (A^{\mathcal{U}} \cap
S')^{\alpha^\U}$ with
\begin{enumerate}
\item \label{Gamma:i1} $\sum_{j = 1}^n p_j = 1$,
\item \label{Gamma:i2} $\tau(ap_j) = \frac{1}{n}\tau(a)$ for all $a \in S$, all $\tau \in T_\mathcal{U}(A)$ and all $j=1,\ldots, n$.
\end{enumerate}
\end{df}

When $\alpha$ is the trivial action the definition above coincides with
the definition of uniform property $\Gamma$ as given in \cite[Definition 2.1]{cpou}.

\begin{rem}
In \autoref{df:EqPropGamma}, we could have equivalently required
that the projections $p_1,\ldots,p_n$ belong to 
$(A^\U\cap A')^{\alpha^\U}$ and only require that condition \eqref{Gamma:i2}
holds for all $a\in A$; this follows from separability of the sets
$S$ considered in \autoref{df:EqPropGamma}, along with a standard diagonal argument
(\autoref{lma:reindexation-uniform}).
\end{rem}

In \cite[Proposition 2.3]{cpou}, it is shown that unital separable 
$\mathcal{Z}$-stable \cas\ with non-empty
trace space have uniform property $\Gamma$, since it is possible to embed matrix algebras
of arbitrary dimension into the central sequence algebra of the uniform tracial ultrapower. 
A modification of that argument allow us to show that
equivariantly $\mathcal{Z}$-stable
actions of discrete countable groups on such algebras have uniform property $\Gamma$.


\begin{prop} \label{prop:Zstable}
Let $G$ be a countable, discrete
group and
let $A$ be a separable, unital, $\mathcal{Z}$-stable \ca\
with non-empty trace space. Let $\alpha\colon G
\to \Aut(A)$ be an action and suppose that $(A,\alpha)$ is cocycle conjugate to $(A\otimes\mathcal{Z}, \alpha \otimes 
\text{id}_\mathcal{Z}$).
Then for any $\| \cdot \|_{2, T_\U(A)}$-separable subset $S\subseteq A^\U$ and every $n \in \N$,
there exists a unital $*$-homomoprhism $M_n\to (A^\U \cap S')^{\alpha^\U}$.
In particular $(A,\alpha)$ has uniform property $\Gamma$.
\end{prop}
\begin{proof}
By \autoref{lma:reindexation-uniform}, we can assume that $S = A$.
By Theorem~3.7 in~\cite{szabo:strself}, because $\alpha$ is
cocycle conjugate to $\alpha\otimes\id_{\mathcal{Z}}$, 
there is an equivariant 
embedding 
$\varphi\colon (\mathcal{Z},\id_{\mathcal{Z}})\to (A_\U \cap A',\alpha_\U)$. Note that the image of
$\varphi$ is contained in the fixed point algebra 
$(A_\U\cap A')^{\alpha_\U}$. 
With
$\kappa\colon A_\U \to A^\U$ denoting the canonical equivariant quotient 
map (see \autoref{lma:SurjRelComm}),
it follows by simplicity of $\mathcal{Z}$ that $\kappa\circ\varphi\colon \mathcal{Z}\to
(A^\U\cap A')^{\alpha^\U}$ is a unital embedding. Since
$\mathcal{Z}$ has a unique trace $\tau_\mathcal{Z}$, we have for all 
$z\in\mathcal{Z}$: 
\[\| z \|_{2,\tau_\mathcal{Z}} =
\| \kappa(\varphi(z)) \|_{2,T_\U(A)}  . \] 
It follows that $\kappa\circ\varphi$ is
$(\| \cdot \|_{2,\tau_\mathcal{Z}}$-$\| \cdot \|_{2,T_\U(A)})$-contractive. 
The completion of $C^*$-norm unit ball of 
$\mathcal{Z}$ under the norm
$\| \cdot \|_{2,\tau_\mathcal{Z}}$ is the $C^*$-norm unit
ball of 
$\pi_{\tau_{\mathcal{Z}}}(\mathcal{Z})''\cong \mathcal{R}$.
As the $C^*$-norm unit ball of $A^\U$ is $\| \cdot \|_{2, T_\U(A)}$-complete 
(\cite[Lemma 1.6]{cpou}), it follows that
$\kappa\circ\varphi$ extends to a unital
homomorphism $\mathcal{R}  \cong
\pi_{\tau_\mathcal{Z}}(\mathcal{Z})''\to 
(A^\U \cap A')^{\alpha^\U}$. By restriction, there is also
a unital homomorphism $\rho\colon M_n\to (A^\U \cap A')^{\alpha^\U}$.

Let $e_1,\ldots,e_n\in M_n$ be the canonical
diagonal projections, and set $p_j=\rho(e_j)$ for all 
$j=1,\ldots,n$. Then condition \eqref{Gamma:i1} in \autoref{df:EqPropGamma}
is automatically satisfied. To check \eqref{Gamma:i2}, let $a\in S$, let $j=1,\ldots,n$, and let 
$\tau\in T_\U(A)$. The map $M_n\to \C$ defined by $b\mapsto \tau(\rho(b)a)$ is 
a (not necessarily
normalized) trace on $M_n$, hence a multiple of the canonical
trace $\tau_{M_n}$ on $M_n$. Taking $b=1$ we deduce that the multiple is $\tau(a)$, so that $\tau(\rho(b)a)=\tau_{M_n}(b)\tau(a)$.
Now taking $b=e_j$, we get $\tau(p_ja)=\frac{1}{n}\tau(a)$, as
desired.
\end{proof}

The following proposition is an equivariant 
version of \cite[Lemma 2.4]{cpou}. It roughly states that, in
the presence of uniform property $\Gamma$, positive
contractions can be replaced by
projections when computing tracial values in $A^\U$.

\begin{prop}\label{prop:tracialproj}
Let $G$ be a countable, discrete
group and
let $A$ be a separable, unital \ca\ with non-empty trace space. Let $\alpha\colon G
\to \Aut(A)$ be an action and suppose that $(A,\alpha)$ has  
uniform property $\Gamma$.
Let $S,S_0 \subseteq A^\mathcal{U}$ be $\| \cdot \|_{2,T_\U(A)}$-separable 
subsets and
let $b \in (A^\mathcal{U} \cap S')^{\alpha^\mathcal{U}}$ be a positive contraction. Then there exists a projection
$p \in (A^\mathcal{U} \cap S')^{\alpha^\mathcal{U}}$ such that $\tau(ab) = \tau(ap)$
for all $a \in S_0$ and $\tau \in T_\mathcal{U}(A)$.
\end{prop}
\begin{proof}
We follow closely the proof of Lemma~2.4 in~\cite{cpou}. 
By countable saturation of ultrapowers (see \autoref{saturation}),
it suffices to find, for every $n\in\N$, a 
positive contraction $e\in (A^\U\cap S')^{\alpha^\U}$ 
with $\|e-e^2\|_{2,T_\U(A)}<1/n$, and 
such that
$\tau(ab) = \tau(ae)$
for all $a \in S_0$ and $\tau \in T_\mathcal{U}(A)$.
Fix $n\in\N$ and let $f_1,\ldots,f_n\in C([0,1])$ be given
by the following graph:
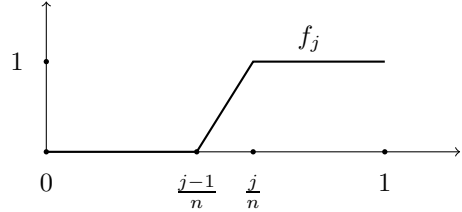
\begin{center}\begin{tikzpicture}
\draw[->] (0,0)--(5.5,0) node[anchor=east]{};
{
	\fill (0,0) circle (1pt) node [below=5pt] {0};
}
{
	\fill (2,0) circle (1pt) node [below=5pt] {$\frac{j-1}{n}$};
}
{
	\fill (2.75,0) circle (1pt) node [below=5pt] {$\frac{j}{n}$};
}
{
	\fill (4.5,0) circle (1pt) node [below=5pt] {$1$};

}
\draw[->] (0,0)--(0,2) node[anchor=north]{};
{
	\fill (0,1.2) circle (1pt) node [left=5pt] {1};
}
\draw[thick] [ black] (0,0)--(2, 0)--(2.75,1.2)--(4.5,1.2);
\draw [black](3.5,1.5) node{$f_{j}$};
\end{tikzpicture}
\end{center}
Set $\widetilde{S}=S\cup \{b\}$. 
Using uniform property $\Gamma$ for $(A,\alpha)$, let 
$p_1,\ldots,p_n$ be projections in $(A^\U\cap \widetilde{S}')^{\alpha^\U}$
such that for all $c\in \widetilde S$ and all $\tau\in T_\U(A)$ we have
\begin{equation}\label{eqn:3.1}
\sum_{j=1}^n p_j=1 \ \ \mbox{ and } \ \ \tau(p_jc)=\frac{1}{n}\tau(c) 
.
\end{equation}
Notice that $t=\frac{1}{n}\sum_{j=1}^n f_j(t)$ for all 
$t\in [0,1]$.
Set $e=\sum_{j=1}^n p_jf_j(b) \in (A^\U\cap S')^{\alpha^\U}$.
Fix $a\in S_0$ and $\tau\in T_\U(A)$. Using 
at the last step that $b=\frac{1}{n}\sum_{j=1}^n f_j(b)$,
we have 
\[\tau(ea)=\sum_{j=1}^n \tau(p_jf_j(b)a)\stackrel{(\ref{eqn:3.1})}{=}\sum_{j=1}^n\frac{1}{n}\tau(f_j(b)a)=\tau(ba).\]
Moreover, using at the second to last step that
$\sum_{j=1}^nf_j -f_j^2\leq 1$, we get
\[\tau(e-e^2)=\sum_{j=1}^n \tau(p_j(f_j(b)-f_j(b)^2))
 \stackrel{(\ref{eqn:3.1})}{=}
 \sum_{j=1}^n \frac{1}{n}\tau(f_j(b)-f_j(b)^2)\leq 
\frac{1}{n}\tau(1)=\frac{1}{n}.
\]
Using that $e-e^2$ is a positive contraction (because this is
the case for $f_j-f_j^2$ and the $p_j$'s are orthogonal), we
conclude that
\[\|e-e^2\|_{2,T_\U(A)}=\sup_{\tau\in T_\U(A)}\tau((e-e^2)^2)
 \leq \sup_{\tau\in T_\U(A)}\tau(e-e^2)\leq \frac{1}{n},
\]
as desired. 
\end{proof}

\section{Equivariant complemented partitions of unity} \label{s4}
In \cite{cpou} uniform property $\Gamma$ is used to infer the existence of  
well-behaved partitions of unity in the central
sequence algebra of uniform tracial ultrapowers,
for nuclear separable \cas. Here we introduce the equivariant version of that definition.


\begin{df}\label{df:CPoU}
Let $A$ be a separable, unital \ca\ with non-empty trace space, let $G$ be a countable, discrete amenable group, and
and let $\alpha\colon G \to \Aut(A)$ be an action.
Given $M>0$, we say that $(A,\alpha)$ has \emph{complemented partitions of
unity (CPoU) with constant $M$},
if for any
$\| \cdot \|_{2, T_\U(A)}$-separable subset 
$S\subseteq A^\mathcal{U}$, for any $n\in\N$, for any $a_1, \ldots, a_n \in 
A_+$ and for any $\delta > 0$ with
\[
\sup_{\tau \in T(A)^\alpha} \min \{\tau(a_1), \ldots, \tau(a_n)\} < \delta,
\]
there exist projections 
$p_1, \ldots, p_n \in (A^{\mathcal{U}} \cap S')^{\alpha^\mathcal{U}}$ such that
\begin{enumerate}
\item \label{cpou.def:1} $\sum_{j=1}^np_j  = 1$,
\item \label{cpou.def:2} $\tau(a_jp_j) \leq M \delta \tau(p_j)$ for every $\tau \in T^\alpha_\U(A)$ and $j=1,\ldots,n$.
\end{enumerate}

We say that $(A, \alpha)$ has \emph{CPoU} if there is a constant
$M > 0$ such that $(A, \alpha)$ has CPoU with constant $M$.
\end{df}

Note that actions as in the above definition always
have invariant traces. 

\begin{rem}
It is not clear to us whether
the above definition is the right one for general 
amenable group actions.
Specifically, one may want to take the supremum over all $T(A)$ in the displayed inequality and require that (2) holds for traces
in $T_\mathcal{U}(A)$. On the other hand,
we will only work with actions 
for which the cardinality of the orbits of the 
induced action on $T(A)$
is bounded, and in this case
the two versions are
equivalent by \autoref{prop:ideal}.
\end{rem}

The intuition behind the name \emph{complemented partition of unity} is the following
(see also the discussion after Definition G in \cite{cpou}): condition
\eqref{cpou.def:1} implies that
the elements $p_1, \dots, p_n$, when identified with the functions $\hat p_j$
on $T_\U(A)$ which send $\tau$ to $\tau(p_j)$,
form a partition of unity, while condition \eqref{cpou.def:2} asserts that $\hat p_j$ is approximately
subordinate to $1 - \hat a_j$, the complement of $\hat a_j$. Indeed when $\tau(1- a_j) =  0$, and
hence $\tau(a_j) = 1$, condition \eqref{cpou.def:2} forces $\tau(p_j)=\tau(a_jp_j)  \le M \delta \tau(p_j)$, and
thus $\tau(p_j)=0$.

The main differences between our \autoref{df:CPoU} and the one in 
\cite[Definition 3.1]{cpou}
are the requirements that $p_1, \dots p_n$ have to be $\alpha^\U$-invariant, the restriction to invariant
traces and the presence of the constant $M$. The motivation for this constant 
is of technical nature\footnote{It is in fact possible to show that 
$(A,\alpha)$ has CPoU with
constant $M$ if and only if it has CPoU with constant 1. This technical 
improvement makes no difference in this paper.};
it will play a role in the proof of \autoref{prop:cpou}, where we 
show that for an action $\alpha$ of an amenable group on a nuclear \ca\ $A$, 
uniform property $\Gamma$ implies the existence of invariant
CPoU. The proof of \autoref{prop:cpou} is not only inspired by some of the results
in \cite[Section 3]{cpou}, but also directly uses
some of them (\cite[Lemma 3.6]{cpou}). In the proof of \autoref{prop:cpou} we 
need
to uniformly bound the images of some elements in $A$ via traces in $T(A)$,
starting from a bound on the images of invariant traces,
which is where the constant
$M$ appears.

\begin{thm} \label{prop:cpou}
Let $G$ be a countable, discrete, amenable group, let $A$ be a separable, nuclear,
unital \ca\ with non-empty trace space and let $\alpha\colon G
\to \Aut(A)$ be an action. Suppose that 
the induced action on $T(A)$ has finite orbits bounded 
in size by some uniform constant $M>0$, and that
$(A,\alpha)$ has uniform property $\Gamma$. 
Then $(A,\alpha)$ has CPoU
with constant $M$.
\end{thm}
\begin{proof}
We fix a $\| \cdot \|_{2, T_\U(A)}$-separable subset 
$S\subseteq A^\mathcal{U}$, positive elements
$a_1, \ldots, a_n \in A$ and $\delta > 0$ with
\[\sup_{\tau \in T(A)^\alpha} \min \{\tau(a_1), \ldots, \tau(a_n)\} < \delta.\] 
 
We divide the proof into two claims. 
The first one does not require
the use of uniform property $\Gamma$; it is an equivariant 
version of Lemma~3.6 in~\cite{cpou}.

\begin{claim} \label{claim:cpou1}
Suppose there exist $t>0$ 
and 
a projection $q \in (A^{\mathcal{U}} \cap A')^{\alpha^{\mathcal{U}}}$  
such that $\tau(q) = t$ for all $\tau \in T_\mathcal{U}(A)$. 
Then there are positive contractions $b_1, \ldots, b_n \in (A^{\mathcal{U}}
\cap S')^{\alpha^\U}$ such that 
\begin{enumerate}[label=(1.\alph*)]
\item \label{cpou:1a} $\sum_{j=1}^n\tau( b_j q) = t$ for all $\tau \in T_\mathcal{U}(A)$,
\item \label{cpou:1b} $\tau(a_jb_jq) \le M\delta \tau(b_jq)$ for all $\tau \in 
T^\alpha_\U(A)$ and for $j=1,\ldots,n$.
\end{enumerate}
\end{claim}

Since $G$ is countable, by replacing $S$ with
$\bigcup_{g\in G}\alpha^\U_g(S)$ we may assume that
$S$ is $\alpha^\U$-invariant. 
Fix $\ep>0$ and a finite subset $F\subseteq G$. 
By saturation of $A^\mathcal{U}$ (\autoref{saturation}), 
it suffices to find positive contractions 
$b_1, \ldots, b_n \in A^{\mathcal{U}}
\cap S'$ satisfying \ref{cpou:1a} and \ref{cpou:1b} and
\[
\max_{g \in F} \max_{j=1,\ldots,n} \| \alpha^\U_g(b_j) - b_j \|_{2,T_\U(A)} < 
\ep
.
\]
Use amenability of $G$ to find a finite subset 
$K \subseteq G$ such that 
\[
\frac{| g K \triangle K |}{|K|} < \ep 
\]
for all $g\in F$. For each $j=1,\ldots,n$, set
$a'_j= \frac{1}{|K|} \sum_{k \in K} \alpha_{k^{-1}}(a_j)$,
which is a positive element in $A$.
For $\tau \in T(A)^\alpha$, we have 
$\tau(a_j') = \tau(a_j)$, and in particular 
\[
\sup_{\tau \in T(A)^\alpha} \min \{\tau(a'_1),\ldots,\tau(a'_n)\} < \delta.
\]
Because $\tau\leq M\tau^\alpha$ (see \autoref{ss2.1} where
the notation $\tau^\alpha$ is introduced), we deduce that
\[
\sup_{\tau \in T(A)} \min \{\tau(a'_1),\ldots,\tau(a'_n)\} < M\delta.
\]
Apply \cite[Lemma 3.6]{cpou} to $a'_1,\ldots,a'_n$ to 
find positive contractions 
$b'_1, \ldots, b'_n \in A^{\mathcal{U}} \cap S'$
satisfying conditions \ref{cpou:1a} and \ref{cpou:1b} in this claim.
For $j=1,\ldots,n$, set 
\[b_j= \frac{1}{|K|} \sum_{k \in K} \alpha^\mathcal{U}_k(b'_j)\in A^\U.\]
Since $S$ is $\alpha^\mathcal{U}$-invariant and since
$b'_1, \ldots, b'_n$ commute with $S$, 
it follows that $b_1,\ldots, b_n$ also commute with $S$. 
Moreover, a routine computation shows that for all $g\in F$ and for 
$j=1,\ldots,n$ we have
\[\| \alpha^\U_g(b_j) - b_j \|_{2,T_\U(A)}  \le \frac{| g K \triangle K |}{|K|} 
< \ep . \]
On the other hand, for every $\tau \in T_\mathcal{U}(A)$ and every
$j=1,\ldots,n$, we have
\begin{align*}
\tau(b_j q) &= \frac{1}{|K|}  \sum_{k \in K}\tau(\alpha^\U_k(b'_j)q) 
= \frac{1}{|K|} \sum_{k \in K}    \tau(\alpha^\U_k(b'_j q)).
\end{align*}
Thus, by \ref{cpou:1a} we have $\sum_{j=1}^n\tau(b_jq)=t$. 

Let $\tau\in T^\alpha_\U(A)$ and $j=1,\ldots,n$. 
In the next computation, we use the fact that $q$ is
$G$-invariant and that $\tau=\tau\circ\alpha^\U_{k^{-1}}$ 
at the second step, and the above displayed equation at the 
last step in combination with the fact that $\tau$ is 
$G$-invariant, to get
\begin{align*}
\tau(a_jb_jq) 
&=\frac{1}{|K|} \sum_{k \in K} \tau(a_j \alpha^\U_k(b'_j)q) \\
&= \frac{1}{|K|} \sum_{k \in K} \tau (\alpha_{k^{-1}}(a_j)b'_jq) \\
&=  \tau\Big( \frac{1}{|K|} \sum_{k \in K} \alpha_{k^{-1}}(a_j)b'_jq\Big) \\
&= \tau(a_j' b'_jq) \stackrel{\ref{cpou:1b}}{\leq} M\delta \tau(b'_jq) =M\delta \tau(b_jq).
\end{align*}
This proves the claim.
\vspace{0.2cm}

By saturation of ultrapowers (see \autoref{saturation}), 
the set $I$ of all numbers $t \in [0,1]$ for 
which there are orthogonal projections
$\widetilde{p}_1, \ldots, \widetilde{p}_n \in (A^{\mathcal{U}} \cap S')^{\alpha^\U}$
such that
\begin{enumerate}[label=(\roman*)]
\item \label{cpou.proof:i} $\tau(\sum_{j=1}^n \widetilde{p}_j) = t$ for all $\tau \in T_\mathcal{U}(A)$,
\item \label{cpou.proof:ii} $\tau(a_j\widetilde{p}_j) \le M\delta \tau(\widetilde{p}_j)$ for all $\tau \in T_\U^\alpha(A)$ and all $j=1,\ldots,n$,
\end{enumerate}
is closed, and it is clearly non-empty as it contains zero. 
Let $t_0$ be the maximal element in this set.
Then 
$(A,\alpha)$ has CPoU if and only if $t_0=1$.
Set $s_0=t_0 +\frac{1 - t_0}{n}$. Then $t_0\leq s_0\leq 1$
and $s_0=t_0$ if and only if $t_0=1$. 

\begin{claim}
 $s_0$ belongs to $I$ (and thus $s_0=t_0=1$).
 \end{claim}
Let $\widetilde{p}_1,\ldots,\widetilde{p}_n\in (A^\U\cap S')^{\alpha^\U}$ be projections satisfying \ref{cpou.proof:i} and \ref{cpou.proof:ii} above
for $t_0$. Set $q=1 - \sum_{j=1}^n \widetilde{p}_j$, and note 
that $\tau(q)=1-t_0$ for all $\tau\in T_\U(A)$.
Apply 
Claim \ref{claim:cpou1} to $\widetilde{S}=S \cup \{q\}$ in place of $S$
to obtain $b_1, \ldots, b_n
\in (A^{\mathcal{U}} \cap \widetilde{S}')^{\alpha^\U}$ 
satisfying conditions \ref{cpou:1a} and \ref{cpou:1b} for $1-t_0$ in place of $t$. 
Use \autoref{prop:tracialproj} to find projections $p'_1, \ldots, p'_n \in (A^{\mathcal{U}}
\cap \widetilde{S}')^{\alpha^\U}$ such that 
\begin{equation} \label{eqn:4.1}
\tau(qb_j) = \tau(qp'_j) \ \ \mbox{ and } \ \ \tau(a_jqb_j) = \tau(a_jqp'_j)
\end{equation}
for all $j=1,\ldots,n$ and all $\tau \in T_\mathcal{U}(A)$.
Set 
\[T= A \cup S \cup \{q, \widetilde{p}_1, \ldots, \widetilde{p}_n, p'_1,\ldots, p'_n\}.\]
Use uniform property $\Gamma$ for $(A,\alpha)$ to find 
orthogonal projections $e_1, \ldots, e_n \in
(A^{\mathcal{U}} \cap T')^{\alpha^\U}$ which add up to 1 
and satisfy 
\begin{equation}\label{eqn:4.2}
\tau(e_jy) = \frac{1}{n}\tau(y)
\end{equation} for all 
$\tau \in T_\mathcal{U}(A)$, for all $y \in T$ and for all $j=1,\ldots,n$.
For $j=1,\ldots, n$, set
\[
p_j= \widetilde{p}_j + qp'_je_j \in (A^{\mathcal{U}} \cap S')^{\alpha^\U}.
\]
Since $q\perp \widetilde{p}_j$, it follows that $p_j$ is a projection. Moreover, $p_j\perp p_k$ if $j\neq k$.
For $\tau \in T_\mathcal{U}(A)$ we have
\begin{eqnarray*}
\tau\Big(\sum_{j=1}^n p_j\Big) &\stackrel{\mathclap{(\ref{eqn:4.2})}}{=}& \tau\Big(\sum_{j=1}^n \widetilde{p}_j\Big) + \frac{1}{n}\tau\Big(\sum_{j=1}^n qp'_j\Big) \\
&\stackrel{\mathclap{\ref{cpou.proof:i},(\ref{eqn:4.1})}}{=}& t_0 + 
\frac{1}{n}\sum_{j=1}^n \tau(qb_j) \\
&\stackrel{\mathclap{\ref{cpou:1a}}}{=}& t_0 +
\frac{1 - t_0}{n}=s_0.
\end{eqnarray*}
In addition, given $\tau \in T_\U^\alpha(A)$, for $j=1,\ldots,n$ we have
\begin{eqnarray*}
\tau(a_jp_j) &=& \tau(a_j \widetilde{p}_j) + \tau(a_jqp'_je_j) \\
&\stackrel{\mathclap{(\ref{eqn:4.2})}}{=}&
\tau(a_j \widetilde{p}_j) + \frac{1}{n}\tau(a_jqp'_j) \\
&\stackrel{\mathclap{\ref{cpou.proof:ii}, \ref{cpou:1b}}}{\leq}& M\delta\tau(\widetilde{p}_j) + \frac{M\delta}{n}\tau(qp'_j) \\
&=& M\delta \tau(\widetilde{p}_j + qp'_je_j) = M\delta \tau(p_j).
\end{eqnarray*}
It follows that $p_1,\ldots,p_n$ witness the fact that $s_0$ belongs to $I$, 
as desired. This proves the claim and the theorem.
\end{proof}

Invariant CPoUs are the main technical ingredient for the `local to global'
arguments employed to prove the implications \eqref{B:1} 
$\Rightarrow$ \eqref{B:2} and \eqref{B:1} $\Rightarrow$ \eqref{B:3} of Theorem 
\autoref{thmB}.
This allows us
to avoid the machinery involving $W^*$-bundles developed in \cite{liao,liaoZm, 
GarHir_strongly_2018} (and in the non-dynamical setting in \cite{bbstww}),
which required the assumption that $T(A)$ is a Bauer simplex.
\autoref{lma:LocalToGlobal} asserts that, in the presence of CPoU,
if a polynomial identity is (approximately) satisfied 
in each \emph{individual} tracial completion of $(A,\alpha)$, then the
identity is (exactly) satisfied in $(A^\U,\alpha^\U)$.
To explicitly describe how this transfer works, we begin by establishing some
terminology. 



\begin{df}
Let $G$ be a discrete group.
Given a tuple of non-commuting variables $\bar x = (x_1, \ldots, x_r)$ and 
$g \in G$, set $g \cdot \bar x = (g \cdot x_1, \ldots, g \cdot x_r)$, which we 
also regard as a tuple of non-commuting variables.
By a \emph{$G$-$*$-polynomial} in the variables $\bar x$ we mean 
a $*$-polynomial in the variables
$\{g \cdot \bar x\colon g \in G\}$.
Let $A$ be a \ca\ and let $\alpha\colon G \to \Aut(A)$ be an action.
Given a tuple $\bar x = (x_1, \ldots, x_r)$, given a
$G$-$*$-polynomial $Q(\bar x)$, and given a coefficient tuple $\bar a = (a_1, \ldots, a_r) \in A^r$, the term $Q(\bar a)$
is computed by interpreting each $g \cdot x_j$ as $\alpha_g(a_j)$ 
for $j=1,\ldots,r$. 
\end{df}


\begin{lma}\label{lma:LocalToGlobal}
Let $A$ be a separable, unital \ca\ with non-empty trace space, let
$G$ be a countable, discrete group, and 
let $\alpha\colon G \to \Aut(A)$ be an action such that the induced action on $T(A)$ 
has orbits which are uniformly bounded in size.
Assume further that $(A,\alpha)$ has CPoU. 
For $m\in\N$, let $Q_m$
be a $G$-$*$-polynomial in $r_m +  s_m$ non-commuting variables, and
let $(a_j)_{j \in \N}$ be a sequence in
$A$. Suppose that
for every $\ep > 0$, for every $n \in \N$ and for every $\tau \in T(A)^\alpha$,
there are contractions $w^\tau_j\in \mathcal{M}_\tau$, for $j\in \N$, 
such that, for $m=1,\ldots,n$ we have
\[
\| Q_m(\pi_\tau(a_1), \ldots, \pi_\tau(a_{r_m}),
w_1^\tau, \ldots , w_{s_m}^\tau) \|_{2,\tau} < \ep.
\]
Then there are contractions $w_j\in A^\mathcal{U}$, for $j\in\N$, such that
\[
 Q_m(a_1, \ldots, a_{r_m},
w_1, \ldots , w_{s_m})
\in J_{T(A)^\alpha},
\]
for all $m \in \N$.
\end{lma}
\begin{proof}
Fix $M > 0$ such that $(A, \alpha)$ has CPoU with constant $M$.
Let $\ep>0$ and let $\ell\in\N$.
By saturation of $(A^\mathcal{U},\alpha^\U)$ (see \autoref{saturation}), it is 
sufficient to find contractions $w_j\in A^\U$, for $j\in\N$, such that for all 
$m \le \ell$ we have
\[
\sup_{\tau\in T_\U^\alpha(A)}\| Q_m(a_1, \ldots, a_{r_m},
w_1, \ldots , w_{s_m}) \|_{2,\tau} < \ep
. \]
By assumption, for each $\tau \in T(A)^\alpha$ there are contractions 
$\widetilde{w}^\tau_j\in \mathcal{M}_{\tau}$, 
for $j\in\N$, such that for $m=1,\ldots,\ell$ we have
\[
\| Q_m(\pi_\tau(a_1), \ldots, \pi_\tau(a_{r_m}),
\widetilde{w}_1^\tau, \ldots , \widetilde{w}_{s_m}^\tau) \|^2_{2,\tau} <
\frac{\ep^2}{M\ell} .
\]
By Kaplansky's density theorem, 
we can choose contractions $w_j^\tau\in A$
with $\widetilde{w}^\tau_j = \pi_\tau(w_j^\tau)$ for all $j\in\N$.
For each $\tau \in T(A)^\alpha$, set
\begin{equation} \label{eqn:4.3}
b_\tau = \sum_{m = 1}^\ell \Big| Q_m(a_1, \ldots, a_{r_m},
w_1^\tau, \ldots , w_{s_m}^\tau)\Big|^2.
\end{equation}
Then
\[
\tau(b_\tau) = \sum_{m =1}^\ell \| Q_m(a_1, \ldots, a_{r_m},
w_1^\tau, \ldots , w_{s_m}^\tau)\|_{2,\tau}^2 < \frac{\ep^2}{M}.
\]
Use compactness of $T(A)^\alpha$ to find $\tau_1, \ldots, \tau_n \in T(A)^\alpha$ such that
\[
\sup_{\tau \in T(A)^\alpha} \min \{\tau(b_{\tau_1}), \ldots, \tau(b_{\tau_n} )\} < \frac{\ep^2}{M}.
\]
Set $S = \{ w^{\tau_1}_j, \ldots,  w^{\tau_n}_j, a_j\colon  j \in \N\}$.
Since $(A,\alpha)$ has CPoU with constant $M$, there
are projections $p_1,\ldots, p_n \in (A^\mathcal{U} \cap S')^{
\alpha^\mathcal{U}}$ adding up to 1 such that
\begin{equation}\label{eqn:4.4}
\tau(b_{\tau_j} p_j) \leq \ep^2 \tau(p_j)\end{equation} 
for all $\tau \in T_\U^\alpha(A)$ and $j=1,\ldots,n$.

For $k \in \N$, set $w_k = \sum_{j = 1}^n p_j w^{\tau_j}_k$. Those are
contractions in $A^\U$. 
Fix $j=1,\ldots,n$.
Because $p_1, \ldots, p_n$
are $\alpha^\mathcal{U}$-invariant projections 
which commute with all elements in $S$, we have
\begin{eqnarray*}
\Big| Q_m(a_1, \ldots, a_{r_m},
w_1, \ldots , w_{s_m})\Big|^2 &=& \sum_{j=1}^n p_j
\Big| Q_m(a_1,\ldots,a_{r_m},w_{s_1}^{\tau_j},\ldots,w_{s_m}^{\tau_j})\Big|^2 \\
&\stackrel{\mathclap{(\ref{eqn:4.3})}}{\le}&  \sum_{j=1}^n p_j b_{\tau_j}.
\end{eqnarray*}

As a consequence, given $\tau\in T_\U^\alpha(A)$ we have
\begin{eqnarray*}
 \| Q_m(a_1, \ldots, a_{r_m},
w_1, \ldots , w_{s_m}) \|^2_{2,\tau}
&=& \tau\Big(\Big|Q_m(a_1, \ldots, a_{r_m},
w_1, \ldots , w_{s_m})\Big|^2\Big)\\
&\leq& 
\sum_{j=1}^n \tau\Big(p_j {\Big| Q_m(a_1,\ldots,a_{r_m},w_{s_1}^{\tau_j},\ldots,w_{s_m}^{\tau_j})\Big|^2}\Big)\\
&\le& \sum_{j = 1}^n \tau(p_j b_{\tau_j}) \\ &\stackrel{\mathclap{(\ref{eqn:4.4})}}{\leq}&
 \sum_{j = 1}^n \varepsilon^2 \tau(p_j ) = \varepsilon^2.
\end{eqnarray*}
This concludes the proof.

\end{proof}

\begin{rem} \label{vNaultra}
By countable saturation of ultrapowers (see \autoref{saturation}), the assumptions of \autoref{lma:LocalToGlobal} are satisfied if 
(and only if) for every $\tau \in T(A)^\alpha$
there exist contractions 
$w^\tau_j$ in the von Neumann ultrapower $\mathcal{M}_\tau^\mathcal{U}$, for $j \in \N$, such that 
\[
Q_m(\pi_\tau(a_1),\ldots,\pi_\tau(a_{r_m}), w^\tau_1,\ldots,w^\tau_{s_m}) = 0,
\]
for all $m \in \N$. This amounts to saying that the polynomial relations one wishes
to realize in $(A^\U,\alpha^\U)$ are \emph{exactly} realized in the tracial ultrapower of every
GNS closure.
\end{rem}

As pointed out in \cite[Remark 4.2.iii]{cpou}, there is a notable difference when employing CPoU in `local to global'
arguments over trace spaces, as opposed to older techniques relying on $W^*$-bundles and 
on the assumption that $T(A)$ is a Bauer simplex. Indeed, 
when $\partial_e T(A)$ is compact,
it is enough to consider extreme traces in order to obtain an analogue of
\autoref{lma:LocalToGlobal} (see \cite[Lemma 3.18]{bbstww}). Concretely, this 
allows one to work
exclusively with von Neumann algebras of the form $\pi_\tau(A)''$ for $\tau \in 
\partial_e T(A)$; when $A$ is nuclear, those are always isomorphic to the 
hyperfinite II$_1$-factor $\mathcal{R}$.
The same applies in the dynamical setting if one furthermore assumes 
$T(A)^\alpha = T(A)$; see
\cite{liao}, \cite{liaoZm} and \cite{sato19}. This is not the case in our framework,
since we want to remove the requirement that $\partial_e T(A)$ is compact (except for \autoref{s7})
and we work in a situation where in general $T(A)^\alpha \not = T(A)$; we may 
even have
$T(A)^\alpha \cap \partial_e T(A) = \emptyset$. As a consequence, we are
forced to consider all traces in order to apply \autoref{lma:LocalToGlobal}, and thus find approximate solutions
in general (tracial) GNS representations, not just factorial ones.
The next section provides a concrete example of this approach.

\section{Tracial ultrapowers of actions of amenable groups} \label{s5}
The main objective of the current section is the following result,
which plays a key role in the proofs in  both \autoref{s6} and \autoref{s7}.
Given a group $G$ and
a normal subgroup $N \le G$, throughout the rest of the paper we let 
$q_N\colon 
G \to N
$ 
denote the quotient
map.

The goal of the present section is to prove the following result.

\begin{thm} \label{cor:EmbedModelAction}
Let $G$ be a countable, discrete, amenable group,
let $A$ be a separable, simple, unital, stably finite, nuclear \ca,
and let $\alpha\colon G \to \Aut(A)$ be an 
action
such that the orbits of the action
induced by $\alpha$ on $T(A)$ are finite and that their cardinality is  
bounded, 
and assume that $(A,\alpha)$ has CPoU. 
Let $N$ be a normal subgroup of $G$ such that 
$\alpha_g$ is strongly outer for all $g\in G\setminus N$, and let $\mu_{G/N} \colon G/N \to \text{Aut}(
\mathcal{R})$ be an outer action on the hyperfinite II$_1$ factor $\mathcal{R}$.
Then there exists an equivariant, unital embedding
\[
(\mathcal{R}, \mu_{G/N}\circ q_N) \to \Big(A^\mathcal{U} \cap A', \alpha^\mathcal{U} \Big).
\]
\end{thm}

The strategy for proving \autoref{cor:EmbedModelAction} is as follows. First, 
we construct 
equivariant, unital embeddings of
$(\mathcal{R}, \mu_{G/N}\circ q_N)$ in the central sequence algebra
of the weak closure of each 
individual invariant trace of $A$. After that, we apply 
\autoref{lma:LocalToGlobal} to glue
those embeddings using CPoU, thus obtaining an equivariant, unital embedding $(\mathcal{R}, \mu_{G/N}\circ q_N) \to
(A^\mathcal{U} \cap A', \alpha^\mathcal{U} )$, thanks to \autoref{prop:ideal}.
The first part translates
into proving the existence of equivariant embeddings of $(\mathcal{R}, \mu_{G/N}\circ q_N)$
into the central sequence algebras of
hyperfinite, not-necessarily factorial 
type II$_1$ von Neumann algebras 
with respect to suitable outer actions; see \autoref{thm:ErgThy}.

We begin with some preliminaries.

\begin{df}\label{df:CocConj}
Let $A$ and $B$ be unital \cas, let $G$ be a discrete group,
and let $\alpha\colon G\to\Aut(A)$ and $\beta\colon G\to\Aut(B)$ be actions.
\be
\item \label{cc:1} We say that $(A,\alpha)$ and $(B,\beta)$ are \emph{conjugate}
if there exists an 
isomorphism $\varphi\colon  A\to B$ satisfying $\varphi\circ\alpha_g=\beta_g\circ\varphi$ for all $g\in G$. 
In this case, we say that $\varphi\colon (A,\alpha)\to(B, \beta)$
is an \emph{equivariant isomorphism}.
\item An $\alpha$-\emph{cocycle} is a function $u\colon G\to\U(A)$ satisfying
$u_{gh}=u_g\alpha_g(u_h)$ for all $g,h\in G$. In this case, we define the \emph{cocycle
perturbation} $\alpha^u$ of $\alpha$ to be the action given by 
$\alpha^u_g=\Ad(u_g)\circ\alpha_g$ for all $g\in G$. (The cocycle condition 
guarantees that this is indeed an action.)
\item We say that $(A,\alpha)$ and $(B,\beta)$ are \emph{cocycle conjugate}, written $(A,\alpha)\cong_{\mathrm{cc}}(B,\beta)$,
if there is an $\alpha$-cocycle $u$ such that $(A,\alpha^u)$
and $(B,\beta)$ are conjugate.
\ee

We say that $(A,\alpha)$ \emph{(tensorially) absorbs} $(B,\beta)$
if $(A,\alpha)\cong_{\mathrm{cc}}(A\otimes_{\mathrm{min}}B,\alpha\otimes \beta)$.

Similar notions apply to actions on von Neumann algebras, where tensorial absorption
is considered with respect of the von Neumann tensor product $\overline \otimes$.
\end{df} 

\begin{prop}\label{thm:AbsR}
Let $(\mathcal{M},\tau)$ be a separably representable, type II$_1$ von Neumann algebra, let $G$ be a countable, discrete, amenable group,
let $\gamma\colon G\to\Aut(\mathcal{M}, \tau)$ be an action and let
$\delta\colon G \to \text{Aut}(\mathcal{R})$ be an action
such that $(\mathcal{R}, \delta)$ is cocycle conjugate to $(\overline{\bigotimes}_{n\in \N} \mathcal{R}, {\bigotimes}_{n\in \N} \delta)$. Suppose
that $(\mathcal{M},\gamma)$ absorbs $(\mathcal{R},\delta)$. Then there is a unital equivariant homomorphism $(\mathcal{R},\delta)\to (\mathcal{M}^\U\cap \mathcal{M}',\gamma^\U)$
\end{prop}
\begin{proof}
Observe that if $(\mathcal{N}_0,\gamma_0)$ is cocycle conjugate to $(\mathcal{N}_1,\gamma_1)$, then
$(\mathcal{N}_0^\U\cap\mathcal{N}_0',\gamma_0^\U)$ is \emph{conjugate} to $(\mathcal{N}_1^\U\cap\mathcal{N}_1',\gamma_1^\U)$, since the inner automorphisms induced by a $\gamma_0$-cocycle act trivially
on the central sequence algebra $\mathcal{N}_0^\U\cap\mathcal{N}_0'$. 
It is therefore enough to show that there is a unital map 
\[(\mathcal{R}, \delta)\to 
\big((\mathcal{M}\overline{\otimes} \mathcal{R})^\U \cap (\mathcal{M} \overline{\otimes} \mathcal{R})',(\gamma{\otimes}\delta)^\U\big).\]
In turn, it suffices to find a unital map $(\mathcal{R}, \delta)\to (\mathcal{R}^\U \cap \mathcal{R}', \delta^\U)$,
 since the latter is unitally contained in $\big((\mathcal{M}\overline{\otimes} \mathcal{R})^\U \cap (\mathcal{M} \overline{\otimes} \mathcal{R})',(\gamma{\otimes}\delta)^\U\big)$.
We use the notation $(\overline{\mathcal{R}}, \bar{\delta})$ to abbreviate
$(\overline{\bigotimes}_{n\in \N} \mathcal{R}, {\bigotimes}_{n\in \N} \delta)$.
By assumption $(\mathcal{R}, \delta) \cong_{\mathrm{cc}} (\overline{\mathcal{R}}, \bar{\delta})$, and hence
$(\mathcal{R}^\U \cap \mathcal{R}', \delta^\U)$ is conjugate to
$(\overline{\mathcal{R}}^\U \cap \overline{\mathcal{R}}', \bar{\delta}^\U)$;
thus, it suffices to find a unital map $(\mathcal{R},\delta)\to 
\big(\overline{\mathcal{R}}^\U \cap \overline{\mathcal{R}}', 
\bar{\delta}^\U\big)$.
Let $\varphi_n\colon(\mathcal{R}, \delta) \to (\overline{\mathcal{R}}, \bar{\delta})$ be the equivariant unital inclusion in the 
$n$-th tensor factor. Then 
$(\varphi_n)_{n \in \N} \colon\mathcal{R} \to \overline{\mathcal{R}}^\U$ is the desired unital map.
\end{proof}

We recall that an automorphism $\alpha$ of a von Neumann algebra $\M$
is said to be \emph{properly outer} if for every
$\alpha$-invariant central non-zero projection $p$ in $\M$,
the restriction of $\alpha$ to $p \M p$ is outer.
Notice that if $\M$ is a factor, an automorphism is properly outer if and only if it is outer.
An action $\gamma\colon G \to \text{Aut}(\M)$ is said to be
\emph{properly outer} if $\gamma_g$ is properly outer for every $g \in G \setminus \{ e \}$.

When $(\M,\tau)$ is a II$_1$-factor, the following is a well-known
result of Ocneanu \cite{ocneanu}. The version we give here
for hyperfinite II$_1$ von Neumann algebras
follows from the classification of actions of discrete amenable groups on semifinite, hyperfinite
von Neumann algebras in \cite{suthtak} (see also \cite[Section 3]{suth}).

\begin{thm} \label{thm:ErgThy}
Let $(\mathcal{M},\tau)$ be a separably representable,
hyperfinite, type II$_1$ von Neumann algebra and let $G$ be a countable, discrete, amenable group.
Fix an action
$\gamma\colon G \to \Aut(\mathcal{M},\tau)$ which preserves $\tau$. Let $N$ be a normal subgroup of $G$ such that
$\gamma_g$ is properly outer for all $g\in G\setminus N$ and let $\mu_{G/N}\colon G/N \to \text{Aut}(\mathcal{R})$ be an outer action. 
Then 
\[(\mathcal{M}, \gamma)\cong_{\mathrm{cc}}(\mathcal{M} \overline{\otimes} \mathcal{R}, \gamma \otimes (\mu_{G/N}\circ q_N)).\]
\end{thm}
\begin{proof}
Denote by $\mathcal{C}$ the center of $\mathcal{M}$, and let $(X,\nu)$ be
its von Neumann spectrum, so that $(\mathcal{C},\tau|_{\mathcal{C}})\cong (L^\I(X,\nu),\int_Xd\nu)$
as tracial von Neumann algebras. 
Without loss of generality, we assume
that $(X,\nu)$ is a standard Borel probability space.
By \cite[Theorem XVI.1.5]{take3}, there is an isomorphism 
$\mathcal{M} \cong \mathcal{C} \overline{\otimes} \mathcal{R}$. Therefore,
 by \cite[Corollary IV.8.30]{take1}, we can identify every $b \in \mathcal{M}$ 
 with a decomposable
element $\int^\oplus_X b_x d \nu$, where each $b_x \in \mathcal{R}$.

Given an arbitrary action $\alpha \colon G \to \text{Aut}(\mathcal{M}, \tau)$, which will later be taken to be
either $\gamma$ or $\gamma\otimes \mu_{G/N}\circ q_N$, we 
let $\sigma^\alpha$ the measurable
$G$-action on $(X, \nu)$ induced by $\alpha$. We recall the notion of the 
\emph{ancillary groupoid}
and \emph{ancillary action} associated to $\alpha$, as well as the cocycle conjugate
invariants introduced in \cite{suthtak}. We refer to \cite{suthtak} for details
(see also \cite[Section 3]{suth} and \cite{jt84}). There 
exists a measurable map $\bar \alpha\colon 
X \times G \to \text{Aut}(\mathcal{R})$ such that, given $b =\int^\oplus_X b_x d \nu \in \mathcal{M}$,
\begin{equation} \label{eq:defgroup}
\bar \alpha_{x,g}\big(b_{\sigma^\alpha_g(x)}\big)=\alpha_g(b)_x,
\end{equation}
for every $(x,g) \in X \times G$. Denote by $\mathcal{G}_\alpha:= X \times G$ the measured transformation
groupoid
(with unit space $X$) corresponding to the $G$-action induced by $\alpha$ on $X$. Let
\[
\mathcal{H}_\alpha := \{(x,g) \in \mathcal{G}_\alpha : \sigma_g(x) = x \},
\ \ 
\mbox{ and } 
\ \ 
\mathcal{N}_\alpha := \{(x, g) \in \mathcal{H}_\alpha : \bar \alpha_{x,g} \text{ is inner} \}.
\]
Note that $\mathcal{N}_\alpha$ is an invariant of cocycle
conjugacy.
Choose a Borel function $U \colon \mathcal{N}_\alpha \to \mathcal{U}(\mathcal{R})$ such that
$\bar \alpha_n = \text{Ad}(U(n))$ for every $n \in \mathcal{N}_\alpha$. The ways in which the unitaries
in the image of $U$ interact with each other and with the rest of the action
is described via a relative cohomology class
$\chi_\alpha$, which is also a cocycle conjugacy invariant of $\alpha$ (see \cite[Section 2]{suth}
for the precise definition).

We identify $\mathcal{M}$ and $\mathcal{M}\overline{\otimes}\mathcal{R}$ throughout.
Note that $\gamma$ and $\gamma \otimes (\mu_{G/N}\circ q_N)$ induce the same action on the center
of $\mathcal{M}$, which we will denote simply by $\sigma$. 
It follows that the groupoids $\mathcal{G}_\gamma$ and $\mathcal{G}_{\gamma
\otimes (\mu_{G/N}\circ q_N)}$ are isomorphic,
which in turn implies that $\mathcal{H}_\gamma \cong \mathcal{H}_{\gamma \otimes
(\mu_{G/N}\circ q_N)}$.
It follows then by \cite[Theorem p. 324]{suth} that
the two actions are cocycle conjugate if and only if
$\mathcal{N}_\gamma \cong \mathcal{N}_{\gamma \otimes
(\mu_{G/N}\circ q_N)}$ and $\chi_\gamma = \chi_{\gamma \otimes (\mu_{G/N}\circ 
q_N)}$.
(The invariant $\delta$ does not play a role when $\mathcal{M}$ is finite.)

\begin{claim} \label{claim:baraction}
Let $g\in G$ and $x\in X$ satisfy $\sigma_g(x)=x$. (In other words, $(x,g) \in \mathcal{H}_\gamma = \mathcal{H}_{\gamma \otimes
(\mu_{G/N}\circ q_N)}$.) Then
\[
(\overline{\gamma \otimes (\mu_{G/N}\circ q_N)})_{x,g} =
\bar \gamma_{x,g} \otimes (\mu_{G/N}\circ q_N)_g.
\]
\end{claim}
We show that the two automorphisms are equal on all elementary tensors $b_0 \otimes
b_1 \in \mathcal{R} \overline \otimes \mathcal{R}$ (which we
identify with $\mathcal{R}$). We have
\begin{eqnarray*}
(\overline{\gamma \otimes (\mu_{G/N}\circ q_N)})_{x,g}(b_0 \otimes b_1) 
&\stackrel{\mathclap{\eqref{eq:defgroup}}}{=}& (\gamma \otimes (\mu_{G/N}\circ q_N))_g(1_{L^\infty(X)}
\otimes b_0 \otimes b_1)_x \\
&=& \gamma_g(1_{L^\infty(X)} \otimes b_0)_x \otimes (\mu_{G/N}\circ q_N)_g(b_1) \\
&\stackrel{\mathclap{\eqref{eq:defgroup}}}{=}&
\bar \gamma_{x,g}(b_0) \otimes (\mu_{G/N}\circ q_N)_g(b_1).
\end{eqnarray*}
This computation proves the claim.

\begin{claim} \label{claim:nuYzero}
Fix $g \in G \setminus N$. We claim that 
\[\nu\big(\big\{x\in X\colon (x,g)\in \mathcal{H}_\gamma 
 \mbox{ and } \overline{\gamma}_{x,g} \mbox{ is inner}\big\}\big)=0.
\]
\end{claim}

Denote by $Y$ the set in the above displayed equation, and
note 
that $Y$ is Borel. To prove the claim, 
assume by contradiction that $\nu(Y)>0$. 
By \cite[Theorem 3.4]{lance}, there is a a non-zero
central projection $q \in \mathcal{M}$ such that $\gamma_g$ is inner on $q \mathcal{M}$. This contradicts the 
assumption that $\gamma_g$ is properly outer (since 
$g \notin N$). 

Let $g\in G$. 
Using Claim~\ref{claim:nuYzero}, and up to removing a $\sigma$-invariant
measure zero set from $X$, we can assume that if there 
is $x \in X$ such that
$(x,g) \in \mathcal{N}_\gamma$ implies $g \in N$. Similarly, it follows from Claim \ref{claim:baraction} that if there exists $x\in X$ such that
$(x,g) \in \mathcal{N}_{\gamma \otimes
(\mu_{G/N}\circ q_N)}$, then $g \in N$. In fact, if $g \in G \setminus N$, then
$(\mu_{G/N}\circ q_N)_g$
is outer, which in turn forces $\bar \gamma_{x,g} \otimes (\mu_{G/N}\circ q_N)_g$
to be outer as well.

Finally, let $(x,g) \in X\times N$. Using Claim \ref{claim:baraction} in the
second to last equivalence, we have
\begin{align*}
(x,g) \in \mathcal{N}_\gamma &\Leftrightarrow \bar \gamma_{x,g} \text{ is inner} \\
&\Leftrightarrow \bar \gamma_{x,g} \otimes \text{id}_\mathcal{R} =
 \bar \gamma_{x,g} \otimes (\mu_{G/N}\circ q_N)_g \text{ is inner}\\
 & \Leftrightarrow (\overline{\gamma \otimes (\mu_{G/N}\circ q_N)})_{x,g}  \text{ is inner}\\
 &\Leftrightarrow
 (x,g) \in \mathcal{N}_{\gamma \otimes
(\mu_{G/N}\circ q_N)}.
\end{align*}
 
The equality $\chi_\gamma = \chi_{\gamma \otimes (\mu_{G/N}\circ q_N)}$ follows by the definition
of the invariant $\chi$ (\cite[Section 2]{suth}). Let 
$U \colon \mathcal{N}_\gamma \to \mathcal{U}(\mathcal{R})$ be a Borel map such that
$\bar \gamma_g = \text{Ad}(U(g))$ for every $g \in \mathcal{N}_\gamma$. Note that 
$(\overline{\gamma \otimes (\mu_{G/N}\circ q_N)})_n =
\bar \gamma_n \otimes \text{id}_{\mathcal{R}}$
by Claim \ref{claim:baraction}. Let $V\colon N\to \mathcal{U}(\mathcal{R}\overline{\otimes}\mathcal{R})$ be given by
$V=U\otimes 1_{\mathcal{R}}$. Then
\[
(\overline{\gamma \otimes (\mu_{G/N}\circ q_N)})_g = \text{Ad}(V(g))
\]
for every $g \in \mathcal{N}_\gamma$. This finishes the proof.
\end{proof}

We record here the following consequence of 
\autoref{thm:ErgThy}, which will be needed in Section~8.

\begin{cor}\label{prop:FinDirSumRMcDuff}
Let $(\mathcal{M},\tau)$ be a separably representable, hyperfinite, type II$_1$ von Neumann algebra, let $G$ be a countable, discrete, amenable group and
let $\gamma\colon G\to\Aut(\mathcal{M}, \tau)$ be an action which preserves $\tau$.
Then for every $d \in \N$, there is a unital homomorphism $M_d\to (\M^\U\cap \M')^{\gamma^\U}$.
\end{cor}
\begin{proof}
By taking $G=N$ in \autoref{thm:ErgThy}, it follows
that $(\mathcal{M},\gamma)$ absorbs $(\mathcal{R},\mathrm{id}_{\mathcal{R}})$ tensorially. 
By \autoref{thm:AbsR}, it follows that there exists a unital embedding $\mathcal{R} \to  (\M^\U\cap \M')^{\gamma^\U}$. Since there exists a unital homomorphism $M_d \to \mathcal{R}$ for every $d \in \N$, the conclusion follows.
\end{proof}

The following result is the
main application of CPoU in this section, and it is last ingredient we need in order to prove \autoref{cor:EmbedModelAction}.

\begin{prop} \label{prop:EmbedTracialUltrapow}
Let $G$ be a countable, discrete group,
let $A$ be a separable, unital \ca\ with non-empty trace space,
and let $\alpha\colon G \to \Aut(A)$ be an
action. Suppose that 
the induced action on $T(A)$ has finite orbits bounded 
in size by some constant, and that $(A,\alpha)$ has CPoU.
Let $\beta\colon G\to \Aut(B)$ be an action of $G$ on a
separable, unital \ca\ $B$, and suppose that for 
every $\tau\in T(A)^\alpha$ there exists 
an equivariant, unital homomorphism 
\[(B,\beta)\to \Big(\mathcal{M}_\tau^\U\cap  \mathcal{M}_\tau',({\alpha}^\tau)^\U \Big).\]
Then there exists an equivariant, unital homomorphism
$(B,\beta)\to (A^\U\cap A',\alpha^\U)$.
\end{prop}
\begin{proof} 
For each $\tau\in T(A)^\alpha$, fix an
equivariant, unital homomorphism 
\[\varphi_\tau\colon (B,\beta)\to \Big(\mathcal{M}_\tau^\U\cap \mathcal{M}_\tau',(\alpha^\tau)^\U \Big).\]

Let $\{a_n\}_{n\in\N}$ be a countable dense subset of the unit ball of $A$. Let $D_1$ be
the unit disk in $\mathbb{C}$ and set $\mathbb{Q}_1=\mathbb{Q}[i]\cap D_1$. 
Let $B_0$ be a countable dense subset of the unit ball
of $B$ containing $1_B$, which is invariant under the adjoint operation, 
multiplication, multiplication by scalars from $\mathbb{Q}_1$,
and the operation $(x,y)\mapsto \frac{1}{2}(x+y)$.
Since $G$
is countable, we can assume without 
loss of generality that $B_0$ is 
$\beta$-invariant.
Let $\{ b_n \}_{n \in \N}$ be an enumeration of 
$B_0$ such that $b_0 = 1_B$. Fix non-commuting variables
$x,y,z,w$, and for 
$\lambda \in \mathbb{Q}_1$
and $g \in G$ define the following
$G$-$*$-polynomials:
\begin{itemize}
\item $Q_{\text{id}}(x) = x -1$,
\item $Q_+(x,y,z) = \frac{1}{2}(x+y) - z$,
\item $Q_{\times}(x,y,z) = x y - z$,
\item $Q_\ast (x,y) = x^* - y$,
\item $Q_\lambda(x,y) = \lambda x  - y $,
\item $Q_g(x,y) = g\cdot x - y$,
\item $Q_{[\cdot, \cdot]}(x,w) = x w - w x$.
\end{itemize}
Clearly, $Q_{\text{id}}(b_n) = 0$ if and only if $n = 0$.
Given $n,m \in \N$ there is a unique $k^+_{(n,m)} \in \N$ such that
$Q_+(b_n,b_m,b_{k^+_{(n,m)}}) = 0$.
Similarly, there is a unique $k^\times_{(n,m)} \in \N$ such that
$Q_\times(b_n,b_m,b_{k^\times_{(n,m)}}) = 0$.
Analogously, given $n \in \N$, $\lambda \in \mathbb{Q}_1$ and $g \in G$ there are unique $k^*_n,
k^\lambda_n,k^g_n \in \N$ such that
\[
Q_*(b_n,b_{k^*_n}) = 0,
\quad
Q_\lambda(b_n,b_{k^\lambda_n}) = 0,
\quad
Q_g(b_n, b_{k^g_n}) = 0.
\]
Since $\varphi_\tau$ is an equivariant, unital homomorphism,
we have $Q_{\text{id}}(\varphi_\tau(1_B)) = 0$, moreover
for every $n,m \in \N$, $\lambda \in
\mathbb{Q}_1$ and $g \in G$ we have
\[
Q_+(\varphi_\tau(b_n),\varphi_\tau(b_m),\varphi_\tau(b_{k^+_{(n,m)}})) = 0, \quad Q_\times(\varphi_\tau(b_n),\varphi_\tau(b_m),\varphi_\tau(b_{k^\times_{(n,m)}})) = 0,
\]
and
\[
Q_*(\varphi_\tau(b_n),\varphi_\tau(b_{k^*_n})) = 0,
\quad
Q_\lambda(\varphi_\tau(b_n),\varphi_\tau(b_{k^\lambda_n})) = 0,
\quad
Q_g(\varphi_\tau(b_n), \varphi_\tau(b_{k^g_n})) = 0.
\]
Finally, 
for all $n,m \in \N$ we have
\[
Q_{[\cdot, \cdot]}(\varphi_\tau(b_n),\pi_\tau(a_m)) = 0,
\]
since $\varphi_\tau(b_n)$ belongs to the relative commutant of 
$\M_\tau$, and thus commutes with $\pi_\tau(A)$. 
By \autoref{lma:LocalToGlobal} (see also \autoref{vNaultra}) and \autoref{prop:ideal}, we can find contractions
$b'_n\in A^\U$, for $n\in\N$, such that
\begin{enumerate}[label=(\alph*)]
\item \label{glue:i0} $Q_\text{id}(b'_0) = 0$, that is $b'_0 = 1_{A^\U}$,
\item \label{glue:ia} 
for every $n,m \in \N$, we have
$Q_+(b'_n,b'_m,b'_{k^+_{(n,m)}}) =0=Q_\times(b'_n,b'_m,b'_{k^\times_{(n,m)}}) $,
\item \label{glue:ib} 
for every $n\in \N$, $\lambda \in \mathbb{Q}_1$ and $g \in G$, we
have
\[Q_*(b'_n, b'_{k_n^*})=Q_\lambda(b'_n, b'_{k_n^\lambda})= Q_g(b'_n, b'_{k_n^g})=0,\]
\item \label{glue:ic} for all $n,m \in \N$ we have
$Q_{[\cdot,\cdot]}(b'_n,\pi_\tau(a_m)) = 0$.
\end{enumerate}

Let $\Phi_0\colon B_0\to A^\U$ be the
map defined by sending
$b_n$ to $b_n'$ for all $n\in\N$.
This map is unital by \ref{glue:i0}, it is additive and multiplicative 
by \ref{glue:ia}, and it is $\ast$-preserving, equivariant and 
$\mathbb{Q}_1$-homogeneous by \ref{glue:ib}. It can be therefore extended uniquely
to a $\mathbb{Q}[i]$-linear, equivariant, unital homomorphism $\Phi_0\colon\text{span}(B_0) \to A^\U$.
The image of
$\Phi_0$ is contained in $A^\U\cap A'$ by \ref{glue:ic}.
Notice that $\Phi_0$ is contractive, since 
it is contractive on $B_0$, which is dense
in the unit ball
of $\text{span}(B_0)$.
This, along with the fact that $\text{span}(B_0)$
is dense in $B$, allows us to extend uniquely $\Phi_0$ to an
equivariant, unital homomorphism 
$(B,\beta)\to \big(A^\U\cap A',\alpha^\U \big)$,
as desired.
\end{proof}


\begin{proof}[Proof of \autoref{cor:EmbedModelAction}]
For every $\tau \in T(A)^\alpha$, we abbreviate $\pi_\tau(A)''$ by 
$\mathcal{M}_\tau$. Then
$(\mathcal{M}_\tau,\tau)$ is a separably representable, hyperfinite, type II$_1$ von 
Neumann algebra. Moreover, 
the dynamical system $(\mathcal{M}, \alpha^\tau)$ absorbs
$(\mathcal{R},\mu_{G/N}\circ q_N)$ by \autoref{thm:ErgThy},
since 
$\alpha^\tau_g$ is properly outer whenever $\alpha_g$ is strongly outer (see, for example, \cite[Remark~2.17]{GarLup_rigid_2018} or \cite[Proposition 5.7]{equiSI}).
The dynamical system $(\mathcal{R},\mu_{G/N}\circ q_N)$ is cocycle conjugate to
$(\overline{\bigotimes}_{n \in \N} \mathcal{R},  {\bigotimes}_{n \in \N}\mu_{G/N}\circ q_N)$ by \cite[Theorem 2.6]{ocneanu}, therefore by \autoref{thm:AbsR}
there exists an equivariant, unital embedding 
\[\varphi_\tau \colon (\mathcal{R}, \mu_{G/N}\circ q_N) \to \Big(\mathcal{M}_\tau^\mathcal{U}
\cap \mathcal{M}_\tau ', (\alpha^\tau)^\mathcal{U}\Big).\]

\begin{claim}
There exists a unital, monotracial, separable, simple, $ \mu_{G/N}\circ q_N$-invariant
$C^\ast$-subalgebra $B$ of $\mathcal{R}$.\hspace{-.1cm}~\footnote{This claim is immediate for certain specific outer actions of $G/N$ on $\mathcal{R}$, such as the Bernoulli shifts. On the other hand, and even though any two outer actions of $G/N$ on
$\mathcal{R}$ are cocycle conjugate, 
it is not clear how to obtain the claim for an arbitrary outer 
action only using that it is true for \emph{some} outer action, 
due to the 1-cocycle.}
\end{claim}
We recall that being a finite von Neumann algebra, the hyperfinite
II$_1$-factor $\mathcal{R}$ has the \emph{strong Dixmier property} (see 
\cite[Definition III.2.5.16, Theorem III.2.5.18]{blackadar}). This means that  
for every $a \in \mathcal{R}$ the norm-closed convex hull of
$\{ uau^*\colon u \in \mathcal{U}(\mathcal{R}) \}$ intersects the center
in exactly one element which, in this case, is necessarily 
$\tau_{\mathcal{R}}(a)$.
Given any $c \in \mathcal{R}$, we can thus find a countable set $W \subseteq \mathcal{U}(\mathcal{R})$
such that the intersection of $\overline{\text{conv}\{ ucu^* \colon u \in W \}}^{\|\cdot\|}$ 
with the center of
$\mathcal{R}$ is precisely $\{\tau_{\mathcal{R}}(c)\}$. If $C \subseteq \mathcal{R}$ is a $C^*$-subalgebra containing
$W\cup\{c\}$, then the ideal $I$
in $C$ generated by $c$ 
contains $\overline{\text{conv}\{ ucu^* \colon u \in W \}}^{\|\cdot\|}$. In particular, $I$ contains the scalar $\tau_{\mathcal{R}}(c)$, and thus $I = C$ if $c$ is a non-zero positive element. Furthermore, since all traces on $C$
are constant on $\overline{\text{conv}\{ ucu^* \colon u \in W \}}^{\|\cdot\|}$, they all map $c$ to $\tau_{\mathcal{R}}
(c)$.
Therefore, given a norm-separable
$C^*$-subalgebra $C \subseteq \mathcal{R}$, the
strong Dixmier property can be used a countable number
of times to obtain a separable,
simple, unital, monotracial \ca\ $B_0$ containing $1_\mathcal{R}$
such that $C \subseteq B_0 \subseteq \mathcal{R}$.
On the other hand, using the fact that $G$ is countable, we can find a 
separable, $(\mu_{G/N}\circ q_N)$-invariant
\ca\ $B_0'$ such that $B_0 \subseteq B_0' \subseteq \mathcal{R}$. By iterating this construction
we can build an increasing sequence
\[
B_0 \subseteq B_0' \subseteq \dots \subseteq B_n \subseteq B'_n \subseteq \dots \subseteq
\mathcal{R},
\]
such that $B_n$ is separable, simple, and monotracial, and such that $B_n'$
is separable, and $ \mu_{G/N}\circ q_N$-invariant, for every $n \in \N$. It follows
that the inductive limit of this sequence is monotracial, separable, simple, $(\mu_{G/N}\circ q_N)$-invariant and it contains $1_\mathcal{R}$. This proves the claim.

Let $B$ be a $C^*$-subalgebra of $\mathcal{R}$ as in the 
above claim, and let $\beta$ denote the restriction
of $\mu_{G/N}\circ q_N$ to $B$.
By restricting $\varphi_\tau$ to $(B,\beta)$, we obtain
an equivariant, unital embedding
$(B,\beta)\to (\mathcal{M}_\tau^\U\cap \mathcal{M}_\tau',(\alpha^\tau)^\U)$.
By \autoref{prop:EmbedTracialUltrapow}, there exists
an equivariant, unital homomorphism $\Phi\colon (B,\beta)
\to (A^\U\cap A',\alpha^\U)$, which is injective as $B$ is simple.
Since $B$ has a unique trace $\tau_B$, the map $\Phi$ is
$(\|\cdot\|_{2, \tau_B}$-$\|\cdot\|_{2, T_\U(A)})$-contractive. As moreover the norm-unit
ball of $A^\U$ is $\| \cdot \|_{2, T_\U(A)}$-complete (see \cite[Lemma 1.6]{cpou}), the map $\Phi$
can be extended by continuity to an equivariant unital 
embedding $(\mathcal{R}, \mu_{G/N}\circ q_N)\to (A^\U \cap A', \alpha_\U)$.
\end{proof}

The following is the main result of this section. It is
an equivariant 
version of \cite[Theorem 4.6]{uniformGamma}, and summarizes
the results of this and the previous section.

\begin{thm} \label{thm:mainCPoU-parts-1-3}
	Let $A$ be a separable, unital, nuclear \ca\ with non-empty trace space and 
	with no finite-dimensional
	quotients. Let $G$ be a countable, discrete, amenable group
	and let $\alpha\colon G \to \text{Aut}(A)$ be an action such that the 
	induced action
	on $T(A)$ has finite orbits bounded in size by a constant $M > 0$.
	Then the following are equivalent
	\begin{enumerate}
		\item \label{mainCPoU:gamma.cpou:1} $(A,\alpha)$ has uniform property 
		$\Gamma$,
		\item \label{mainCPoU:gamma.cpou:2} $(A, \alpha)$ has CPoU with 
		constant $M$,
		\item \label{mainCPoU:gamma.cpou:3} for every $n \in \N$ there is a 
		unital 
		embedding $M_n \to (A^\U \cap A')^{\alpha^\U}$.
	\end{enumerate}
\end{thm}

\begin{proof}
\eqref{mainCPoU:gamma.cpou:1} $\Rightarrow$ \eqref{mainCPoU:gamma.cpou:2} 
follows by \autoref{prop:cpou}, while
\eqref{mainCPoU:gamma.cpou:3} $\Rightarrow$ \eqref{mainCPoU:gamma.cpou:1} is an 
immediate consequence of the
definition of uniform property $\Gamma$, and both these implications do not 
require the assumption
that $A$ has no finite-dimensional quotients.

\eqref{mainCPoU:gamma.cpou:2} $\Rightarrow$ \eqref{mainCPoU:gamma.cpou:3}. 
Since $A$ is nuclear and has no
finite-dimensional quotients, $\mathcal{M}_\tau$ is a hyperfinite, type II$_1$ 
von Neumann
algebra for every $\tau \in T(A)$. By \autoref{thm:ErgThy} and 
\autoref{thm:AbsR},
the dynamical system $(\mathcal{R},\id_{\mathcal{R}})$ embeds unitally into 
$(\mathcal{M}^\U_\tau
\cap \mathcal{M}_\tau', (\alpha^\tau)^\U)$ for every $\tau \in T(A)^\alpha$.
Given $n \in \N$, by composing with a unital embedding
$M_n\to \mathcal{R}$ we get a unital equivariant embedding
$ (M_n, \text{id}_{M_n}) \to (\mathcal{M}_\tau^\U\cap \mathcal{M}_\tau', 
(\alpha^\tau)^\U)$.
The conclusion then follows by \autoref{prop:EmbedTracialUltrapow}.
\end{proof}

\section{Model actions with Rokhlin towers}
\label{s.model}

In this section, we analyze specific model examples of actions of residually 
finite groups which either have finite Rokhlin dimension or satisfy part of the 
definition. This is a technical step in our proof of the equivalences of Theorem \autoref{thmB}
(see \autoref{thm:mainRokDim})
concerning the Rokhlin dimension of the dynamical system $(A,\alpha)$:
it will be shown that under the
assumptions of Theorem \autoref{thmB}, strong outerness of $\alpha$ implies that those model 
actions can be embedded equivariantly into the central sequence algebra of $A$.

We begin by computing the Rokhlin dimension of a natural product-type action.

\begin{prop}\label{prop:DimRok1}
	Let $G$ be a finite group. Set $D= \bigotimes_{n\in\N} 
	\B(\ell^2(G)^{\otimes n}\oplus \C)$. Denote by 
	$\lambda\colon G\to \U(\ell^2(G))$ the left regular representation. Define 
	an action
	$\alpha\colon G\to\Aut(D)$ by 
	$\alpha_g=\bigotimes_{n\in\N}\Ad(\lambda^{\otimes n}_g\oplus 1)$,
	for all $g\in G$. Then $\dimRok(\alpha)=1$.
\end{prop}
\begin{proof}
	Given $m\in\N$, set $D_m= \B(\ell^2(G)^{\otimes m}\oplus \C)$ and let 
	$\alpha^{(m)}\colon G\to\Aut(D_m)$
	be the action given by
	$\alpha^{(m)}_g=\Ad(\lambda^{\otimes m}_g\oplus 1)$ for all $g\in G$.

	\begin{claim}
		Let $\varepsilon > 0$ and fix $n_0\in\N$. Then there exist $m\in\N$ 
		with $m\geq n_0$ and positive contractions $f_g^{(j)}\in D_m$,
		for $g\in G$ and $j=0,1$, satisfying
		\be
		\item \label{claimDR:1} $\alpha^{(m)}_g(f_h^{(j)})=f_{gh}^{(j)}$ for 
		all $g,h\in G$ and for all $j=0,1$,
		\item \label{claimDR:2} $f_g^{(j)}f_h^{(j)}=0$ for all $g,h\in G$ with 
		$g\neq h$ and for all $j=0,1$,
		\item \label{claimDR:3} $\left \| 1- \sum_{j=0}^1\sum_{g\in 
		G}f^{(j)}_{g} \right \| < \varepsilon$. \ee
	\end{claim}
	
	For the $\ep>0$ given, choose $m\in\N$ such that $|G|^{m-1}>1/\ep$ and also 
	$m\geq n_0$.
	By Fell's absorption principle (\cite[Theorem 2.5.5]{brownozawa}), if 
	$\pi\colon G\to \U(\mathcal{H})$ is any finite dimensional representation of $G$ on a 
	separable Hilbert space
	$\mathcal{H}$, then $\lambda\otimes\pi$ is unitarily equivalent to a direct sum of 
	$\dim(\mathcal{H})$ copies of $\lambda$.
	It follows that $\lambda^{\otimes m}$ is unitarily equivalent to 
	the direct sum of $|G|^{m-1}$ copies of $\lambda$, and thus
	$\lambda^{\otimes m}\oplus 1$ is unitarily equivalent to 
	$(\bigoplus_{k=1}^{|G|^{m-1}} \lambda)\oplus 1$.
	We fix such an identification for the remainder of the proof.
	
	Since 
	$\lambda$ contains a copy of the trivial representation 1,
	there is a unitary representaion $\widetilde{\lambda}\colon G\to 
	\mathcal{U}(V)$ 
	such that
	$\lambda$ is unitarily equivalent to $1 \oplus \widetilde{\lambda}$. Then
	$\lambda^{\otimes m}\oplus 1$ is unitarily conjugate to the diagonal
	representation
	$\diag(1,\widetilde{\lambda},1,\widetilde{\lambda},1 
	\dots,\linebreak \widetilde{\lambda},1)$, where the trivial representation appears 
	$|G|^{m-1}+1$ times and $\widetilde{\lambda}$ appears $|G|^{m-1}$ times.
	
	For $g\in G$, let $\delta_g\in \ell^2(G)$ be the corresponding
	Dirac function, and let
	$e_g \in \mathcal{B}(\ell^2(G))$ be the projection 
	onto the span of $\delta_g$.
	By taking a suitable unitary 
	conjugation of the $e_g$,
	we find projections $p_g\in \B(\C\oplus V)$ satisfying 
	$\sum_{g \in G} p_g = 1$ and $\Ad (1 \oplus \widetilde{\lambda}_g)(p_h) = 
	p_{gh}$ 
	for all $g,h \in G$. Similarly, let $q_g\in \B(V\oplus \C)$, 
	for $g\in G$, be 
	projections satisfying $\sum_{g \in G} q_g = 1$ and $\Ad 
	(\widetilde{\lambda}_g \oplus 1)(q_h) = q_{gh}$ for all $g,h \in G$.

	Let $a_0\colon [0,|G|^{m-1}]\to [0,1]$ be defined as 
	$a_0(x)=\frac{x}{|G|^{m-1}}$, and set $a_1=1-a_0$.
	For $g\in G$, set
	\[f_g^{(0)}=\diag(0,a_0(1)q_g,\ldots,a_0(|G|^{m-1})q_g)\]
	and
	\[f_g^{(1)}=\diag(a_1(1)p_g,\ldots,a_1(|G|^{m-1})p_g,0) .\]
	These are positive contractions satisfying conditions \eqref{claimDR:1}, 
	\eqref{claimDR:2} and \eqref{claimDR:3} above, and the claim is 
	proved.
	
We note that, since $G$ is a finite group, it suffices to take $H=\{ e\}$ in 
\autoref{def:frd}.
The claim shows that there exist Rokhlin towers in $D$ satisfying 
conditions \ref{equiv:a}, \ref{equiv:b} and \ref{equiv:c} in 
\autoref{rmk:EquivalenceDimRok}
for $d=1$ and $H=\{e\}$. We now explain how to find new towers satisfying these 
	conditions in addition to condition \ref{equiv:d}.
	Let $F\subseteq D$ be a finite set. For the $\ep>0$ given above, find 
	$n_0\in\N$ such that for every $n\geq n_0$, the unital equivariant 
	embedding $\varphi_n\colon (D_n,\alpha^{(n)})\hookrightarrow (D,\alpha)$ 
	into the $n$-th coordinate satisfies
	\[\|\varphi_n(x)a-a\varphi_n(x)\|\leq \frac{\ep}{2} \|x\|\]
	for all $x\in D_n$ and all $a\in F$. Use the claim to find $m\in\N$ with 
	$m\geq n_0$ and positive contractions $f_g^{(j)}\in D_m$,
	for $g\in G$ and $j=0,1$, satisfying conditions \eqref{claimDR:1}, 
	\eqref{claimDR:2} and \eqref{claimDR:3} above for the tolerance $\ep/2$.
	One checks that the positive contractions $\varphi_m(f_g^{(j)})\in D$,
	for $g\in G$ and $j=0,1$, satisfy conditions \ref{equiv:a} through 
	\ref{equiv:d} in \autoref{rmk:EquivalenceDimRok}, as desired.
	It follows that $\dimRok(\alpha)\leq 1$.
	
	Finally, $\dimRok(\alpha)=0$ is impossible because the unit of $D$ is not 
	divisible by $|G|$ in $K_0(D)$.
\end{proof}

Using the computation above, we will show that certain canonical actions on 
dimension drop algebras admit
Rokhlin towers that satisfy all the conditions in \autoref{def:frd} except for 
centrality. We define these actions next.

\begin{df}\label{df:ActonIk}
	Let $G$ be a finite group, and let $k\in\N$. We denote by
	$I^{(k)}_G$ the dimension drop algebra
	\[ I_{G}^{(k)}=\left\{f\in C\Big([0,1], \B(\ell^2(G)^{\otimes k}) \otimes 
	\B(\ell^2(G)^{\otimes k}\oplus \C)\Big)\colon
	\begin{tabular}{@{}l@{}}
		$f(0)\in \B(\ell^2(G)^{\otimes k})\otimes 1,$\\
		$f(1)\in 1\otimes \B(\ell^2(G)^{\otimes k}\oplus \C)$
	\end{tabular}
	\right\} .\]
	We denote by $\mu^{(k)}_G\colon G\to \Aut(I^{(k)}_{G})$ the restriction 
	to $I_{G}^{(k)}$ of the action of $G$ on $C([0,1]) \otimes
	\mathcal{B}(\ell^2(G)^{\otimes k}) \otimes \mathcal{B}(\ell^2(G)^{\otimes 
	k} \oplus \mathbb{C})$
	given by $\text{id}_{C([0,1])} \otimes \text{Ad}(\lambda^{\otimes k}) 
	\otimes \text{Ad}(\lambda^{\otimes
		k} \oplus 1_{\mathbb{C}})$.
\end{df}

\begin{prop}\label{prop:GactIkdim2}
	Let $\ep>0$, let $G$ be a finite group, and adopt the notation 
	for $(I_G^{(k)},\mu_G^{(k)})$ from \autoref{df:ActonIk}.
	Then there exist $k\in\N$ and positive contractions $f_g^{(j)}\in 
	I_G^{(k)}$, for
	$g\in G$ and $j=0,1,2$, satisfying
	\begin{enumerate}[label=(\alph*)]
		\item \label{claimIk:a} 
		$\big\|(\mu^{(k)}_G)_{g}(f^{(j)}_{h})-f^{(j)}_{gh}\big\|<\ep$ for  
		$j=0,1,2$, and for all $g,h\in G$,
		\item \label{claimIk:b} $\|f^{(j)}_{g}f^{(j)}_{h}\|<\ep$ for  
		$j=0,1,2$, and for all $g,h\in G$ with $g\neq h$,
		\item \label{claimIk:c} $\left\|1-\sum_{g\in 
		G}f^{(0)}_{g}+f^{(1)}_{g}\right\|<\ep$.
	\end{enumerate}
\end{prop}
\begin{proof}
By \autoref{prop:DimRok1} (see specifically the claim in its proof), we can 
find $k\in\N$ and positive contractions
$\widetilde{f}_g^{(j)}\in \B(\ell^2(G)^{\otimes k}\oplus \C)$, for $g\in G$ 
and $j=0,1$, satisfying
\be\item[(i)] $\|\Ad(\lambda^{\otimes k}_g\oplus 
1)(\widetilde{f}^{(j)}_{h})-\widetilde{f}^{(j)}_{gh}\|<\ep$ for all 
$j=0,1$, and for all $g,h\in G$,
\item[(ii)] $\|\widetilde{f}^{(j)}_{g}\widetilde{f}^{(j)}_{h}\|<\ep$ for 
all $j=0,1$, and for all $g,h\in G$ with $g\neq h$,
\item[(iii)] $\left\|1-\sum_{g\in G}(\widetilde{f}^{(0)}_g+ 
\widetilde{f}^{(1)}_{g})\right\|<\ep$.
\ee

For $g\in G$, denote by $e_g\in \B(\ell^2(G))$ the projection onto the span 
of $\delta_g\in \ell^2(G)$. Then
$\Ad(\lambda_g)(e_h)=e_{gh}$ for all $g,h\in G$, and $\sum_{g\in G}e_g=1$. 
Regard $e_g$ as an element in
$\B(\ell^2(G)^{\otimes k})\cong \B(\ell^2(G))^{\otimes k}$ via the first 
factor embedding.

Let $h_0\in C([0,1])$ denote the inclusion of $[0,1]$ into $\C$, and set 
$h_1=1-h_0\in C([0,1])$. For $g\in G$
and $j=0,1,2$, set
\[f_g^{(j)}=\begin{cases} h_0\widetilde{f}_g^{(j)} &\mbox{if } j=0,1; \\
h_1e_g & \mbox{if } j=2.\end{cases}\]

One checks that conditions \ref{claimIk:a}, \ref{claimIk:b} and 
\ref{claimIk:c} in the statement are satisfied, completing the proof.
\end{proof}

Next, we give a recipe for constructing unital equivariant homomorphisms from 
$(I^{(k)}_{G},\mu^{(k)}_G)$; see \autoref{thm:univprop}.
We do so in a generality greater than necessary, because the proof is not more 
complicated and in fact the
higher level of abstraction makes the argument conceptually clearer. 

We need some preparatory facts about $C(X)$-algebras first.
The following definition is standard.
For a \ca\ $A$, we denote its center by $Z(A)$.

\begin{df}\label{df:CX-alg}
	Let $X$ be a compact Hausdorff space.
	A unital \emph{$C(X)$-algebra} is a pair $(A,\zeta)$ consisting of a unital \ca\ $A$ 
	and
	a unital homomorphism $\zeta\colon C(X)\to Z(A)$. 
	If $A$ is a unital $C(X)$-algebra and $U\subseteq X$ is open, then 
	$\zeta(C_0(U))A$
	is an ideal in $A$. Given $x\in X$, the \emph{fiber over $x$} is the 
	quotient
	\[A(x):=A/\zeta(C(X\setminus\{x\}))A.\]
	For $a\in A$, we write $a_x$ for its image in $A(x)$.
	
	If $(A,\zeta_A)$ and $(B,\zeta_B)$ are unital $C(X)$-algebras and 
	$\varphi\colon A\to B$ is a homomorphism, we say 
	that $\varphi$ is a \emph{$C(X)$-homomorphism} if 
	$\varphi\circ \zeta_A=\zeta_B$. If this is the case, then for every 
	$x\in X$, the map $\varphi$ induces a unital homomorphism between the 
	corresponding fibers, which we denote by
	$\varphi_x\colon A(x)\to B(x)$.
	
\end{df}

As is costumary, we will usually suppress $\zeta$ from the 
notation for $C(X)$-algebras. Recall that if $A$ is a 
$C(X)$-algebra and $a\in A$, then 
the function $x\mapsto \|a_x\|$ is upper-semicontinuous.
We assume that the following proposition may well be known, but we were not 
able to find it in the literature. 

\begin{prop}\label{prop:CXmaps}
	Let $X$ be a compact Hausdorff space, let $A$ and $B$ be unital
	$C(X)$-algebras, and let $\varphi\colon A\to B$ be a
	$C(X)$-homomorphism. Then $\varphi$ is injective (respectively,
	surjective) if and only if $\varphi_x$ is injective (respectively, 
	surjective) for all $x\in X$. In particular, $\varphi$ is an 
	isomorphism if and only if $\varphi_x$ is an isomorphism for 
	every $x\in X$.
\end{prop}
\begin{proof}
	We begin with the assertion regarding injectivity. 
	Assume that $\varphi$ is injective, and fix $x\in X$. 
	Arguing by contradiction, assume that $\varphi_x$ is not 
	injective, and find $a\in A$ such 
	that $\|a_x\|=1$ and $\varphi_x(a_x)=0$. By upper-semicontinuity
	of the norm function on $B$, there is an open neighborhood $U$ of 
	$x$ such that $\|\varphi(a)_y\|<1/2$ for all $y\in U$.
	Let $V$ be an open neighborhood of $x$ such that 
	$\overline{V}\subseteq U$.
	Let $f\in C(X)$ be supported on $V$ and such that $f(x)=1$. Using part~(ii) 
	of 
	Lemma~2.1 of \cite{dadarlat09} at 
	the fourth step, we get
	\[1 = \|f(x)a_x\| \leq \|\varphi(fa)\|=\|f\varphi(a)\|\leq \sup_{y\in 
		\overline{V}}
	\|f(y)\varphi(a)_y\|<1/2 .
	\]
	This contradiction implies that $\varphi_x$ is injective, as desired.
	
	Conversely, assume that $\varphi_x$ is injective for all $x\in X$
	and let 
	$a\in \ker(\varphi)$. Since $a_x\in \ker(\varphi_x)$ for all
	$x\in X$, we must have $a_x=0$ for all $x\in X$, and thus $a=0$. 
	
	We now prove the statement about surjectivity. Assume that 
	$\varphi$ is surjective, and let $x\in X$. Denote by 
	$\pi_x^A\colon A\to A(x)$ and $\pi_x^B\colon B\to B(x)$ the 
	corresponding quotient maps, and note that 
	$\pi_x^B\circ \varphi = \varphi_x\circ \pi_x^A$. Let $c\in B(x)$,
	and find $b\in B$ such that $\pi_x^B(b)=c$. Since $\varphi$ is surjective, 
	there is $a\in A$ with $\varphi(a)=b$. Then 
	$\varphi_x(a_x)=\varphi(a)_x=c$. Thus $\varphi_x$ is surjective.
	
	Conversely, assume that $\varphi_x$ is surjective for all $x\in X$.
	Let $b\in B$ and let $\ep>0$. 
	Given $x\in X$, find $c^{(x)}\in A(x)$ such that
	$\varphi_x(c^{(x)})=b_x$. Since $\pi^A_x$ is surjective,
	we find $a^{(x)}\in A$ such that $a^{(x)}_x=c^{(x)}$.
	Since the norm function on $A$ is upper-semicontinuous, there
	exists an open set $U_x\subseteq X$ containing $x$ such that 
	$\|\varphi(a^{(x)}_y)-b_y\|<\ep$ for every $y\in U_x$. Since 
	$X$ is compact, we can find $x_1,\ldots,x_n$ such that 
	$\mathcal{U}=\{U_{x_1},\ldots, U_{x_n}\}$ covers $X$. Let $f_1,\ldots,f_n 
	\in C(X)$ be a partition of unity subordinate to 
	$\mathcal{U}$, and set $a=\sum_{j=1}^n f_ja^{(x_j)}\in A$. 
	Then $\|\varphi(a)_x-b_x\|<\ep$ for every $x\in X$, and hence
	$\|\varphi(a)-b\|<\ep$ by part~(ii) in Lemma~2.1 of 
	\cite{dadarlat09}.
	Since $\ep>0$ is arbitrary, 
	we conclude that $\varphi$ is surjective. \end{proof}


\begin{rem}
	\label{rem:generators-Mn}
	Let $n\in\N$. Recall that the matrix algebra $M_n$ is the universal $C^*$-algebra generated by
	elements $\{e_{1,k}\}_{k=1}^n$ satisfying $e_{1,k} e_{1,j} = 0$ when $k 
	\not= 1$,
	$e_{1,k}e_{1,j}^* = \delta_{k,j}e_{1,1}$ for all $j, k =1, \dots, n $ and 
	such that $e_{1,1}$ is a
	projection.
	By setting $e_{j,k} = e_{1,j}^* e_{1,k}$ for every $j, k =1, \dots, n $, it 
	is well-known that $M_n$
	is also the universal \ca\ generated by $\{ e_{j,k} \}_{1\leq j,k \le n}$ 
	with the relations $e_{j,k}^* = e_{k,j}$
	and $e_{i,j}e_{k,\ell} = \delta_{j,k} e_{i, \ell}$ for every $i,j,k, \ell 
	=1,\ldots, n$.
	We will use freely both representations.
\end{rem}

The following is an equivariant version of a well-known characterization of the 
dimension
drop algebra from \cite[Proposition 5.1]{RorWin_algebra_2010}. As it is not 
clear how the proof in \cite[Proposition 5.1]{RorWin_algebra_2010} can be 
adapted to the equivariant setting, we give here a different and more explicit 
proof.

\begin{thm}\label{thm:univprop}
	Let $G$ be a finite group,
	let $B$ be a unital \ca, and let $\beta\colon G\to\Aut(B)$ be an action.
	Let $n \in \N$, let $v \colon G \to M_n$ be a unitary representation, and suppose there is
	a rank-one projection $e$ in $M_n$ such that $v_g e = e$ for all $g \in G$.
	Suppose that there exist a completely positive contractive equivariant 
	order zero map
	\[\xi\colon \left(M_n,\Ad(v) \right) \to (B,\beta)\]
	and a contraction $s\in B^\beta$ satisfying $\xi(e)s=s$ and $\xi(1)+s^*s=1$.
	Let $\gamma$ be the restriction to $I_{n,n+1}$ of the action of $G$ on 
	$C([0,1]) \otimes M_n \otimes M_{n+1}$ given by
	\[
	\id_{C([0,1])} \otimes \Ad(v) \otimes \Ad(v \oplus 1_{\C}).
	\]
	Then there exists a unital, equivariant homomorphism
	$
	\varphi\colon (I_{n,n+1},\gamma)\to (B,\beta)$.
	
\end{thm}

\begin{proof}
	Denote by $D$ the universal \ca~generated by the set $\{s,f_{j,k}\colon 
	j,k=1,\ldots,n\}$ of contractions satisfying:
	\begin{enumerate}[label=(R.\arabic*)]
		\item \label{item:f1}$f_{j,k}^* = f_{k,j}$ for all $j,k=1,\ldots,n$,
		\item \label{item:f2} $f_{j,k}f_{\ell,m} = 
		\delta_{k,\ell}f_{j,j}f_{j,m}$ for all $j,k,\ell,m=1,\ldots,n$,
		\item \label{item:f3} $f_{1,1}s = s$,
		\item \label{item:f4} $\sum_{j=1}^n f_{j,j} + s^*s = 1$.
	\end{enumerate}

It is clear that $B$ admits a unital homomorphism
from $D$. Most of the proof consists of
showing that $D$ is isomorphic to the dimension drop algebra $I_{n,n+1}$, which will then give us a unital
homomorphism $\varphi\colon I_{n,n+1}\to B$. The last
step will be to show that the map $\varphi$ can be 
chosen to be equivariant.

Consider the following matrices in $M_{n(n+1)}$
(each vertical line comes after $n+1$ entries, and the horizontal line 
appears after $n$ rows):
\[
\begin{matrix}
F_{1,1} =	\left(\begin{array}{@{}ccccc|ccccc|c|ccccc@{}}
			1 & 0 & \cdots & 0 & 0  & 0 & 0 & \cdots & 0 & 0		    & 
			\cdots &  0 & 0 & \cdots & 0 & 0 \\
			0 & 1 & \cdots & 0 & 0 			  & 0 & 0 & \cdots & 0 & 0 	& 
			\cdots &  0 & 0 & \cdots & 0 & 0 \\
			\vdots & \vdots &  \vdots &  \vdots & \vdots			  &  \vdots 
			&  \vdots &  \vdots &  \vdots &  \vdots  &  \vdots &   \vdots &  
			\vdots & \vdots &  \vdots & \vdots \\
			0 & 0 & \cdots & 1 & 0			  & 0 & 0 & \cdots & 0 & 0  & 
			\cdots &  0 & 0 & \cdots & 0 & 0 \\\hline
			\phantom{0} & \phantom{0} & \vdots &\phantom{0} & \phantom{0} &
			\phantom{0} & \phantom{0} & \vdots &\phantom{0} & \phantom{0} & 
			\vdots &
			\phantom{0} & \phantom{0} & \vdots &\phantom{0} & \phantom{0} \\
			\hline
			0 & 0 & \cdots & 0 &0			  & 0 & 0 & \cdots & 0 & 0  & 
			\cdots &  0 & 0 & \cdots & 0 & 0 \\
			\vdots & \vdots &  \vdots &  \vdots & \vdots			  &  \vdots 
			&  \vdots &  \vdots &  \vdots &  \vdots  &  \vdots &   \vdots &  
			\vdots & \vdots &  \vdots & \vdots \\
			0 & 0 & \cdots & 0 &0			  & 0 & 0 & \cdots & 0 & 0  & 
			\cdots &  0 & 0 & \cdots & 0 & 0 \\
		\end{array}\right)\end{matrix} \]

\[
\begin{matrix}
F_{1,2} =	\left(\begin{array}{@{}ccccc|ccccc|c|ccccc@{}}
	0 & 0 & \cdots & 0 & 0  & 1 & 0 & \cdots & 0 & 0		    & 
	\cdots &  0 & 0 & \cdots & 0 & 0 \\
	0 & 0 & \cdots & 0 & 0 			  & 0 & 1 & \cdots & 0 & 0 	& 
	\cdots &  0 & 0 & \cdots & 0 & 0 \\
	\vdots & \vdots &  \vdots &  \vdots & \vdots			  &  \vdots 
	&  \vdots &  \vdots &  \vdots &  \vdots  &  \vdots &   \vdots &  
	\vdots & \vdots &  \vdots & \vdots \\
	0 & 0 & \cdots & 0 & 0			  & 0 & 0 & \cdots & 1 & 0  & 
	\cdots &  0 & 0 & \cdots & 0 & 0 \\\hline
	\phantom{0} & \phantom{0} & \vdots &\phantom{0} & \phantom{0} &
	\phantom{0} & \phantom{0} & \vdots &\phantom{0} & \phantom{0} & 
	\vdots &
	\phantom{0} & \phantom{0} & \vdots &\phantom{0} & \phantom{0} \\
	\hline
	0 & 0 & \cdots & 0 &0			  & 0 & 0 & \cdots & 0 & 0  & 
	\cdots &  0 & 0 & \cdots & 0 & 0 \\
	\vdots & \vdots &  \vdots &  \vdots & \vdots			  &  \vdots 
	&  \vdots &  \vdots &  \vdots &  \vdots  &  \vdots &   \vdots &  
	\vdots & \vdots &  \vdots & \vdots \\
	0 & 0 & \cdots & 0 &0			  & 0 & 0 & \cdots & 0 & 0  & 
	\cdots &  0 & 0 & \cdots & 0 & 0 \\
\end{array}\right)\end{matrix} \]
	
	\[
	\vdots
	\]
	\[  \begin{matrix} F_{1,n}= 	
	\left(\begin{array}{@{}ccccc|ccccc|c|ccccc@{}}
			0 & 0 & \cdots & 0 & 0  & 0 & 0 & \cdots & 0 & 0 			    & 
			\cdots &  1 & 0 & \cdots & 0 & 0 \\
			0 & 0 & \cdots & 0 & 0 			  & 0 & 0 & \cdots & 0 & 0 	& 
			\cdots &  0 & 1 & \cdots & 0 & 0 \\
			\vdots & \vdots &  \vdots &  \vdots & \vdots			  &  \vdots 
			&  \vdots &  \vdots &  \vdots &  \vdots  &  \vdots &   \vdots &  
			\vdots & \vdots &  \vdots & \vdots \\
			0 & 0 & \cdots & 0 & 0			  & 0 & 0 & \cdots & 0 & 0  & 
			\cdots &  0 & 0 & \cdots & 1 & 0\\\hline
			\phantom{0} & \phantom{0} & \vdots &\phantom{0} & \phantom{0} &
			\phantom{0} & \phantom{0} & \vdots &\phantom{0} & \phantom{0} & 
			\vdots &
			\phantom{0} & \phantom{0} & \vdots &\phantom{0} & \phantom{0} \\
			\hline
			0 & 0 & \cdots & 0 &0			  & 0 & 0 & \cdots & 0 & 0  & 
			\cdots &  0 & 0 & \cdots & 0 & 0 \\
			\vdots & \vdots &  \vdots &  \vdots & \vdots			  &  \vdots 
			&  \vdots &  \vdots &  \vdots &  \vdots  &  \vdots &   \vdots &  
			\vdots & \vdots &  \vdots & \vdots \\
			0 & 0 & \cdots & 0 &0			  & 0 & 0 & \cdots & 0 & 0  & 
			\cdots &  0 & 0 & \cdots & 0 & 0 \\
		\end{array}\right) \end{matrix}
	\vspace{3pt}
	\]
	\[
	\begin{matrix}F_{1,n+1} = 	\left(\begin{array}{@{}ccccc|ccccc|c|ccccc@{}}
			0 & 0 & \cdots & 0 & 1  & 0 & 0 & \cdots & 0 & 0 			    & 
			\cdots &  0 & 0 & \cdots & 0 & 0 \\
			0 & 0 & \cdots & 0 & 0 			  & 0 & 0 & \cdots & 0 & 1 	& 
			\cdots &  0 & 0 & \cdots & 0 & 0 \\
			\vdots & \vdots &  \vdots &  \vdots & \vdots			  &  \vdots 
			&  \vdots &  \vdots &  \vdots &  \vdots  &  \vdots &   \vdots &  
			\vdots & \vdots &  \vdots & \vdots \\
			0 & 0 & \cdots & 0 & 0			  & 0 & 0 & \cdots & 0 & 0  & 
			\cdots &  0 & 0 & \cdots & 0 & 1 \\\hline
			\phantom{0} & \phantom{0} & \vdots &\phantom{0} & \phantom{0} &
			\phantom{0} & \phantom{0} & \vdots &\phantom{0} & \phantom{0} & 
			\vdots &
			\phantom{0} & \phantom{0} & \vdots &\phantom{0} & \phantom{0} \\
			\hline
			0 & 0 & \cdots & 0 &0			  & 0 & 0 & \cdots & 0 & 0  & 
			\cdots &  0 & 0 & \cdots & 0 & 0 \\
			\vdots & \vdots &  \vdots &  \vdots & \vdots			  &  \vdots 
			&  \vdots &  \vdots &  \vdots &  \vdots  &  \vdots &   \vdots &  
			\vdots & \vdots &  \vdots & \vdots \\
			0 & 0 & \cdots & 0 &0			  & 0 & 0 & \cdots & 0 & 0  & 
			\cdots &  0 & 0 & \cdots & 0 & 0 \\
		\end{array}\right)
	\end{matrix}
	\]
	The elements $F_{1,j}$, for $j=1,2,\ldots,n+1$, are 
	partial isometries satisfying
	$F_{1,j}F_{1,i} = 0$ when $j \not=1$,
	$F_{1,j}F_{1,i}^* = \delta_{i,j}F_{1,1}$ for all $i,j \in 
	\{1,2,\ldots,n+1\}$, and $\sum_{j=1}^{n+1} F_{1,j}^*F_{1,j} = 1$. Therefore 
	they generate a unital copy of $M_{n+1}$.
	We identify $I_{n,n+1}$ with the algebra of continuous functions from 
	$[0,1]$ to $M_{n(n+1)}$ such that $f(1)$ is in the $C^*$-algebra generated 
	by $\{F_{1,k} \mid k=1,2,\ldots,n+1 \}$ (which is isomorphic to $M_{n+1}$) 
	and $f(0)$ is in the commutant of the $C^*$-algebra generated by $\{F_{1,k} 
	\mid k=1,2,\ldots,n+1 \}$ (which is isomorphic to $M_n$).
	
	Denote by $\rho \colon [0,1] \to [0,1]$ the identity function.
	For $j,k=1,\ldots,n$, let $\tilde{f}_{j,k} \in I_{n,n+1}$ be the 
	matrix-valued function which,
	written in block form where each block has size $(n+1) \times (n+1)$, has 
	in its $(j,k)$-th block the diagonal matrix valued function
	$\diag(\underbrace{1,1,\ldots,1}_{n \textrm{ times}},1-\rho)$, and $0$ 
	elsewhere. Let $\tilde{s}$ be the matrix-valued function
	\[
	\tilde{s} = \left(\begin{array}{@{}ccccc|ccccc|c|ccccc@{}}
		0 & 0 & \cdots & 0 & \sqrt{\rho}  & 0 & 0 & \cdots & 0 & 0 			    
		& \cdots &  0 & 0 & \cdots & 0 & 0 \\
		0 & 0 & \cdots & 0 & 0 			  & 0 & 0 & \cdots & 0 & \sqrt{\rho} 	
		& \cdots &  0 & 0 & \cdots & 0 & 0 \\
		\vdots & \vdots &  \vdots &  \vdots & \vdots			  &  \vdots &  
		\vdots &  \vdots &  \vdots &  \vdots  &  \vdots &   \vdots &  \vdots & 
		\vdots &  \vdots & \vdots \\
		0 & 0 & \cdots & 0 & 0			  & 0 & 0 & \cdots & 0 & 0  & \cdots &  
		0 & 0 & \cdots & 0 & \sqrt{\rho} \\\hline
		\phantom{0} & \phantom{0} & \vdots &\phantom{0} & \phantom{0} &
		\phantom{0} & \phantom{0} & \vdots &\phantom{0} & \phantom{0} & \vdots &
		\phantom{0} & \phantom{0} & \vdots &\phantom{0} & \phantom{0} \\
		\hline
		0 & 0 & \cdots & 0 &0			  & 0 & 0 & \cdots & 0 & 0  & \cdots &  
		0 & 0 & \cdots & 0 & 0 \\
		\vdots & \vdots &  \vdots &  \vdots & \vdots			  &  \vdots &  
		\vdots &  \vdots &  \vdots &  \vdots  &  \vdots &   \vdots &  \vdots & 
		\vdots &  \vdots & \vdots \\
		0 & 0 & \cdots & 0 &0			  & 0 & 0 & \cdots & 0 & 0  & \cdots &  
		0 & 0 & \cdots & 0 & 0 \\
	\end{array}\right),
	\]
	where each vertical dividing line represents $n+1$ entries.

	One checks that the functions $\tilde{f}_{j,k}$ for $j,k=1,\ldots,n$ and 
	$\tilde{s}$ satisfy the relations defining $D$.
	Fix a unital homomorphism $\pi \colon D \to I_{n,n+1}$ satisfying 
	$\pi(f_{j,k})=\tilde{f}_{j,k}$ for all $j,k=1,\ldots,n$ and 
	$\pi(s)=\tilde{s}$.
	We will show that $\pi$ is an isomorphism.

	\begin{claim} \label{claim:ddrop1}
		The following identities hold:
		\be
		\item  $f_{j,j}f_{j,k} = f_{j,k}f_{k,k}$ for all $j,k =1,\ldots,n$.
		\item For  $j =1,\ldots,n$, we have $s f_{j,1} s = 0$.
		\ee
		
	\end{claim}
	
	The first equality follows from item \ref{item:f2} in the list of relations 
	defining $D$. The following computation establishes the second identity:
	\begin{eqnarray*}
		(s f_{j,1} s)^*(s f_{j,1} s) &=&
		s^*f_{1,j}s^*sf_{j,1}s  \\
		&\stackrel{\mathclap{\text{\ref{item:f4}}}}{=}&   s^*f_{1,j}\Big(1 - 
		\sum_{i=1}^n f_{i,i} \Big)f_{j,1}s \\
		&\stackrel{\mathclap{\text{\ref{item:f2}}}}{=}& s^* f_{1,j}f_{j,1}s -  
		s^* f_{1,j}f_{j,j}f_{j,1}s \\
		&\stackrel{\mathclap{\text{\ref{item:f2}}}}{=}& s^*f_{1,1}^2s - 
		s^*f_{1,1}^3s \ \ \stackrel{\mathclap{\text{\ref{item:f3}}}}{=} \ \ 0.
	\end{eqnarray*}
	This proves the claim.
	
	Set
	\begin{equation} \label{eqn:b}
		b = s^*s + \sum_{j=1}^n f_{j,1}ss^*f_{1,j} = s^*s + ss^* + \sum_{j=2}^n 
		f_{j,1}ss^*f_{1,j}.
	\end{equation}
	Note that $b$ is a positive contraction (since all of 
	the summands in its definition are pairwise orthogonal positive contractions).
	One checks that $\pi(b) = \rho \cdot 1$, and therefore $\spec(b) = [0,1]$.
	\newline
	
	\begin{claim}
		The element $b$ belongs to the center of $D$.
	\end{claim}
	Let $j,k=1,\ldots,n$.
	Then
	\[
	f_{j,k}b \ \  \stackrel{\mathclap{\text{\ref{item:f2}}}}{=} \ \ f_{j,k}s^*s + 
	f_{j,j}f_{j,1}ss^*f_{1,k} \ \mbox{ and } \
	bf_{j,k} \ \ \stackrel{\mathclap{\text{\ref{item:f2}}}}{=} \ \ s^*sf_{j,k} 
	+f_{j,1}ss^*f_{1,k} f_{k,k}.
	\]
	We show term by term that both expressions agree, which will imply the 
	claim.
	For the first terms in both right hand sides, we have
	\begin{eqnarray*}
		f_{j,k}s^*s \ \ \stackrel{\mathclap{\text{\ref{item:f4}}}}{=} \ \ f_{j,k} \left 
		(1 -\sum_{i=1}^n f_{i,i} \right) 
		&\stackrel{\mathclap{\text{\ref{item:f2}}}}{=}&  f_{j,k} - f_{j,k}f_{k,k} 
                \ \ \stackrel{\mathclap{\text{\ref{item:f2}}}}{=} \ \  f_{j,k} - f_{j,j}f_{j,k} 
		\\
		&\stackrel{\mathclap{\text{\ref{item:f2}}}}{=}& \Big(1 -\sum_{i=1}^n 
		f_{i,i} \Big) f_{j,k} \ \ \stackrel{\mathclap{\text{\ref{item:f4}}}}{=} \ \  
		s^*sf_{j,k}.
	\end{eqnarray*}
	For the second terms in both right hand sides, we have
	\begin{eqnarray*}
		f_{j,j}f_{j,1}ss^*f_{1,k}  
		&\stackrel{\mathclap{\text{\ref{item:f2}}}}{=}&  f_{j,1} (f_{1,1}s) 
		s^*f_{1,k} \stackrel{\mathclap{\text{\ref{item:f3}}}}{=}  
		f_{j,1}ss^*f_{1,k} \\
		&\stackrel{\mathclap{\text{\ref{item:f3}}}}{=}&   
		f_{j,1}s(f_{1,1}s)^*f_{1,k} 
		\stackrel{\mathclap{\text{\ref{item:f2}}}}{=}  f_{j,1}ss^*f_{1,k} 
		f_{k,k}.
	\end{eqnarray*}
	Likewise, using the fact that $s f_{i,1} s = 0$ for all $i=1,\ldots,n$, 
	proved in Claim \autoref{claim:ddrop1}, we get
	\begin{equation} \label{eqn:sb}
		sb = ss^*s + s\sum_{i=1}^n f_{i,1}ss^*f_{1,i}  = ss^*s.
	\end{equation}
	Moreover, by Claim \ref{claim:ddrop1}, $s^* s^2 = s^*s f_{1,1} s = 0$, which implies $s^{*2} 
	s^2 = 0$
	and therefore
	\begin{equation} \label{eqn:6.1}
		s^2 =0.
	\end{equation}
	which allows us to compute, using \ref{item:f2} and \ref{item:f3} for the 
	second equality and using \eqref{eqn:6.1} and \ref{item:f3} for the third 
	equality:
	\[
	bs = s^*s^2 + \sum_{i=1}^n f_{i,1}ss^*f_{1,i}s = s^*sf_{1,1}s + 
	f_{1,1}ss^*f_{1,1}s = 0 + ss^*s = sb,
	. \]
	This completes the proof of the claim.
	
	It follows that $b$ endows $D$ with a $C([0,1])$-algebra structure, and the 
	map $\pi$ is a $C([0,1])$-homomorphism.
	For $t \in [0,1]$, denote by $D(t)$ and by $I_{n,n+1}(t)$ the induced 
	fibers, and by
	$\pi_t \colon D(t) \to I_{n,n+1}(t)$ the corresponding unital 
	homomorphism. 
	By \autoref{prop:CXmaps}, it
	suffices to show that $\pi_t$ is an isomorphism for all 
	$t\in [0,1]$.
	
	The fiber $I_{n,n+1}(t)$ is isomorphic to $M_n$ if $t = 0$, to $M_{n+1}$ if 
	$t = 1$, and to $M_n \otimes M_{n+1}$ otherwise.
	Since $\pi_t$ is unital for all $t\in [0,1]$, it suffices to show that the 
	fibers $D(t)$ are also isomorphic to those corresponding matrix algebras.
	Denote by $f_{j,k}(t)$, $s(t)$ and $b(t)$ the images of the elements 
	$f_{j,k}$, $s$ and $b$ in the fiber $D(t)$.
	
	\underline{Case I: $t=0$.} Here $b(0) = 0$. In particular, as $b(0) \geq 
	s^*(0) s(0)$, it follows that $s(0) = 0$.
	Therefore, $\{f_{j,k}(0)\}_{j,k=1}^n$ generates $D(0)$, and these are 
	precisely the matrix units of $M_n$. Thus $D(0) \cong M_n$.
	
	\underline{Case II: $t = 1$.} 
	Using $b(1)=1$, we compute
	\begin{eqnarray*}
		s(1)s^*(1) &=& s(1) s^*(1) b(1) \\
		&\stackrel{\mathclap{\text{\eqref{eqn:b}}}}{=}& s(1) s^*(1) s^*(1) s(1) 
		 + s(1)s^*(1) \sum_{j=1}^{n} f_{j,1}(1) s(1)s^*(1) 
		f_{1,j}(1) \\
		&\stackrel{\mathclap{\text{\ref{item:f3},\eqref{eqn:6.1}}}}{=}& \quad 0 
		+ (s(1) s^*(1))^2 +s(1)s^*(1) f_{1,1}(1) \sum_{j=2}^{n} f_{j,1}(1) 
		s(1)s^*(1) f_{1,j}(1) \\
		& \stackrel{\mathclap{\text{\ref{item:f2}}}}{=}&  (s(1) s^*(1))^2 
	\end{eqnarray*}
	Thus $s(1)s^*(1)$ is a projection, so $s(1)$ is a partial isometry. 
	Therefore, $s^*(1)s(1)$ is a projection as well, which is orthogonal to 
	$s(1)s^*(1)$ as $s^*(1)s(1) + s(1)s^*(1) \leq 1$. 
	By \ref{item:f4}, we have $\sum_{j=1}^n f_{j,j}(1) = 1- s^*s(1)$. The 
	summands on 
	the left hand side are pairwise orthogonal, and right hand side is a 
	projection. Therefore, $f_{j,j}(1)$ are pairwise orthogonal projections for 
	$j=1,2,\ldots,n$. It therefore follows that $f_{j,k}(1)$ are partial 
	isometries, which satisfy $f_{j,k}(1)f_{\ell,m}(1) = 
	\delta_{k,\ell}f_{j,m}(1)$ for all $j,k,\ell,m=1,\ldots,n$.	
	
	Thus, the family $\{ f_{1,j}(1) \}_{j=1,2,\ldots,n} \cup \{s(1)\}$ 
	generates $D(1)$, and it also satisfies the relations from Remark 
	\ref{rem:generators-Mn}, which shows that the $C^*$-algebra they generate 
	is isomorphic to $M_{n+1}$. 
	
	\underline{Case III: $t \in (0,1)$.} Here $b(t) = t1$. For 
	$j,k,\ell=1,\ldots,n$, set
	\[c_{j,k} = \frac{1}{t-t^2} s(t)f_{1,j}(t)s^*(t)f_{1,k}(t)\ \mbox{ and } \ 
	d_{\ell} = \frac{1}{\sqrt{t}(1-t)} s(t)f_{1,\ell}(t).\]
	We verify that these elements satisfy the conditions from 
	\autoref{rem:generators-Mn},
	where we regard $\{ c_{j,k} \}_{j,k \le n}$ as matrix units of the form
	$\{e_{1,i} \}_{i \le n^2}$ and each $d_\ell$ as $e_{1, n^2 + \ell}$.
	To that end, note that
	\begin{eqnarray*}
		t s(t)f_{1,1}(t)^2 s^*(t)
		&=& s(t)f_{1,1}(t) \cdot t1 \cdot f_{1,1}(t) s^*(t) \\
		&\stackrel{\mathclap{\text{\eqref{eqn:b}}}}{=}&	s(t)f_{1,1}(t) \Big( 
		s^*(t)s(t) + \sum_{j=1}^n f_{j,1}(t)s(t)s^*(t)f_{1,j}(t) \Big) 
		f_{1,1}(t)s^*(t) \\
		&\stackrel{\mathclap{\text{\ref{item:f2}}}}{=}& 	
		s(t)f_{1,1}(t)s^*(t)s(t)f_{1,1}(t)s^*(t) \\
		& \ & \ \ \ + s(t)f_{1,1}(t)^2s(t)s^*(t)f_{1,1}(t)^2s^*(t) \\
		&\stackrel{\mathclap{\text{\ref{item:f3}}}}{=}&  
		s(t)f_{1,1}(t)s^*(t)s(t)f_{1,1}(t)s^*(t) 
		+s(t)^2s^*(t)f_{1,1}(t)^2s^*(t) \\
		&\stackrel{\mathclap{\text{\eqref{eqn:6.1}}}}{=}& 
		s(t)f_{1,1}(t)s^*(t)s(t)f_{1,1}(t)s^*(t) + 0 \\
		&\stackrel{\mathclap{\text{\ref{item:f4}}}}{=}& s(t)f_{1,1}(t) \Big( 
		1-\sum_{i=1}^nf_{i,i}(t)  \Big) f_{1,1}(t)s^*(t)  \\
		&\stackrel{\mathclap{\text{\ref{item:f2}}}}{=}& s(t)\Big(f_{1,1}(t)^2 
		- f_{1,1}(t)^3 \Big) s^*(t).
	\end{eqnarray*}
	Likewise, we see that
	$$
	t s(t)f_{1,1}(t)s^*(t) = s(t) \Big(f_{1,1}(t) - f_{1,1}(t)^2 \Big) 
	s^*(t).
	$$
	Thus
	\begin{equation} \label{eqn:ff2}
		s(t)f_{1,1}^2(t)s^*(t) = (1-t)s(t)f_{1,1}(t)s^*(t),
	\end{equation}
	and therefore
	\begin{eqnarray} \label{eqn:6.2}
		s(t)\Big(f_{1,1}(t)^2 - f_{1,1}(t)^3 \Big) s^*(t) &=& 
		t(1-t)s(t)f_{1,1}(t)s^*(t) \\
		&\stackrel{\mathclap{\text{\ref{item:f3}}}}{=}& 
		t(1-t)s(t)f_{1,1}(t)s^*(t)f_{1,1}(t). \nonumber
	\end{eqnarray}
	We can now verify that the elements $c_{j,k}$ and $d_{\ell}$ from above 
	satisfy the conditions from \autoref{rem:generators-Mn}.
	For any $j,k,\ell,m=1,\ldots,n$ we have:
	\begin{eqnarray*}
		c_{j,k}c_{\ell,m}^* &=& \frac{1}{(t-t^2)^2}
		s(t)f_{1,j}(t)s^*(t)f_{1,k}(t) \cdot f_{m,1}(t)s(t)f_{\ell,1}(t) s^*(t) 
		\\
		& \stackrel{\mathclap{\text{\ref{item:f2}}}}{=}& \delta_{k,m} \frac{1}{(t-t^2)^2} s(t)f_{1,j}(t)s^*(t) \Big(f_{1,1}(t)^2 
		s(t) \Big)f_{\ell,1}(t) s^*(t) \\
		&\stackrel{\mathclap{\text{\ref{item:f3}}}}{=}& \delta_{k,m} \frac{1}{(t-t^2)^2} s(t)f_{1,j}(t)s^*(t) s(t)f_{\ell,1}(t) 
		s^*(t) \\
		&\stackrel{\mathclap{\text{\ref{item:f4}}}}{=}& \delta_{k,m} \frac{1}{(t-t^2)^2} s(t)f_{1,j}(t)\Big( 
		1-\sum_{i=1}^nf_{i,i}(t)  \Big)f_{\ell,1}(t) s^*(t) \\
		&\stackrel{\mathclap{\text{\ref{item:f2}}}}{=} &\delta_{k,m} 
		\delta_{j,\ell}  \frac{1}{(t-t^2)^2}s(t)\Big( 
		f_{1,1}(t)^2 -  f_{1,1}(t)^3 \Big) s^*(t) \\
		&\stackrel{\mathclap{\text{\eqref{eqn:6.2}}}}{=}&  \delta_{k,m} 
		\delta_{j,\ell} \frac{1}{(t-t^2)^2}\cdot  
		t(1-t)s(t)f_{1,1}(t)s^*(t)f_{1,1}(t) =  \delta_{k,m} \delta_{j,\ell} 
		c_{1,1}
	\end{eqnarray*}
	For $j,k=1,\ldots,n$, we have:
	\begin{eqnarray*}
		d_jd_k^* &=& \frac{1}{t(1-t)^2}  s(t)f_{1,j}(t) 
		\cdot  f_{k1}(t) s^*(t) \\
		&\stackrel{\mathclap{\text{\ref{item:f2}}}}{=}& \delta_{j,k} 
		\frac{1}{t(1-t)^2} s(t)f_{1,1}(t)^2  s^*(t)  \\
		&\stackrel{\mathclap{\text{\eqref{eqn:ff2}}}}{=}& \delta_{j,k}  
		\frac{1}{t(1-t)^2} \cdot (1-t)s(t)f_{1,1}(t)s^*(t) 
		\ \ \stackrel{\mathclap{\text{\ref{item:f3}}}}{=} \ \ \delta_{j,k} c_{1,1}
	\end{eqnarray*}
	Similarly, for $j,k,\ell=1,\ldots,n$ we have:
	\begin{eqnarray*}
		c_{j,k}d_\ell^* &=&   \frac{1}{t^{3/2}(1-t)^2}  
		s(t)f_{1,j}(t)s^*(t)f_{1,k}(t) \cdot  f_{\ell,1}(t) s^*(t) \\
		&\stackrel{\mathclap{\text{\ref{item:f2}}}}{=}& \delta_{k,\ell}  
		\frac{1}{t^{3/2}(1-t)^2}  s(t)f_{1,j}(t)s^*(t)f_{1,1}(t)^2 s^*(t) \\
		&\stackrel{\mathclap{\text{\ref{item:f3}}}}{=} & \delta_{k,\ell}  
		\frac{1}{t^{3/2}(1-t)^2}  s(t)f_{1,j}(t)(s^*(t))^2 
	\ \ 	\stackrel{\mathclap{\text{\eqref{eqn:6.1}}}}{=} \ \  0
	\end{eqnarray*}
	and likewise $d_\ell c_{j,k}^*=0$. By \autoref{rem:generators-Mn}, it 
	follows that
	$\{c_{j,k},d_l\}_{j,k,l=1,\ldots,n}$ generates a copy of $M_{n(n+1)}$.
	
	\begin{claim}
		The set $\{c_{j,k},d_l\colon j,k,l=1,\ldots,n\}$ also generates $D(t)$.
	\end{claim}
 It suffices to show that this family generates $s(t)$ and $\{f_{1,j}(t)\}_{ 
	j=1}^n$. Indeed, 
	\begin{eqnarray*}
		\sqrt{t} \sum_{j=1}^n c_{j,1}^*d_j 
		&\stackrel{\mathclap{\text{\ref{item:f3}}}}{=}& \frac{1}{t(1-t)^2} 
		\sum_{j=1}^n s(t)f_{j,1}(t)s^*(t)s(t)f_{1,j}(t) \\
		&\stackrel{\mathclap{\text{\ref{item:f4}}}}{=} &\frac{1}{t(1-t)^2} 
		\sum_{j=1}^n s(t)f_{j,1}(t) \Big( 1- \sum_{m=1}^n f_{m,m}(t) \Big) 
		f_{1,j}(t) \\
		&\stackrel{\mathclap{\text{\ref{item:f2}}}}{=} &\frac{1}{t(1-t)^2} 
		\sum_{j=1}^n s(t) \Big( f_{j,j}(t)^2 -  f_{j,j}(t)^3 \Big)  \\
		&\stackrel{\mathclap{\text{\ref{item:f4},\ref{item:f2}}}}{=} &\ 
		\frac{1}{t(1-t)^2} s(t) \Big( (1-s^*(t)s(t))^2 - (1-s^*(t)s(t))^3 
		\Big) \\
		&\stackrel{\mathclap{\text{\eqref{eqn:sb}}}}{=} &\frac{1}{t(1-t)^2} 
		\Big( (1-t)^2 - (1-t)^3 \Big) s(t) = s(t).
	\end{eqnarray*}
	Given $j,k=1,\ldots,n$, we want to show that $f_{j,k}(t) = \sum_{\ell=1}^n c_{\ell,j}^*c_{\ell,k} + (1-t)d_j^*d_k$.
	We have
	\begin{eqnarray*}
		\sum_{\ell=1}^n c_{\ell,j}^*c_{\ell,k} &=& 
		\frac{1}{t^2(1-t)^2} \sum_{\ell=1}^n  f_{j,1}(t) s(t) f_{\ell,1}(t) 
		s^*(t)\cdot s(t)f_{1,\ell}(t) s^*(t) f_{1,k}(t) \\
		&\stackrel{\mathclap{\text{\ref{item:f4}}}}{=}& \frac{1}{t^2(1-t)^2} 
		\sum_{\ell=1}^n  f_{j,1}(t) s(t) f_{\ell,1}(t) \Big( 1- \sum_{m=1}^n 
		f_{m,m}(t) \Big) f_{1,\ell}(t) s^*(t) f_{1,k}(t) \\
		&\stackrel{\mathclap{\text{\ref{item:f2}}}}{=}&  \frac{1}{t^2(1-t)^2} 
		\sum_{\ell=1}^n f_{j,1}(t) s(t) ( f_{\ell,\ell}(t)^2 - 
		f_{\ell,\ell}(t)^3 ) s^*(t) f_{1,k}(t) \\
		&\stackrel{\mathclap{\text{\ref{item:f4}}}}{=}& \frac{1}{t^2(1-t)^2} 
		f_{j,1}(t) s(t) ( (1-s^*(t)s(t))^2 - (1-s^*(t)s(t))^3 ) s^*(t) 
		f_{1,k}(t) \\
		&\stackrel{\mathclap{\text{\eqref{eqn:sb}}}}{=}& \frac{1}{t^2(1-t)^2} 
		\cdot t(1-t)^2  f_{j,1}(t) s(t) s^*(t) f_{1,k}(t) \\
		&=& \frac{1}{t} f_{j,1}(t) s(t) s^*(t) f_{1,k}(t),
	\end{eqnarray*}
	and
	\begin{eqnarray*} 
		(1-t)d_j^*d_k &=&  \frac{1}{t(1-t)} f_{j,1}(t)s^*(t)s(t)f_{1,k}(t) \\
		& \stackrel{\mathclap{\text{\ref{item:f2},\ref{item:f4}}}}{=}& \ \ 
		\frac{1}{t(1-t)} f_{j,1}(t)(1 -  f_{1,1}(t))f_{1,k}(t) \\
		& \stackrel{\mathclap{\text{\ref{item:f2}}}}{=} & \frac{1}{t(1-t)} 
		(f_{j,j}(t) f_{j,k}(t) - f_{j,j}(t) f_{j,1}(t) f_{1,k}(t)) \\
		& \stackrel{\mathclap{\text{\ref{item:f2},\ref{item:f4}}}}{=}&  \  \ 
		\frac{1}{t(1-t)} (1 - s^*(t) s(t)) (f_{j,k}(t)
		- f_{j,1}(t) f_{1,k}(t)).
	\end{eqnarray*}
	Moreover 
	\[
	f_{j,1} ss^* f_{1,j} f_{j,k} \ \  \stackrel{\mathclap{\text{\ref{item:f2}}}}{=}\ \ 
	f_{j,1} ss^* f_{1,1} f_{1,k}  \ \ 
	\stackrel{\mathclap{\text{\ref{item:f3}}}}{=} \ \  f_{j,1} ss^* f_{1,k},
	\]
	and
	\[
	f_{j,1} ss^* f_{1,j} f_{j,1} f_{1,k}  
	\ \ \stackrel{\mathclap{\text{\ref{item:f2}}}}{=} \ \ f_{j,1} ss^* f_{1,1}^2 f_{1,k}
	\ \ \stackrel{\mathclap{\text{\ref{item:f3}}}}{=} \ \  f_{j,1} ss^* f_{1,k}.
	\]
	Thus, by definition of $b$ and by the fact that $b(t) = t$ we obtain
	\begin{eqnarray*} 
		(1-t)d_j^*d_k &= & \frac{1}{t(1-t)} (1 - s^*(t) s(t)) (f_{j,k}(t) - 
		f_{j,1}(t) f_{1,k}(t)) \\
		& = &\frac{1}{t(1-t)}(1-b(t)) (f_{j,k}(t)
		- f_{j,1}(t) f_{1,k}(t)) \\
		& \stackrel{\mathclap{\text{\ref{item:f2}}}}{=}& \frac{1}{t} (f_{j,k}(t) 
		- f_{j,j}(t) f_{j,k}(t)) \\
		& \stackrel{\mathclap{\text{\ref{item:f2}}}}{=}& \frac{1}{t}\Big( 1- 
		\sum_{m=1}^n f_{m,m}(t)\Big) f_{j,k}(t) \\
		& \stackrel{\mathclap{\text{\ref{item:f4}}}}{=}& \frac{1}{t} s^*(t)s(t) 
		f_{j,k}(t).
	\end{eqnarray*}
	Combining those observations, we get
	\begin{eqnarray*}
		\sum_{k=1}^n c_{k,1}^*c_{k,j} + (1-t)d_1^*d_j &=& \frac{1}{t}  ( 
		f_{j,1}(t) s(t) s^*(t) f_{1,k}(t) + s^*(t)
		s(t) f_{j,k}(t) ) \\
		& \stackrel{\mathclap{\text{\ref{item:f3}}}}{=}&  \frac{1}{t}  ( 
		f_{j,1}(t) s(t) s^*(t) f_{1,1} f_{1,k}(t) + s^*(t)
		s(t) f_{j,k}(t) ) \\
		&  \stackrel{\mathclap{\text{\ref{item:f2}}}}{=}& \frac{1}{t}  ( 
		f_{j,1}(t) s(t) s^*(t) f_{1,j} f_{j,k}(t) + s^*(t)
		s(t) f_{j,k}(t) ) \\
		&  \stackrel{\mathclap{\text{\ref{item:f2}}}}{=}& \frac{1}{t} 
		\Big(s^*(t) s(t) +  \sum_{m=1}^n
		f_{m,1}(t)s(t) s^*(t) f_{1,m}(t) \Big) f_{j,k}(t) \\
		& = & \frac{1}{t} b(t) f_{j,k} (t) = f_{j,k}(t).
	\end{eqnarray*}
	This proves the claim, so have proved that $D$ is isomorphic to $I_{n,n+1}$.
	
	By universality of $I_{n,n+1}$,
	there is a unique unital homomorphism $\varphi\colon I_{n,n+1}\to B$ 
	satisfying $\varphi (f_{i,j}) = \xi(e_{i,j})$
	for all $i,j = 1, \dots, n$ and $\varphi(s) = s$. (To lighten the notation, 
	we use the same letter $s$ to denote the given element in $B$ and the 
	element $s$ in the universal \ca~ $I_{n,n+1}$.) We are left
	with checking that $\varphi$ is equivariant.

	Note that $\beta$ leaves $\varphi(I_{n,n+1})$ invariant. Furthermore, using 
	the fact that
	$v_g e_{1,1} = e_{1,1}$ and $e_{1,j} = e_{1,1} e_{1,j}$, we have 
	$v_ge_{1,j} = e_{1,j}$ for $j=1,\ldots,n$ and for all $g\in G$. It follows 
	that
	\[
	\beta_g(\varphi(b)) = s^*s + ss^* + \sum_{j=2}^n 
	\xi(v_ge_{j,1})ss^*\xi(e_{1,j}v_g^*)
	\]
	for all $g\in G$, we deduce thus that $\beta_g(\varphi(b)) = \varphi(b)$ 
	for all $g\in G$.
	Thus, the restriction of $\beta$ to $\varphi(I_{n,n+1})$ is an action via 
	$C([0,1])$-automorphisms. It follows that it suffices to check equivariance 
	on each fiber.
	
	Let $t \in (0,1)$.
	Define finite-dimensional Hilbert spaces $\mathcal{H}_0, \mathcal{H}_1$ and 
	$\mathcal{H}_2$, contained in $\varphi(I_{n,n+1})(t)$, via
	$\mathcal{H}_0 = \textrm{span}\{\varphi_t(c_{1,1})\}$,
	\[\mathcal{H}_1 = \textrm{span}\{\varphi(c_{j,k})\colon j,k 
	=1,\ldots,n\} \ \mbox{ and } \  
\mathcal{H}_2 = \textrm{span}\{\varphi(d_l)\colon l=1,\ldots,n\}.\]
	Then $\mathcal{H}_0$, $\mathcal{H}_1$ and $\mathcal{H}_2$ are invariant 
	under $\beta$.
	Set $E = \mathrm{span}\{e_{1,1},e_{1,2},\ldots,e_{1,n}\}$. Note that there 
	are natural isomorphisms $\mathcal{H}_1 \cong E \otimes E$ and 
	$\mathcal{H}_2 \cong E$, the first one given by
	identifying $e_{1,j} \otimes e_{1,k}$ with $c_{j,k}$, and the second one 
	given by identifying $e_{1,k}$ with $d_k$, for $j,k=1,\ldots,n$.
	With these identifications, $\beta_t$ acts as $v_g \otimes v_g$ on 
	$\mathcal{H}_1$ while leaving $\mathcal{H}_0$ fixed, and acts as $v_g$ on 
	$\mathcal{H}_2$.
	Thus, the action induced by $\beta$ on the fiber corresponding to some 
	$t\in (0,1)$ is conjugate to $\Ad (v \otimes (v \oplus 1_{\C}))$. The 
	end-cases $t=0$ and $t=1$ are verified
	similarly, thus concluding the proof.
\end{proof}

We will apply \autoref{thm:univprop} in the proof of \autoref{thm:mainRokDim}, 
at the end of next section,
to representations $v$ of the form $\lambda^{\otimes k}\colon G\to 
\U(\ell^2(G)^{\otimes k})$, for $k\in\N$, where $\lambda$ is the left regular 
representation.



\section{Regularity properties for actions of amenable groups} \label{s6}

This section is devoted to the proofs of Theorem \autoref{thmB} and Theorem 
\autoref{thmA}, starting with the latter.
The equivalence of item (\ref{gamma.cpou:4}) with the other items in Theorem \autoref{thmA}, which is stated under the assumption of $\mathcal{Z}$-stability, is actually
obtained in the presence of
equivariant property (SI). This property, which we recall below, is an adaptation to the 
equivariant setting of Sato's property (SI), which first appeared in 
\cite{satoRoh} and was further used in \cite{matuisato:strict}, 
\cite{matuisato:I}, \cite{matuisato:II}, \cite{sato19}.

\begin{df} \label{df:eqSI}
Let $G$ be a countable, discrete, amenable group, let $A$ be a unital, separable \ca~with non-empty trace space, and let $\alpha\colon G\to\Aut(A)$ be an action.
We say that $(A,\alpha)$ has \emph{equivariant property (SI)} if for any 
positive contractions $a, b \in (A_\U \cap A')^{\alpha_\U}$ such that $a \in J_A$ and $\sup_{m \in \N}
\| 1 - b^m \|_{2, T_\U(A)} < 1$, there is $s \in (A_\U \cap A')^{\alpha_\U}$ such that
$bs = s$ and $s^* s = a$.
\end{df}

We also need the following analogue of Kirchberg’s $\sigma$-ideal 
(\cite[Definition 1.5]{kir}) in the equivariant setting. For 
$\mathbb{Z}$-actions, this notion was  considered in \cite{liao}.

\begin{df}\label{df:EqSigmaIdeal}
Let $G$ be a discrete group, let $B$ be a \ca, let $\beta\colon G\to\Aut(B)$ be an action, and let $J\subseteq B$ be a $\beta$-invariant ideal. We say that $J$ is an \emph{equivariant $\sigma$-ideal (with respect to $\beta$)}, if for every separable, $\beta$-invariant
subalgebra $C\subseteq B$, there is a positive contraction $x\in (J\cap C')^\beta$ with $xc=c$ for all $c\in C\cap J$.
\end{df}

A $\sigma$-ideal is simply an equivariant $\sigma$-ideal with respect of the trivial action.
The trace kernel ideal $J_A$ is a $\sigma$-ideal by \cite[Proposition 
4.6]{kirchror}. 
Recall that $J_\tau$ (defined in \autoref{prelim:ultrapowers}) is the kernel of 
the quotient map $\kappa_\tau\colon A_\U 
\to M_\tau^\U$ (see \cite[Theorem 3.3]{kirchror}); it is also a $\sigma$-ideal 
(see \cite[Remark 4.7]{kirchror}).

It is easy to see that if a finite group acts on a \ca, then any $\sigma$-ideal 
is automatically an equivariant $\sigma$-ideal: one simply
averages the positive contraction in the definition of a $\sigma$-ideal to 
obtain a fixed one. When the group is amenable, one can average over F{\o}lner 
sets to get equivariant $\sigma$-ideals in ultrapowers, as we show below. 

\begin{prop}\label{prop:JAwSigmaId}
Let $A$ be a unital \ca\ with non-empty trace space, let $G$ be a countable, discrete, amenable group, let $\alpha\colon G\to\Aut(A)$ be any action.
Let $T\subseteq T(A)$ be a $G$-invariant closed subset. Then $J_T\subseteq A_\U$ is a $G$-invariant, equivariant $\sigma$-ideal in $A_\U$.
\end{prop}
\begin{proof} Abbreviate $J_T$ to $J$. It is clear
that $J_T$ is a $G$-invariant ideal in $A_\U$.
Let $C\subseteq A_\U$ be a separable, $\alpha_\U$-invariant subalgebra. Since $J$ is a $\sigma$-ideal in $A_\U$ (\cite[Proposition 4.6, Remark 4.7]{kirchror}), there is
a positive contraction $x\in J\cap C'$ with $xc=c$ for all $c\in C\cap J$. By Kirchberg's $\varepsilon$-test (\cite[Lemma A.1]{kir}),
it is enough to prove that for every finite subset $K\subseteq G$ and every $\varepsilon>0$, there is a
positive contraction $y\in J\cap C'$ with $\|(\alpha_\U)_k(y)-y\|<\ep$ for all $k\in K$ and $yc=c$ for all $c\in C\cap J$.

We fix a finite subset $K\subseteq G$ and $\ep>0$.
Using amenability of $G$, find a
finite subset $F$ of $G$ such that $|kF\triangle F|\leq \frac{\ep}{2} |F|$ for all $k\in K$.
Set $y=\frac{1}{|F|}\sum_{g\in F}(\alpha_\U)_g(x)$. Since $C$ is $\alpha_\U$-invariant, it follows that
$(\alpha_\U)_g(x) c = c$ for every $g \in G$ and $c \in C$, therefore
$yc = c$ for all $c\in C\cap J$.
For $k\in K$, we have
\begin{align*}
\|(\alpha_\U)_k(y)-y\| &=\Big\| \frac{1}{|F|} \sum_{g \in F} (\alpha_\U)_{kg}(x)-(\alpha_\U)_g(x)\Big\|\\
&=\Big\| \frac{1}{|F|} \sum_{g \in F \setminus (k F \cup k^{-1} F) } (\alpha_\U)_{kg}(x)-(\alpha_\U)_g(x)\Big\|\\
&\le \frac{1}{|F|} \sum_{g \in F \setminus (k F \cup k^{-1} F)} \left\|(\alpha_\U)_{kg}(x)-(\alpha_\U)_g(x)\right\| \\ &\leq 2 \frac{| k F \triangle F|}{| F |} \le 2 \ \frac{\ep}{2}=\ep.
\end{align*}
Since $J$ is an $\alpha_\U$-invariant ideal, the positive contraction $y$ also belongs to $J$. Finally,
it is also easy to check that $y$ commutes with $C$, since $C$ is also invariant under $\alpha_\U$. This
concludes the proof.
\end{proof}

We record the following consequence of \autoref{prop:JAwSigmaId}, which 
is used repeatedly in the 
sequel. Recall that $\kappa\colon (A_\U\cap A',\alpha_\U)\to(A^\U\cap A',\alpha^\U)$ denotes the canonical quotient map 
(\autoref{lma:SurjRelComm}).

\begin{cor}\label{cor:LiftAlongKappa}
Let $A$ be a unital separable \ca\ with non-empty trace space, let $G$ be a 
countable, discrete, amenable group, let $\alpha\colon G\to\Aut(A)$ be any 
action. Let $B$ be a seprable, nuclear \ca, let $\beta\colon
G\to\Aut(B)$ be an action, and let 
\[\Psi\colon (B,\beta)\to
(A^\U\cap A',\alpha^\U)\]
be an equivariant completely positive order zero map.
Then there exists an equivariant completely positive order zero map
$\Phi\colon (B,\beta)\to
(A_\U\cap A',\alpha_\U)$ satisfying $\Psi=\kappa\circ\Phi$.
\end{cor}
\begin{proof}
By the Choi-Effros Lifting Theorem \cite{choieffros}
there is a completely positive map
$\Phi_0\colon B \to A_\U \cap A'$ 
such that $\Psi=\kappa\circ\Phi_0$. Let $C$ be a separable, $\alpha_\U$-invariant subalgebra of $A_\U$
containing $A \cup \Phi_0(B)$. Since $J_A$ is an equivariant $\sigma$-ideal by \autoref{prop:JAwSigmaId}, there exists 
$ x \in (J_A \cap C')^{\alpha_\U}$ such that
$xc = c$ for all $c \in J_A \cap C$.
Define $\Phi\colon B\to A_\U\cap A'$ by
$\Phi(b) = (1- x) \Phi_0(b) (1-x)$ for all $b\in B$. 
One readily checks that $\Phi$ is completely positive, contractive,
order zero and equivariant, and that $\kappa\circ\Phi=\kappa\circ\Phi_0=\Psi$, 
as required.
\end{proof}

Another useful consequence of \autoref{prop:JAwSigmaId} is the following 
dynamical analogue of
\cite[Theorem 3.3]{kirchror}. We point out that
the following lemma also follows from the main
result of \cite{ForGarTho_asymptotic_2021}.

\begin{lma}\label{lma:SurjRelCommEq}
Let $A$ be a separable unital \ca, let $G$ be a discrete group, let 
$\alpha\colon
G\to\Aut(A)$ be an action and $\tau \in T(A)^\alpha$.
Then the quotient maps $\kappa: A_\U \to A^\U$ and
$\kappa_\tau: A_\U \to \M_\tau$
restrict to surjective, 
equivariant maps
\[
\kappa\colon (A_\U\cap A')^{\alpha_\U}\to (A^\U\cap A')^{\alpha^\U},
\]
\[
\kappa_\tau\colon (A_\U\cap A')^{\alpha_\U}\to (\M_\tau^\U\cap \M_\tau')^{(\alpha^\tau)^\U}.
\]
\end{lma}
\begin{proof}
We only show the statement for $\kappa: A_\U \to A^\U$, as the proof for $\kappa_\tau: A_\U
\to \M_\tau^\U$ is analogous. Given $a \in (A^\U\cap A')^{\alpha^\U}$, there is $b \in A_\U$
such that $\kappa(b) = a$. Let $C\subseteq A_\U$
be the separable \ca\ generated by $A$ and $b$.
By \autoref{prop:JAwSigmaId} there is $x \in (J_A \cap C')^{\alpha_\U}$ such that
$xc = c$ for  all $c \in J_A \cap C$. This implies that $b': = (1 - x)b$ belongs
to $(A_\U\cap A')^{\alpha_\U}$ and it satisfies $\kappa(b') = a$.
\end{proof}


\begin{thm} \label{thm:mainCPoU}
	Let $A$ be a separable, unital, nuclear,  
	\ca\
	with non-empty trace space and with no finite-dimensional
	quotients. Let $G$ be a countable, discrete, amenable group
	and let $\alpha\colon G \to \text{Aut}(A)$ be an action such that the 
	induced action
	on $T(A)$ has finite orbits bounded in size by a constant $M > 0$.
	Then the following are equivalent:
	\begin{enumerate}
		\item $(A,\alpha)$ has uniform property $\Gamma$,
		\item $(A, \alpha)$ has CPoU with constant $M$,
		\item for every $n \in \N$ there is a unital 
		embedding $M_n \to (A^\U \cap A')^{\alpha^\U}$.
            \end{enumerate}
If $A$ is moreover simple and $\mathcal{Z}$-stable, then the above are also equivalent to:
\begin{enumerate}
\setcounter{enumi}{3}
		\item $(A,\alpha)$ is cocycle conjugate to $(A 
		\otimes 
		\mathcal{Z}, \alpha \otimes \text{id}_\mathcal{Z})$.
	\end{enumerate}
\end{thm}

\begin{proof}
The equivalence of \eqref{gamma.cpou:1}, \eqref{gamma.cpou:2} 
and \eqref{gamma.cpou:3} was already proved in \autoref{thm:mainCPoU-parts-1-3}, and \eqref{gamma.cpou:4} $\Rightarrow$ 
\eqref{gamma.cpou:3} is \autoref{prop:Zstable}. 
It remains to show that \eqref{gamma.cpou:3} $\Rightarrow$ 
\eqref{gamma.cpou:4}.
Fix a unital embedding $\Psi\colon M_2\to (A^\U\cap A')^{\alpha^\U}$,
which we regard as a unital equivariant homomorphism 
$\Psi\colon (M_2,\id_{M_2})\to (A^\U\cap A',\alpha^\U)$.
By \autoref{cor:LiftAlongKappa}, there
exists an equivariant completely positive contractive order zero
map $\Phi\colon M_2  \to (A_\mathcal{U} \cap A')^{\alpha_\mathcal{U}}$
making the following diagram commute:
\begin{align*}
\xymatrix{ && (A_\U\cap A')^{\alpha_\U}\ar[d]^-{\kappa}\\
M_2\ar[urr]^-{\Phi}\ar[rr]_-{\Psi} && (A^\U\cap A')^{\alpha^\U}.}
\end{align*}
Note that $\tau(\Phi(e_{1,1})^m)=1/2$ for all $\tau\in T_\U(A)$ and for all $m\in\N$.
Set $c_1=\Phi(e_{1,1})$ and $c_2=\Phi(e_{1,2})$. Then
\[c_1\geq 0, \ \ c_2c_2^*=c_1^2, \ \mbox{ and } \ c_1c_2^*=c_2c_1^*=0.\]
By \cite[Theorem B]{equiSI}, the dynamical system $(A, \alpha)$ has equivariat
property (SI). Using this, fix a contraction
$s\in (A_\U\cap A')^{\alpha_\U}$ satisfying $s^*s=1-\sum_{j=1}^2 c_j^*c_j$
and $c_1s=s$. By 
\autoref{thm:univprop},
the elements $c_1,c_2,s$ generate a copy of the prime 
dimension drop algebra
\[
I_{2,3} = \{f\in C([0,1],M_2\otimes M_3) \colon f(0)\in M_2\otimes 1, f(1)\in 1\otimes M_3\}
\]
so
there exists a unital homomorphism $I_{2,3}\to (A_\U\cap A')^{\alpha_\U}$. By 
repeatedly using \autoref{lma:reindexation-uniform}, we can find a countable 
sequence of unital homomorphisms  $I_{2,3}\to (A_\U\cap A')^{\alpha_\U}$ with 
commuting ranges, and therefore a unital homomorphism  $I_{2,3}^{\otimes 
\infty} \to (A_\U\cap A')^{\alpha_\U}$. By \cite[Theorem 1.1]{Dadarlat-Toms}, 
$\mathcal{Z}$ embeds unitally in $I_{2,3}^{\otimes \infty}$, and in particular, 
we 
obtain a unital homomorphism $\mathcal{Z}\to (A_\U\cap A')^{\alpha_\U}$.
This implies, by \cite[Theorem 2.6]{szabo:strself}, that $\alpha$ is cocycle 
conjugate
to $\alpha\otimes\id_{\mathcal{Z}}$.\end{proof}


We now turn our attention to the proof of Theorem \autoref{thmB}. We briefly 
describe 
our strategy for proving (\ref{rokdim:1}) $\Rightarrow$ (\ref{rokdim:2}).
First, we prove the existence of projections with 
properties 
analogous
to those in \autoref{df:wtRp} for Bernoulli shifts on the hyperfinite 
II$_1$-factor $
\overline{\bigotimes}_{g \in G} \mathcal{R} \cong \mathcal{R}$. Then
\autoref{thm:AbsR} and \autoref{thm:ErgThy} are used to show that there are 
similar 
projections
also in the central sequence algebra of each fiber $\mathcal{M}_\tau$ for $\tau 
\in T(A)^\alpha$.
Using \autoref{cor:EmbedModelAction} we then perform a `local to global' 
argument via CPoU to glue these projections and obtain a Rokhlin 
tower in $A^\U \cap A'$. Finally, we exploit the fact that $J_A$
is an equivariant $\sigma$-ideal (\autoref{df:EqSigmaIdeal}) to lift 
those towers to $A_\U \cap A'$
and conclude the proof.

We start by addressing the first part, namely the existence of projections 
satisfying
conditions analogous to those in \autoref{df:wtRp} for the Bernoulli shift on
$\mathcal{R}$. 
We do this in the following proposition, using the  
tiling result for amenable groups from \cite{tilings}.

\begin{prop} \label{rem:muGRp}
	Let $G$ be a countable, discrete, amenable group, and let $\beta_G\colon 
	G\to\Aut(\mathcal{R})$
	be the Bernoulli shift acting by left multiplication on the indices of the 
	elementary tensors of $\overline{\bigotimes}_{g \in G} \mathcal{R} \cong 
	\mathcal{R}$.
	Given a finite set $K \subseteq G$ and $\delta > 0$, there are 
	$(K,\delta)$-invariant,
	finite sets $S_1, \dots, S_n \subseteq G$ and projections $p_{\ell,g} \in 
	\mathcal{R}$ for $\ell = 1,\dots, n$ and $g \in S_\ell$, such that
	\begin{enumerate}[label=(\roman*)]
		\item \label{II1:1} $(\beta_G)_{gh^{-1}}(p_{\ell, h}) = p_{\ell, g}$ 
		for all all $\ell = 1, \dots, n$ and all $g,h \in S_\ell$ ,
		\item \label{II1:2} $p_{\ell,g} p_{k,h} = 0$ for all $\ell, k = 1, 
		\dots, n$, $g \in S_\ell$, $k \in S_h$, whenever $(\ell,
		g) \not = (k,h)$,
		\item \label{II1:3} $\sum_{\ell = 1}^n \sum_{g \in S_\ell} p_{\ell, g} 
		= 1$,
		\item \label{II1:4} $\tau_{\mathcal{R}}(p_{\ell,g})$ is positive and 
		independent of $g\in S_\ell$.
	\end{enumerate}
\end{prop}
\begin{proof}
	Fix a finite set $K \subseteq G$ and $\delta > 0$.
	Fix an embedding $\theta$ of $L^\infty([0,1])$ with
	the Lebesgue measure
	in $\mathcal{R}$, and let $\tilde \beta_G$ be 
	the Bernoulli shift on $\overline{\bigotimes}_{g \in G} L^\infty([0,1])$. 
	Notice that the embedding $\theta$ naturally induces an equivariant 
	embedding
	\[
	\Theta\colon\left (\overline{\bigotimes_{g \in G}} L^\infty([0,1] ), \tilde 
	\beta_G \right) \to \left (\overline{\bigotimes_{g \in G}} \mathcal{R}, 
	\beta_G \right),
	\]
	where $\Theta$ sends, for every $g \in G$, the $g$-th coordinate of the 
	domain to the corresponding $g$-th coordinate
	of the codomain via $\theta$.
	
	Set $X = \prod_{g \in G} [0,1]$ with the product measure.
	It follows that $\overline{\bigotimes}_{g \in G} L^\infty([0,1]) \cong
	L^\infty(X)$, and that $\tilde \beta_G$ induces an action on 
	$X$ which is easily seen (and well known) to 
	be measure preserving and ergodic.
	By \cite[Theorem 3.6]{tilings} there are  $(K, \delta)$-invariant sets 
	$S_1, \dots, S_n \subseteq G$ and projections
	$\tilde p_{\ell,g} \in L^\infty(X)$ for $\ell = 1,\dots, n$ and $g \in 
	S_\ell$ such that
	\begin{enumerate}[label=(\roman*)']
		\item $(\tilde \beta_G)_{gh^{-1}}(\tilde p_{\ell, h}) = \tilde p_{\ell, 
		g}$ for all all $\ell = 1, \dots, n$ and all $g,h \in S_\ell$ ,
		\item $\tilde p_{\ell,g} \tilde p_{k,h} = 0$ for all 
		$\ell, k = 1, \dots, n$, $g \in S_\ell$, $k \in S_h$, whenever $(\ell,
		g) \not = (k,h)$,
		\item $\sum_{\ell = 1}^n \sum_{g \in S_\ell} \tilde 
		p_{\ell, g} = 1$.
	\end{enumerate}
	It is clear then that the projections
	$p_{\ell, g} = \Theta(\tilde p_{\ell,g})$ for $\ell = 1,\dots, n$ and $g 
	\in S_\ell$ satisfy conditions
	(i), (ii) and (iii) of the statement.
	
	Finally, item \ref{II1:4} is an automatic consequence of the uniqueness of 
	$\tau_{\mathcal{R}}$ as a (faithful) normal trace on
	$\mathcal{R}$, which is therefore $\beta_G$-invariant. More specifically, 
	for every $\ell = 1,\dots, n$, by condition \ref{II1:1} and 
	$\beta_G$-invariance of $\tau_{\mathcal{R}}$,
	the value $\tau_{\mathcal{R}}(p_{\ell,g})$ is independent of $g \in 
	S_\ell$. Thus,
	if $\tau_{\mathcal{R}}(p_{\ell,g}) = 0$ for some $\ell = 1,\dots, n$ and $g 
	\in S_\ell$,
	then by faithfulness it follows that $p_{\ell,h} = 0$
	for all $h \in S_\ell$. Therefore, up to discarding some of the $S_\ell$'s, 
	we obtain a family
	of projections satisfying condition \ref{II1:4}.
\end{proof}

Although (iii) implies (ii) in the statement of the above proposition, we state it explicitly 
because it 
is a condition that can be lifted along quotient maps 
via $\sigma$-ideals; see in particular the proof of
(2) $\Rightarrow$ (1) in 
\autoref{thm:mainRokDim}.

For the reader's convenience, we reproduce the statement of 
Theorem \autoref{thmB}. 
\begin{thm}  \label{thm:mainRokDim}
	Let $A$ be a separable, simple, nuclear unital \ca\ with non-empty trace 
	space, let $G$ be a countable, discrete, amenable group,
	and let $\alpha\colon G \to \Aut(A)$ be an action. 
	Suppose that the orbits of the action
	induced by $\alpha$ on $T(A)$ are finite and with uniformly bounded cardinality.
	Then the following are equivalent:
	\be\item \label{rokdim:1} $\alpha$ is strongly outer,
	\item \label{rokdim:2} $\alpha\otimes\mathrm{id}_{\mathcal{Z}}$ has the weak tracial Rokhlin property.
	\ee
When $G$ is residually finite, then the above statements are also equivalent to:
\be\setcounter{enumi}{2}
	\item \label{rokdim:3} $\alpha\otimes\mathrm{id}_{\mathcal{Z}}$ has finite 
	Rokhlin dimension,
	\item \label{rokdim:4} $\alpha\otimes\mathrm{id}_{\mathcal{Z}}$ has  
	Rokhlin dimension at most 2.
	\ee
\end{thm}


\begin{proof}
Note that $\alpha$ is stongly outer if and only if
$\alpha\otimes\mathrm{id}_{\mathcal{Z}}$ is strongly outer.
By replacing $\alpha$ with $\alpha\otimes\mathrm{id}_{\mathcal{Z}}$, we may assume that 
$(A,\alpha)\cong_{\mathrm{cc}}(A\otimes \mathcal{Z}, \alpha \otimes \text{id}_\mathcal{Z})$.

\eqref{rokdim:2} $\Rightarrow$  \eqref{rokdim:1}.
The argument is mostly standard, but we include it since we work with a slightly different notion of
the weak tracial Rokhlin property.
Fix $g\in G\setminus\{1\}$ and
$\tau \in T(A)^{\alpha_g}$. We denote as usual the von
Neumann algebra generated by $\pi_\tau(A)$ by $\mathcal{M}_\tau$,
and note that $\alpha_g$ extends to an automorphism
$\alpha_g^\tau$ of $\M_\tau$, even though $\alpha$ does not
necessarily induce
an action on $\M_\tau$.
Fix $K = \{ g \}$ and $\delta = 1/2$. 
Let $S_1,\ldots, S_n$ be a family of $(\{ g \}, 1/2)$-invariant subsets of $G$,
and
$f_{\ell,h}\in A_\U\cap A'$, for $\ell=1,\ldots,n$ and $h\in S_\ell$, be contractions as in
\autoref{df:wtRp}. Let $\kappa_\tau \colon A_\U \to \mathcal{M}_\tau^\U$ be the quotient map (see \autoref{lma:SurjRelComm}), and set
 $p_{\ell,h}=\kappa_\tau(f_{\ell,h})\in \M_\tau^\U\cap \M_\tau'$. Since
$\sum_{\ell = 1}^n \sum_{h \in S_\ell} p_{\ell, h} = 1$,
 we can assume that $\{p_{1,h} \}_{h \in S_1}$ is a collection
of non-zero, pairwise orthogonal (hence distinct) projections.
By $(\{g\},1/2)$-invariance of $S_1$, there exists $h \in S_1$ such that $gh 
\in S_1$, thus we have
\[
(\alpha^\tau)^\U_g(p_{1,h}) = (\alpha^\tau)^\U_{ghh^{-1}}(p_{1,h}) = p_{1,gh}.
\]
It follows that the automorphism 
$(\alpha^\tau)_g^\U$
acts non-trivially on the family $\{p_{1,h}\}_{ h\in S_1}$. Therefore
$(\alpha^\tau)^\U_g$ acts non-trivially on $\mathcal{M}^\U_\tau \cap \mathcal{M}_\tau'$, which in turn
implies that
$\alpha^\tau_g$ is outer, as desired. 

\eqref{rokdim:1} $\Rightarrow$  \eqref{rokdim:2}.
Assume that $\alpha$ is strongly outer.
Let $\beta_G$ be the Bernoulli shift on $\mathcal{R}$ (see \autoref{rem:muGRp}), which is an outer action. 
Since $(A,\alpha)$ is assumed to be 
cocycle conjugate to $(A \otimes \mathcal{Z}, \alpha \otimes \text{id}_{\mathcal{Z}})$, we deduce from \autoref{thm:mainCPoU} that $(A,\alpha)$ has CPoU.
By \autoref{cor:EmbedModelAction} (with $N=\{e\}$), there exists a unital, 
equivariant embedding $\Psi\colon (\mathcal{R},\beta_G)\to (A^\U\cap A',\alpha^\U)$.

Let $K \subseteq G$ be finite and let $\delta > 0$.
Using \autoref{rem:muGRp}, find $(K, \delta)$-invariant finite subsets $S_1, \dots, S_n$ of $G$ and
projections
$p_{\ell,g} \in \R$, for $\ell=1,\ldots,n$ and $g\in S_\ell$, satisfying \ref{II1:1} through \ref{II1:4} in 
\autoref{rem:muGRp}.
Denote by $\tau_{\R}$ the unique trace on $\R$.
By lifting the projections $\Psi(p_{\ell,g})$
along the surjective map 
$\kappa\colon A_\U\cap A'\to A^\U\cap A'$, 
we find positive contractions 
$e_{\ell,g}\in A_\U\cap A'$, for $\ell=1,\ldots,n$ and $g\in S_\ell$, satisfying:
\begin{enumerate}[label=(\alph*)]
\item \label{i:R1}$(\alpha_\U)_{gh^{-1}}(e_{\ell, h})-e_{\ell,g}\in J_A$ for  
$\ell=1,\ldots,n$ and for all $g,h\in S_\ell$,
\item \label{i:R2} $e_{\ell,g}e_{k,h}\in J_A$  for $\ell,k=1,\ldots,n$, and for 
all $g \in S_\ell$, $h\in S_k$, whenever $(\ell,g)\neq (k,h)$,
\item \label{i:R3} $1-\sum_{\ell=1}^n\sum_{g\in S_\ell} e_{\ell, g}\in J_A$,
\item \label{i:R4} 
$\tau(e_{\ell,g})=\tau(\Psi(p_{\ell,g}))=\tau_{\R}(p_{\ell,h})>0$ for all $\tau 
\in T_\U(A)$, for $\ell=1,\ldots,n$ and for all $g,h\in S_\ell$.
\end{enumerate}

Let $C$ be a separable $G$-invariant subalgebra of $A_\U$ containing $A$ and 
the finite set $\{e_{\ell,h}\colon \ell=1,\ldots,n, h\in S_\ell\}$.
By \autoref{prop:JAwSigmaId}
there exists a positive contraction $x\in (J_A\cap C')^{\alpha_\U}$ satisfying 
$xc=c$ for
all $c\in C\cap J_A$. 
For $\ell=1,\ldots,n$
and $g\in S_\ell$, we set $f_{\ell,g}=(1-x)e_{\ell,g}(1-x)\in A_\U\cap A'$. 
We claim that these elements satisfy the conditions of \autoref{df:wtRp}.

Condition \eqref{wtRp:3} in \autoref{df:wtRp} is satisfied by item \ref{i:R3} above and because $x
\in J_A$.
To check \eqref{wtRp:1} in \autoref{df:wtRp}, let $\ell=1,\ldots,n$ and $g,h\in S_\ell$. Observe that $(1-x)c=0$ for every $c\in C\cap J_A$, and use this, along with invariance of $x$, to get
\begin{align*}
(\alpha_\U)_{gh^{-1}}(f_{\ell, h})-f_{\ell,g}&=(\alpha_\U)_{gh^{-1}}((1-x)e_{\ell, h}(1-x))-(1-x)e_{\ell,g}(1-x)\\
&=(1-x)\left((\alpha_\U)_{gh^{-1}}(e_{\ell, h})-e_{\ell,g}\right)(1-x)
\stackrel{\ref{i:R1}}{=} 0.
\end{align*}
To check condition \eqref{wtRp:2} in \autoref{df:wtRp}, let $\ell,k=1,\ldots,n$, let $g\in S_\ell$ and let $h\in S_k$ with $(\ell,g)\neq (k,h)$.
We use $[x,e_{\ell,g}]=0$ at the second step to get
\[ f_{\ell,g}f_{k,h}=(1-x)e_{\ell,g}(1-x)^2e_{k,h}(1-x) = (1-x)e_{\ell,g}e_{k,h}(1-x)^3\stackrel{\ref{i:R2}}{=}0.\]
Finally, for item \eqref{wtRp:4} of \autoref{df:wtRp} observe that, given $\tau 
\in T_\U(A)$, for $\ell=1,\ldots,n$ and for all $g\in S_\ell$, we have
\begin{align*}\tau(f_{\ell,g})&=\tau((1-x)e_{\ell,g}(1-x))=\tau(e_{\ell,g})+\tau(e_{\ell, g}x)+\tau(xe_{\ell, g})+\tau(xe_{\ell,g}x)\\
&=\tau(e_{\ell,g}) . \end{align*}
Hence $\tau(f_{\ell,g})=\tau_{\R}(p_{\ell,g})>0$, 
and this value depends only on $\ell$ by condition \ref{II1:4} of \autoref{rem:muGRp}. 


From now on, we assume that $G$ is residually finite.

\eqref{rokdim:3} $\Rightarrow$ \eqref{rokdim:1}. 
Fix $g\in G\setminus\{1\}$ and
$\tau \in T(A)^{\alpha_g}$. Abbreviate
$\pi_\tau(A)''$ to $\mathcal{M}_\tau$, 
and let $\kappa_\tau\colon A_\U \cap A' \to \mathcal{M}^\U_\tau \cap
\mathcal{M}_\tau'$ be the equivariant quotient map. 
Since $G$ is residually finite,
there is a finite-index, normal subgroup $H \le G$ such that $g \notin H$. 
Using the fact that $d=\dimRok(\alpha)<\I$, find positive contractions 
$f_{\overline k}^{(j)} \in A_\mathcal{U} \cap A'$, for $j = 0, \ldots, d$ and 
$\overline k \in G/H$,
which satisfy, for $j=0,\ldots,d$, for all $g \in G$ and for all $\overline 
k,\overline k'\in G/H$ with $\overline k\neq \overline k'$:
\[f^{(j)}_{\overline k}f^{(j)}_{\overline k'}=0, \ \ \ \sum_{j=0}^d\sum_{\overline k\in G/H} f_{\overline k}^{(j)}=1 \ \ \mbox{ and } \  \ \ 
 (\alpha_\U)_g(f_{\overline k}^{(j)}) = f_{\overline{gk}}^{(j)} . \]
Fix $j_0=0,\ldots,d$ such that $\kappa_\tau(f_{\overline{k}}^{(j_0)})$ is not 
zero for some (and hence all) $\overline k \in G/H$. 
Note that the positive contractions $\kappa_\tau(f_{\overline{k}}^{(j_0)})$, for $\overline k \in G/H$, are thus pairwise orthogonal and satisfy 
$(\alpha^\tau)^\U_g(\kappa_\tau(f_{\overline k}^{(j)})) = \kappa_\tau(f_{\overline {gk}}^{(j)})$.
Since $g \notin H$ there is $\overline k \in G/H$ such that $\overline{gk} \not=
\overline{k}$. In particular, $(\alpha^\tau_g)^\U$
acts non-trivially on $\mathcal{M}_\tau^\U \cap \mathcal{M}'$, thus $\alpha^\tau_g$ cannot be 
an inner automorphism of $\mathcal{M}_\tau$.

\eqref{rokdim:4} $\Rightarrow$ \eqref{rokdim:3}. This implication is trivial.

\eqref{rokdim:1} $\Rightarrow$ \eqref{rokdim:4}.
Let $H$ be a normal subgroup of $G$ of finite index.
After identifying $\mathcal{R}$ with the weak closure of
$\bigotimes_{n \in \N} \mathcal{B}(\ell^2(G/H))$, denote by $\mu_{G/H}$ the 
action given by $\bigotimes_{n \in \N} \text{Ad}(\lambda_{G/H})$.
Since $(A, \alpha) \cong_{\text{cc}} (A \otimes \mathcal{Z}, \alpha \otimes \text{id}_\mathcal{Z}$),
by \autoref{thm:mainCPoU} the dynamical system $(A, \alpha)$ has CPoU.
As moreover $\alpha$ is a strongly outer action, we can apply \autoref{cor:EmbedModelAction}
to obtain an equivariant, unital homomorphism
\[
\Psi\colon (\mathcal{R}, \mu_{G/H} \circ q_H) \to (A^\mathcal{U} \cap A', \alpha^\mathcal{U}).
\]
Fix $\ep > 0$, and let $k \in \N$ be given by
\autoref{prop:GactIkdim2}. 
Denote by 
\[\varphi \colon \left(\B(\ell^2(G/H)^{\otimes k}),\Ad(\lambda_{G/H}^{\otimes k})\circ q_H\right) \to (A^\mathcal{U} \cap A', \alpha^\mathcal{U})\] 
the restriction of $\Psi$ to $\B(\ell^2(G/H)^{\otimes k})\subseteq \R$.
By \autoref{cor:LiftAlongKappa}, there exists an equivariant
completely positive contractive order zero map
\[\rho\colon \left(\B(\ell^2(G/H)^{\otimes k}),\Ad(\lambda_{G/H}^{\otimes k}) \circ q_H \right) \to (A_\mathcal{U} \cap A', \alpha_\mathcal{U})\]
such that the following diagram commutes:
\begin{align*}
\xymatrix{ && A_\mathcal{U} \cap A'\ar[d]^-{\kappa}\\
\B(\ell^2(G/H)^{\otimes k})\ar[urr]^-{\rho}\ar[rr]_-{\varphi} && A^\mathcal{U} \cap A'.}
\end{align*}

We denote by $e\in \B(\ell^2(G/H))$ the projection onto the constant functions, and regard $e^{\otimes k}$ as a projection
in $\B(\ell^2(G/H)^{\otimes k})$. One can verify that 
\[\tau(\rho(e^{\otimes k})^m)=
\tau(\varphi(e^{\otimes k})^m) = \tau(\varphi(e^{\otimes k}))=1/[G:H]^k\] for all $\tau\in T_\U(A)$ and for all $m\in\N$, using that 
$\varphi(e^{\otimes k})$ is a projection.

By \cite[Theorem B]{equiSI}, the system $(A, \alpha)$ has equivariant property 
(SI), so we can pick
a contraction
$s\in (A_\mathcal{U} \cap A')^{\alpha_\U}$ satisfying $s^*s=1-\rho(1)$
and $\rho(e^{\otimes k})s=s$. By 
\autoref{thm:univprop}, there exists
a unital equivariant homomorphism 
\[\theta\colon \big(I_{G/H}^{(k)},\mu^{(k)}_{G/H} \circ q_H\big)\to (A_\mathcal{U} \cap A',\alpha_\mathcal{U}).\]
Let $f_{\overline{g}}^{(j)}\in I_{G/H}^{(k)}$, for ${\overline{g}}\in G/H$ and 
$j=0,1,2$, be positive contractions as in the conclusion of
\autoref{prop:GactIkdim2}. Then 
the positive contractions $\theta(f_{\overline{g}}^{(j)}) \in A_\mathcal{U} \cap A'$ satisfy the conditions
of \autoref{def:frd} up to $\ep$. Since $\ep>0$ is arbitrary, the result 
follows by saturation of $A_\U$ (\autoref{saturation}).
\end{proof}

\section{Equivariant \texorpdfstring{$\mathcal{Z}$}{Z}-stability} \label{s7}
We conclude the paper with the proof of Theorem \autoref{thmC}, which improves 
the main
result of \cite{sato19} by allowing the group to act nontrivially
on $T(A)$.
In the case of integer actions, the recent paper \cite{wou} 
used the techniques in this work to remove the assumptions
that we make on the induced action $G\curvearrowright T(A)$. 

We recall the following equivalent definition of covering dimension from 
\cite{kirchwin}: a topological space $X$ has \emph{(covering) dimension} $m \in 
\N$, in symbols $\text{dim}(X) = m$, if $m$ is the minimal value such that 
every open cover $\mathcal{O}$ of $X$ has a refinement $\mathcal{O}' = 
\mathcal{O}'_0 \sqcup \cdots \sqcup
\mathcal{O}'_m$ such that the elements in $\mathcal{O}'_j$ are pairwise disjoint, for every $j = 0, \dots, m$.
If no such value exists, we say that $X$ has infinite dimension.

Under the assumptions of Theorem~\autoref{thmC}, it is proved
in \autoref{thm:mainCPoU} that $(A,\alpha)$ is
$(\mathcal{Z}, \mathrm{id}_{\mathcal{Z}})$-stable 
if and only if for all $d\geq 2$ 
there exists a unital homomorphism $M_d \to (A_\U \cap A')^{\alpha_\U}$.
The strategy to prove the existence of such homomorphisms is, once again, via a `local to global'
argument over the trace space. More specifically, in
\autoref{prop:FinDirSumRMcDuff} we use results from
\autoref{s5} to show that for every $\tau \in T(A)^\alpha$ and every $d\in\N$, there exists a unital
homomorphism $M_d \to (\M_\tau^\U \cap \M')^{(\alpha^\tau)^\U}$. Then,
in \autoref{prop:Gammabyhand}, we glue these maps together to obtain
a unital homomorphism $M_d \to (A^\U \cap A')^{\alpha^\U}$.

To perform this gluing argument, since in this case we do not assume that $(A,\alpha)$ has CPoU,
we rely 
on the techniques from \cite[Section 6, 7]{kirchror}, which make essential use of the assumption that $\partial_e T(A)$ 
is compact and finite-dimensional.
Those arguments could also be carried out using
Ozawa's work on $W^*$-bundles in \cite{ozawa_approx} (see for instance \cite{liao},
\cite{liaoZm}), whose
role in our setting is replaced by $A^\U$.
(One can show that $A^\U$ is isomorphic to the ultrapower of the $W^*$-bundle 
associated to the tracial
completion of $A$ when $\partial_e T(A)$ is
compact.)

The main `local to global'
argument is contained
in \autoref{prop:Gammabyhand}; it is an equivariant analogue of 
\cite[Proposition 7.4, Lemma 7.5]{kirchror}. Related arguments in the 
equivariant setting, under the additional assumption that $\alpha^*$ is trivial
on $\partial_e T(A)$, can be found in \cite{sato19}, \cite{liao}, \cite{liaoZm}.

We start with a preliminary lemma.

\begin{lma} \label{lma:taualpha}
Let $A$ be a unital, separable \ca\ such that
$\partial_e T(A)$ is compact.
Let $G$ be a countable, discrete group,  
and let $\alpha\colon G\to\Aut(A)$ be an action. Suppose that
the orbits of the induced action $\alpha^*$ of $G$ on $\partial_eT(A)$ are
finite with bounded cardinality, and that the orbit space $\partial_eT(A)/G$
is Hausdorff. Then the averaging function 
$\partial_eT(A)\to T(A)$ mapping $\tau \in \partial_e T(A)$ to
$\tau^\alpha = \frac{1}{| G \cdot \tau | }\sum_{\sigma \in G \cdot \tau} \sigma$ is continuous.
\end{lma}
\begin{proof}
Since $A$ is separable, both $\partial_e T(A)$ and $T(A)$ are metrizable. It is therefore enough
to show that if
$(\tau_n)_{n\in\N}$ is a sequence in $\partial_eT(A)$ such that
$\tau_n \to \tau$ in $\partial_eT(A)$,
then $\tau_n^\alpha \to \tau^\alpha$ in $T(A)$. Fix $\varepsilon > 0$ and $a_1,\dots, a_k \in A$.

Since the orbits of the $G$-action on $\partial_eT(A)$ have bounded
cardinality, upon passing to a subsequence 
we can assume without loss of generality that there is $N \in \N$ such that $| 
G \cdot \tau_n | = N$ for every
$n \in \N$. Let $G_\tau$ be the stabilizer of $\tau$ and let $S = \{ s_1, \dots, s_M \}$ be a set
of representatives of left cosets of $G_\tau$ with $s_1 = e$. In particular $| G \cdot \tau | = | S | = M$.
Let $V$ be an open subset of $\partial_e T(A)$ such that
\begin{enumerate}[label=(\roman*)]
\item $\tau \in V$,
\item \label{i:V2} $\alpha^*_{s_j}(V) \cap V = \emptyset$ for every $1 < j \le M$,
\item \label{i:V3} $| \sigma(\alpha_{s_j}(a_i)) -
\tau(\alpha_{s_j}(a_i)) | < \varepsilon$, for all $\sigma\in V$, all $j \le M$ and all $i \le k$.
\end{enumerate}
\begin{claim}
There exists $m \in \N$ such that $G \cdot \tau_n \subseteq \bigcup_{j=1}^M \alpha^*_{s_j}(V)$
for every $n > m$.
\end{claim}
Suppose this is not the case. Up to taking a subsequence, for every $n \in \N$, let $\sigma_n \in G \cdot \tau_n$ such that
$\sigma_n  \notin \bigcup_{j=1}^M \alpha^*_{s_j}(V)$. By compactness of $\partial_e T(A)$, up
to taking a subsequence, we can assume that $\sigma_n$ converges to some $\sigma \in
\partial_e T(A)$. Since $\partial_e T(A) /G$
is Hausdorff, it follows that $\sigma \in G \cdot \tau \subseteq \bigcup_{j=1}^M \alpha^*_{s_j}(V)$,
which is a contradiction. This concludes the proof of the claim.

Given $n > m$, let $\{\sigma_{1, n} \dots, \sigma_{K, n} \}$ be the intersection
$G \cdot \tau_n \cap V$. Therefore, as each $\alpha^*_{s_j}$ is a homeomorphism,
it follows that
$
G\cdot \tau_n = \bigcup_{j=1}^M \alpha^*_{s_j}(G \cdot \tau_n \cap V),
$
and that in particular $N = M \cdot K$, so $K$ does not depend on $n > m$. Finally,
for all $i \le k$ and every $n> m$
big enough so that $\tau_n \in V$, we conclude that
\begin{align*}
\lvert \tau_n^\alpha(a_i) - \tau^\alpha(a_i) \rvert & = \Big\lvert \frac{1}{M \cdot K} \sum_{j=1}^M \sum_{h = 1}^K
\sigma_{h,n}(\alpha_{s_j}(a_i)) - \frac{1}{M} \sum_{j=1}^M \tau(\alpha_{s_j}(a_i)) \Big\rvert \\
& = \frac{1}{M}  \Big\lvert \frac{1}{K} \sum_{j=1}^M \sum_{h = 1}^K  \sigma_{h,n}(\alpha_{s_j}(a_i))
- \frac{1}{K} \sum_{j=1}^M \sum_{h = 1}^K \tau(\alpha_{s_j}(a_i)) \Big\rvert \\
& \le \frac{1}{M \cdot K}  \sum_{j=1}^M \sum_{h = 1}^K \lvert \sigma_{h,n}(\alpha_{s_j}(a_i))  -
\tau(\alpha_{s_j}(a_i)) \rvert 
\ \stackrel{\mathclap{\text{\ref{i:V3}}}}{<}\ \varepsilon.\qedhere
\end{align*}
\end{proof}

We fix some notation for the next couple of propositions.

\begin{nota} \label{remark:traceaction}
Let $A$ be a unital \ca\ such that
$\partial_e T(A)$ is compact, let $G$ be a countable group,
and let $\alpha\colon G\to\Aut(A)$ be an action.
Recall that $\alpha^*$ denotes the affine action of $G$ on $T(A)$ induced by $\alpha$; see \autoref{ss2.1}.
Let $\alpha^{**}\colon G\to\Aut(C(\partial_e T(A)))$ be 
defined as $\alpha^{**}_g(f) = f \circ \alpha^*_{g^{-1}}$ for every $g \in G$ and all $f\in C(\partial_eT(A))$. 
Note that $\alpha^{\ast\ast}$ is the restriction of 
the double-dual action to $C(\partial_e(T(A)))\subseteq A^{\ast\ast}$, thus justifying the notation.

Let $\mathcal{T}\colon A \to C(\partial_e T(A))$ be the unital,
completely positive map defined as $\mathcal{T}(a)(\tau) =
\tau(a)$ for all $a\in A$ and all $\tau\in T(A)$. 
Then $\mathcal{T}\colon (A, \alpha)\to (C(\partial_e T(A), \alpha^{**})$ is equivariant, since for $g\in G$, $a\in A$ and $\tau\in\partial_eT(A)$ we have:
\[
\mathcal{T}(\alpha_g(a))(\tau) = \tau(\alpha_g(a)) = 
\mathcal{T}(a)(\alpha^*_{g^{-1}}(\tau)) = \alpha^{**}_g(\mathcal{T}(a))(\tau).
\]
\end{nota}

The following proposition is an equivariant version of
\cite[Corollary 6.8]{kirchror}.

\begin{prop} \label{prop:approximation}
Let $A$ be a separable, simple, nuclear, infinite-dimensional, \uca. Let $G$ be a countable, discrete, amenable group,  
and let $\alpha\colon G\to\Aut(A)$ be an action.
Suppose that $\partial_eT(A)$ is compact and nonempty.

Let
$\{f_k^{(j)}\colon 1 \le j \le m, 1 \le k \le s_j\}$  be an 
$\alpha^{**}$-invariant 
partition of unity in $C(\partial_e T(A))$ 
such that $f_k^{(j)} f_h^{(j)} =0$ for
all $j=1,\dots, m$ and all $k,h = 1,\dots, s_j$ with $k \not= h$.
Given $\varepsilon > 0$, a compact subset $\Omega \subseteq A$ and a finite subset
$G_0 \subseteq G$, there exist positive contractions $c_k^{(j)}\in A$, for 
$j=1,\dots,m,$ and for $k=1,\dots, s_j$, such that
\begin{enumerate}
\item $\| c_k^{(j)} b - b c_k^{(j)} \| < \varepsilon$ for all $b \in \Omega$,
\item $\sup_{\tau \in \partial_e T(A)} \lvert f_k^{(j)}(\tau) - \tau(c_k^{(j)}) \rvert < \varepsilon$,
\item $c_k^{(j)} c_h^{(j)} = 0$ for all $h = 1, \dots, s_j$ such that $h \not = k$,
\item $\big\| \alpha_g(c_k^{(j)}) - c_k^{(j)} \big\| < \varepsilon$ for all $g \in G_0$,
\item $\sum_{j=0}^m \sum_{k = 1}^{s_j} (c_k^{(j)})^2 \le 1 + \varepsilon_0$.
\end{enumerate}
\end{prop}
\begin{proof}
Let $\mathcal{T}\colon A \to C(\partial_e T(A))$ be the unital, completely positive equivariant
map from \autoref{remark:traceaction}.
By \cite[Proposition 6.6]{kirchror}, when $\partial_e T(A)$ is closed, the induced map
$\mathcal{T}_\U \colon A_\U \to C(\partial_e T(A))_\U$ maps the unit ball of $A_\U$ onto
the unit ball of $C(\partial_e T(A))_\U$. Moreover,
denoting by 
$\text{Mult}(\mathcal{T}_\U)$ the multiplicative domain of $\mathcal{T}_\U$, we have
\[J_A \subseteq \text{Mult}(\mathcal{T}_\U) \subseteq J_A +
(A_\U \cap A'),\] 
and the restriction $\mathcal{T}_\U\colon\text{Mult}(\mathcal{T}_\U) \to C(\partial_e T(A))_\U$ is a surjective homomorphism whose kernel is $J_A$ (\cite[Proposition 6.6.vi]{kirchror}).
Moreover, the unital completely positive map 
$\mathcal{T}_\U\colon (A_\U, \alpha_\U)\to (C(\partial_e T(A))_\U, \alpha^{**}_\U)$ is equivariant by
\autoref{remark:traceaction}.

Let $C$ be the ($\alpha^{**}_\U$-invariant) $C^*$-subalgebra of $C(\partial_e T(A))_\U$ 
generated by the family
$\{f_k^{(j)}\}_{1 \le j \le m, 1 \le k \le s_j}$.
By the previous observations, there is an injective homomorphism $\Phi\colon C \to (A^\U \cap A')^{\alpha^\U}$ satisfying
$\tau(\Phi(f)) = f(\tau)$
for $\tau \in \partial_e T(A)$ and $f \in C$.
Regarding $\Phi$ as an equivariant map $(C,\id_C)\to (A^\U \cap A',\alpha^\U)$, by \autoref{cor:LiftAlongKappa} there exists a completely positive, contractive,
order zero map $\Psi$
such that the following diagram commutes:
\begin{align*}
\xymatrix{ && (A_\U\cap A')^{\alpha_\U}\ar[d]^-{\kappa}\\
C\ar[urr]^-{\Psi}\ar[rr]_-{\Phi} && (A^\U\cap A')^{\alpha^\U}.}
\end{align*}
The positive contractions $\Psi(f_k^{(j)})$, for $1 \le j \le m$ and $1 \le k \le s_j$, are pairwise orthogonal, and hence can be 
lifted to pairwise orthogonal positive contractions
$(d_{k,n}^{(j)})_{n \in \N} \in \ell^\I(A)$, for 
$1 \le j \le m$ and $1 \le k \le s_j$.
We can thus find $n \in \N$ big enough so that setting $c_k^{(j)} = d_{k,n}^{(j)}$,
for all $j=1,\dots, m$ and $k=1, \dots, s_j$, gives the required elements.
\end{proof}

The following equivariant analogue of \cite[Proposition 7.4, Lemma 7.5]{kirchror}
extends
\cite[Corollary 3.2]{liaoZm} and \cite[Proposition 3.3]{liao} to all amenable 
groups, while at the 
same time relaxing the assumptions on the induced action
$G\curvearrowright T(A)$.
\begin{prop} \label{prop:Gammabyhand}
Let $A$ be a separable, simple, nuclear, infinite-dimensional, \uca,
let $G$ be a countable, discrete, amenable group,  
and let $\alpha\colon G\to\Aut(A)$ be an action.
Suppose that $\partial_eT(A)$ is compact, that $\dim(\partial_eT(A))<\I$,
that the orbits of the induced action of $G$ on $\partial_eT(A)$ are
finite with bounded cardinality, and that the orbit space $\partial_eT(A)/G$
is Hausdorff.
Then, for every $d \in \N$, there exists a unital homomorphism $\psi\colon M_d \to (A^\U \cap A')^{\alpha^\U}$.
\end{prop}
\begin{proof}

Fix $d\ge 1$ and
let $\tau \in T(A)^\alpha$. By the assumptions on $A$, the von Neumann algebra $\M_\tau$
is a separably representable, hyperfinite, and type II$_1$, thus 
by \autoref{prop:FinDirSumRMcDuff} there is a unital
homomorphism
\[
\theta_\tau\colon M_d \to (\M^\U_\tau \cap \M')^{(\alpha^\tau)^\U}.
\]
The rest of the proof is divided into two claims,
respectively inspired by Proposition~7.4 and Lemma~7.5 in~\cite{kirchror}. 
\begin{claim} \label{claim:approx1}
Let $\ep>0$, let $m\in\N$, let $G_0\subseteq G$ be finite and let
$\Omega \subseteq A$ be compact. Then there exist an open cover $\mathcal{O}$ of $\partial_eT(A)$,
and families $\Phi^{(j)}=\{\varphi_1^{(j)},\ldots,\varphi_{r_j}^{(j)}\}$, for $j=1,\ldots,m$, consisting of unital completely
positive contractive maps $\varphi_k^{(j)}\colon M_d\to A$, for 
$k=1,\ldots,r_j$, such that all open sets $V \in \mathcal{O}$ are 
$\alpha^*$-invariant, and for every
$j=1,\ldots,m$, and $k \in \{1,\ldots,r_j \}$ we have
\begin{enumerate}[label=(1.\alph*)]
\item \label{item:1a} $\big\| \varphi_k^{(j)}(a)b-b\varphi_k^{(j)}(a) 
\big\|<\ep\|a\|$ for all $a\in M_d$ and for all $b\in \Omega$,
\item $\big\| \varphi_k^{(j)}(a)\varphi^{(i)}_h(b)-\varphi^{(i)}_h(b)\varphi_k^{(j)}(a) \big\|<\ep\|a\|\|b\|$ for all $i=1,\ldots,m$ with $i \neq j$, all $h=1,\ldots, r_i$, and for all $a,b\in M_d$,
\item \label{item:1d} $\big\| \alpha_g(\varphi_k^{(j)}(a))-\varphi_k^{(j)}(a) \big\|<\ep\|a\|$ for all $g\in G_0$,
and for all $a\in M_d$,
\item \label{item:1c} for every $V \in \mathcal{O}$ and every $j=1,\ldots,m$ there is
$k \in \{1,\ldots,r_j \}$ such that
\[
\sup_{\tau\in V}\Big\| \varphi_k^{(j)}(a^*a)-\varphi^{(j)}_k(a)^*\varphi_k^{(j)}(a) \Big\|_{2,\tau}<\ep\|a\|^2,
\]
for all $a\in M_d$.
\end{enumerate}
We call the tuple $(\mathcal{O}; \Phi^{(1)},\ldots,\Phi^{(m)})$ an \emph{$(\varepsilon, \Omega, G_0)$-commuting covering system}, in analogy with Definition~7.1 in~\cite{kirchror}.
\end{claim}
To prove the claim, fix $\Omega \subseteq A$ and $G_0 \subseteq G$ as in the statement.
Let $\tau \in \partial_e T(A)$ and set
\[
\tau^\alpha = \frac{1}{|G\cdot \tau|} \sum_{\sigma \in G \cdot \tau} \sigma \in T(A)^\alpha.
\]
By \autoref{lma:SurjRelCommEq} and by the the Choi-Effros Lifting Theorem 
\cite{choieffros},
there exists a unital, completely positive, contractive map $\tilde \theta$
such that the following diagram commutes:
\begin{align*}
\xymatrix{ && (A_\U\cap A')^{\alpha_\U}\ar[d]^-{\kappa_\tau}\\
M_d\ar[urr]^-{\tilde \theta}\ar[rr]_-{\theta_{\tau^{\alpha}}} && 
(\M_{\tau^\alpha}^\U\cap \M_{\tau^\alpha}')^{(\alpha^{\tau})^\U}.}
\end{align*}

Applying the Choi-Effros Lifting Theorem again, we find unital, completely positive maps $\tilde \theta_n \colon
M_d \to A$, for $n\in\N$, such that $(\widetilde{\theta}_n)_{n\in\N}\colon M_d\to \ell^{\infty}(A)$ lifts $\tilde \theta$.

Let $M$ be an upper bound for the cardinality of the orbits of the $G$-action $\alpha^*$
on $\partial_e T(A)$.
By choosing a map far enough in the sequence, 
we claim that there exists a unital completely positive
map $\varphi_\tau\colon M_d\to A$ satisfying
\begin{enumerate}[label=(\alph*)]
\item \label{item:local1} $\|\varphi_\tau(a)b-b\varphi_\tau(a)\|<\ep
\|a\|$ for all $a\in M_d$ and all $b\in F$,
\item \label{item:local2} $\|\alpha_g(\varphi_\tau(a))-\varphi_\tau(a)\|<\ep\|a\|$ for all $g\in G_0$ and all $a\in M_d$,
\item \label{item:local3} $\|\varphi_\tau(a^*a)-\varphi_\tau(a)^*\varphi_\tau(a)\|_{2,\tau^\alpha}<\ep\|a\|^2 M^{1/2}$ for all $a\in M_d$.
\end{enumerate}

Conditions \ref{item:local1} and \ref{item:local2} are immediate, since $G_0$ and $F$ are finite, and 
the unit ball of $M_d$ is compact.
To justify
condition \ref{item:local3}, note that $\tilde \theta(a^*a) - \tilde \theta(a)^* \tilde \theta(a)$
belongs to $J_{\tau^\alpha}$ for all $a \in M_d$, as $\tilde \theta$ is itself a lift of the homomorphism
$\theta_{\tau^\alpha}$.

By compactness of the unit ball of $M_d$ and \autoref{lma:taualpha},
we can find an
open set $V'_\tau$ of $\partial_e T(A)$ containing $\tau$
such the inequality in \ref{item:local3} holds also when substituting $\| \cdot \|_{2, \tau^\alpha}$
with $\|\cdot\|_{2,\sigma^\alpha}$ for all $\sigma\in V'_\tau$.
Set $V_\tau = \bigcup_{ g \in G} \alpha^*_g(V'_\tau)$, and note that again for every
$\sigma \in V_\tau$ the inequalities in item \ref{item:local3} hold
with respect to the $\|\cdot\|_{2,\sigma^\alpha}$-norm. Arguing as in \autoref{prop:ideal},
we see that $\sigma \le M\sigma^\alpha$ for every $\sigma \in \partial_eT(A)$,
which in turn implies that
\[
\|\varphi_\tau(a^*a)-\varphi_\tau(a)^*\varphi_\tau(a)\|_{2,\sigma}<\ep\|a\|^2,
\]
for all $a\in M_d$ and all $\sigma \in V_\tau$.

We finish the proof of the claim by induction on $m$. When $m=1$, we cover $\partial_e T(A)$ by the open sets $V_\tau$ obtained
in the previous paragraph and, by compactness of $\partial_ eT(A)$, we can find
an integer $r_1\in\N$ and traces $\tau_1,\ldots,\tau_{r_1}\in \partial_e T(A)$, such that
$\mathcal{O}=\{V_{\tau_1},\ldots,V_{\tau_{r_1}}\}$ is a cover of $\partial_e T(A)$.
Then $\mathcal{O}$ and
$\Phi^{(1)}=\{\varphi_{\tau_1},\ldots,\varphi_{\tau_{r_1}}\}$ satisfy the desired properties.

Assume that we have found an open, $\alpha^*$-invariant cover $\mathcal{O}'$ and families $\Phi^{(j)}$ for $j=1,\ldots,m-1$,
satisfying the conditions in the statement.
Let $M_d^1$ denote the unit ball of $M_d$. 
For every $\tau \in \partial_e T(A)$, find a map $\varphi_\tau\colon M_d\to A$ and an open, $\alpha^*$-invariant neighborhood $W_\tau$ of $\tau$ as in the first part of the proof
for the finite set 
\[
\Omega\cup \bigcup\limits_{j=1}^{m-1}\bigcup\limits_{k=1}^{r_j}\varphi^{(j)}_k(M_d^1)\subseteq A
\]
instead of $\Omega$.
Find an integer $r_m \in\N$ and traces $\tau^{(m)}_1,\ldots,\tau^{(m)}_{r_m}\in \partial_e T(A)$, such that
$\{W_{\tau^{(m)}_1},\ldots,W_{\tau^{(m)}_{r_m}}\}$ is an $\alpha^*$-invariant cover of $\partial_e T(A)$. Let $\mathcal{O}$ be the
family of $\alpha^*$-invariant open sets of the form $V\cap W_{\tau^{(m)}_k}$, for $k=1,\ldots, r_m$ and $V \in \mathcal{O}'$,
and set
$\Phi^{(m)}:=\{\varphi_{\tau^{(m)}_1},\ldots,\varphi_{\tau_{r_m}^{(m)}}\}$. Then
$\mathcal{O}$ and $\{ \Phi^{(j)}\}_{1 \le j \le m}$ satisfy the desired properties.

\begin{claim} \label{claim:glue2}
Let $\ep>0$, let $m=\dim(\partial_e T(A))$, let $G_0\subseteq G$ be finite and let
$\Omega \subseteq A$ be compact. 
Set $m=\dim(\partial_eT(A))$.
There exist
completely positive contractive maps $\psi^{(0)},\ldots,\psi^{(m)}\colon M_d\to A$ satisfying
\begin{enumerate}[label=(2.\alph*)]
\item \label{item:2a} $\|\psi^{(j)}(a)b-b\psi^{(j)}(a)\|<\ep$ for all $j=0,\ldots,m$, for all $a\in M_d^1$ and all $b\in \Omega$,
\item \label{item:2b} $\|\psi^{(j)}(a)\psi^{(k)}(b)-\psi^{(k)}(b)\psi^{(j)}(a)\|<\ep$ for all $j,k=0,\ldots,m$ with $j\neq k$, and for all $a,b\in M_d^1$,
\item \label{item:2d} $\|\alpha_g(\psi^{(j)}(a))-\psi^{(j)}(a) \|<\ep$ for all $g\in G_0$, for all $j=0,\ldots,m$, and for all $a\in M_d^1$,
\item \label{item:2c} $\|\psi^{(j)}(1)\psi^{(j)}(a^*a)-\psi^{(j)}(a)^*\psi^{(j)}(a)\|_{2,u}<\ep$ for all $j=0,\ldots,m$, and for all $a\in M_d^1$,
\item \label{item:2e} $\sup_{\tau \in T(A)} \tau \big ( \sum_{j=0}^m \psi^{(j)}(1) - 1 \big ) < \varepsilon$.
\end{enumerate}
\end{claim}
Let $(\mathcal{O}; \Phi^{(1)},\ldots,\Phi^{(m)})$ be an $(\varepsilon/2, \Omega, G_0)$-commuting covering system given by Claim \ref{claim:approx1}.
Since the orbit space $\partial_e T(A) /G$ is Hausdorff (in 
addition to compact and second countable),
it is metrizable.
Moreover since the quotient map  $\pi\colon \partial_eT(A) \to 
\partial_eT(A)/G$ is 
open and all orbits are finite, $\pi$
preserves the topological dimension (\cite[Proposition 2.16]{pears}).
In particular, $\dim(\partial_eT(A)/G)=\dim(\partial
_eT(A))=m$. 
Find a refinement $\mathcal{P}$ of $\pi(\mathcal{O})$ witnessing $\text{dim}(\partial_eT(A)/G) = m$,
and let $\mathcal{O}'$ be the collection of preimages of the elements of $\mathcal{P}$.
Every open set in $\mathcal{O}'$ is then $\alpha^*$-invariant, being a preimage of a set via the map $\pi$.
Since the preimages of disjoint sets are themselves disjoint, we can decompose $\mathcal{O}'$
into finite subsets $\mathcal{O}'_0,\ldots, \mathcal{O}'_m$ such that
$\mathcal{O}'_j=\{V_1^{(j)},\ldots, V_{s_j}^{(j)}\}$ consists of pairwise disjoint
open subsets of $\partial_e T(A)$ which are moreover
$\alpha^*$-invariant. It is clear that $\mathcal{O}'$ is a refinement of $\mathcal{O}$, since
the latter is composed of invariant open sets.

Let $\{\tilde f_k^{(j)}\colon j=0,\ldots,m, k=1,\ldots, s_j\} \subseteq C(\partial_e T(A))/G$ be a partition of unity
of $\partial_e T(A)/G$ with $\supp(\tilde f_k^{(j)})\subseteq\pi( V_k^{(j)})$ for
$j=0,\ldots,m$ and $k=1,\ldots,s_j$. For every $j=0,\ldots,m$ and $k=1,\ldots, s_j$, define
\[
f_k^{(j)}= \tilde f_k^{(j)} \circ \pi \in C(\partial_eT(A)).
\]
It is immediate that each $f_k^{(j)}$ is $\alpha^{**}$-invariant. Moreover, the family $\{ f_k^{(j)}\colon j=0,\ldots,m, k=1,\ldots, s_j\} \subseteq C(\partial_e T(A))$
is also a partition of unity of $\partial_e T(A)$ with $\supp( f_k^{(j)})\subseteq V_k^{(j)} = \pi^{-1}(
\pi(V_k^{(j)}))$ for
$j=0,\ldots,m$ and $k=1,\ldots,s_j$, since all elements in $\mathcal{O}'$ are $\alpha^*$-invariant.
For a fixed $j$, the functions $ f_1^{(j)},\ldots, f_{s_j}^{(j)}$ are
pairwise orthogonal, since they have disjoint supports.

For every $j\in \{0,\ldots,m\}$ and $k\in \{1,\ldots,s_j\}$,
there is a unital completely positive map $\varphi_k^{(j)}\colon M_d\to A$ belonging to $\Phi^{(j)}$ such that conditions \ref{item:1a}-\ref{item:1c} in Claim \ref{claim:approx1} hold. Set
\begin{equation} \label{epsilon0}
K = \max \Big\{ 8 \cdot \max_{0 \le j\le m} s_j^2,  4 \cdot \sum_{j=0}^m s_j 
\Big\} \ \ ,  \ \ \varepsilon_0 = \frac{\varepsilon^2}{2K},
\end{equation} 
and \[
\Omega_0 = \Omega \cup \{ \varphi_k^{(j)}(a) \colon a \in M_d^1, 0 \le j \le m, 1 \le k \le s_j \}.
\]
By \autoref{prop:approximation}, there are positive contractions $ c_k^{(j)}\in A$, for $j = 0, \dots, m$ and $k=1, \dots,
s_j$, satisfying the following for all $j= 0,\dots m$ and $k=1, \dots,s_j$:
\begin{enumerate}[label=(\roman*)]
\item \label{item:7.5i}$\| b c_k^{(j)} - b c_k^{(j)} \| < \varepsilon_0$ for all for all $b \in \Omega_0$,
\item $\| c_k^{(j)}  c_\ell ^{(i)} - c_\ell ^{(i)}c_k^{(j)} \| < \varepsilon_0$ for all $i = 0,\dots, m$ and all $\ell  = 1,\dots, s_i$.
\item \label{funct:i3} $\sup_{\tau \in \partial_e T(A)} \lvert f_k^{(j)}(\tau) - \tau(c_k^{(j)}) \rvert < \varepsilon_0^{1/2}$,
\item \label{item:7.5iv} $c_k^{(j)} c^{(j)}_\ell  = 0$ for all $\ell  = 1,\dots, s_j$ such that $\ell  \not= k$,
\item \label{item:7.5v} $\| \alpha_g(c^{(j)}_k) - c^{(j)}_k \| < \varepsilon_0$ for all $g \in G_0$.
\item \label{item:7.5vi} $\sum_{j=0}^m \sum_{k = 1}^{s_j} (c_k^{(j)})^2 \le 1 + \varepsilon_0$.
\end{enumerate}


Fix $j \in \{0,\dots, m \}$, and define a linear map $\psi^{(j)}\colon M_d\to A$ by
\[
\psi^{(j)}(a)=\sum_{k=1}^{s_j} c_k^{(j)}\varphi^{(j)}_k(a) c_k^{(j)}
\]
for all $a\in M_d$. As all $c_k^{(j)}$ are mutually orthogonal positive contractions,
it follows that $\psi^{(j)}$ is completely positive and contractive.

It remains to show that $\psi^{(0)},\ldots,\psi^{(m)}$ satisfy the conditions of the
claim. The verification of conditions \ref{item:2a} and \ref{item:2b} is the same as verification
of conditions (i) and (iii) in the proof of \cite[Lemma 7.5]{kirchror}.
To verify condition \ref{item:2d},
fix $g\in G_0$, an index $j\in \{0,\ldots,m\}$, and a contraction $a\in M_d$.
We have
\begin{align*}
\|\alpha_g(\psi^{(j)}(a))-\psi^{(j)}(a)\| 
\ &\stackrel{\mathclap{\text{\ref{item:7.5v}}}}{\le} \
\left\|\sum_{k=1}^{s_j} c_k^{(j)}\left(\alpha_g(\varphi^{(j)}_k(a))-\varphi^{(j)}_k(a)\right)c_k^{(j)}\right\| +
2s_j \varepsilon_0 \\
 & \stackrel{\mathclap{\text{\ref{item:7.5iv}}}}{\le} \ \max_{1\le k \le s_j}
\left\| c_k^{(j)}\left(\alpha_g(\varphi^{(j)}_k(a))-\varphi^{(j)}_k(a)\right)c_k^{(j)}\right\| + \frac{\varepsilon}{2} \\
 &\le \ \max_{1\le k \le s_j} \left\|\alpha_g(\varphi^{(j)}_k(a))-\varphi^{(j)}_k(a)  \right\| + \frac{\varepsilon}{2}\stackrel{\text{\ref{item:1d}}}{<}\ep.
\end{align*}

Item \ref{item:2c} corresponds to condition (iv) in \cite[Lemma 7.5]{kirchror}, and can be
inferred as follows. 
Fix $j=0,\ldots,m$ and $a \in M_d^1$. For $k=1,\ldots,s_j$, 
set 
\[
T_k= (c_k^{(j)})^2 \varphi_k^{(j)}(a^*a)
-\varphi_k^{(j)}(a)^*(c_k^{(j)})^2\varphi_k^{(j)}(a),\]
and note that $\|T_k\|\leq 2$.
Fix $\tau \in \partial_e T(A)$. Then
\begin{align*}
 \| \psi^{(j)}(1)\psi^{(j)}(a^*a)-\psi^{(j)}(a)^*\psi^{(j)}(a)\|_{2, \tau} \
 &\stackrel{\mathclap{\text{\ref{item:7.5iv}}}}{=}  \ \Big\| \sum_{k=1}^{s_j}c_k^{(j)} T_k c_k^{(j)} \Big\|_{2,\tau}\\
 &\stackrel{\mathclap{\text{\ref{item:7.5v}}}}{=} \
 \max_{k=1,\ldots,s_j}\big\|c_k^{(j)} T_k c_k^{(j)} \big\|_{2,\tau}.
\end{align*}
By the above computation and since $T(A)$ is convex, 
it is enough to prove
that $\big\|c_k^{(j)} T_k c_k^{(j)} \big\|_{2,\tau}<\varepsilon$ for every $k=1,\ldots,s_j$,
Fix $k=1,\ldots,s_j$.
If $\tau \notin V_k^{(j)}$, then
\[
\big\|
c_k^{(j)} T_k c_k^{(j)} \big\|_{2,\tau} \le \|T_k\| \| c^{(j),2}_k \|_{2,\tau}
 \leq 2\tau\big((c_k^{(j)})^4\big)^{1/2} \le
2\tau(c_k^{(j)})^{1/2} \ \stackrel{\mathclap{\text{\ref{funct:i3}}}}{<}  \ 2\varepsilon_0<\varepsilon.
\]
as desired. 
Assume instead that $\tau \in V_k^{(j)}$. Set 
$D_k^{(j)} := \varphi_k^{(j)}(a^*a) - \varphi_k^{(j)}(a)^*
\varphi_k^{(j)}(a)$, which is positive by
the Schwarz inequality for completely positive maps (\cite[Corollary 2.8]{choi}). Then
\begin{align*}
\| c_k^{(j)} T_k c_k^{(j)} \|_{2, \tau} &\stackrel{\mathclap{\text{\ref{item:7.5i}}}}{\le}
\big\|c_k^{(j),2}D_k^{(j)} c_k^{(j),2} \big\|_{2,\tau} + 3\varepsilon_0 \\
&\le \big\| (D_k^{(j)})^{1/2} c_k^{(j),2} (D_k^{(j)})^{1/2} \big\|_{2,\tau}+ 3\varepsilon_0 \\
&\le \big\| D_k^{(j)} \big\|_{2,\tau} + 3\varepsilon_0 
\ \ \stackrel{\mathclap{\text{\ref{item:1c}}}}{<} \ \ \frac{\varepsilon}{2} +3\varepsilon_0 < \varepsilon,
\end{align*}
which concludes the proof of \ref{item:2c}.
Item \ref{item:2e} is verified in the same way as condition~(v) from \cite[Lemma 7.5]{kirchror}, and we omit the proof.
This proves the claim.

By iterating Claim \ref{claim:glue2}
for larger and larger $\Omega \subseteq A$, $G_0 \subseteq G$, and for smaller and smaller $\varepsilon > 0$,
we obtain $m = \dim(\partial_e T(A))$ completely positive, contractive, order zero maps
$\psi^{(0)}, \dots, \psi^{(m)}\colon M_d \to (A^\U \cap A')^{\alpha^\U}$ with commuting images.
Moreover, by condition \ref{item:7.5vi} on the $c_k^{(j)}$, it follows that
$\sum_{j=0}^m \psi^{(j)} (1) \le 1$, which in combination with \ref{item:2e} from Claim \ref{claim:glue2}
 yields $\sum_{j=0}^m \psi^{(j)} (1) = 1$. By \cite[Lemma 7.6]{kirchror} this gives a unital
homomorphism $\psi\colon M_d \to (A^\U \cap A')^{\alpha^\U}$, as desired.
\end{proof}

For the reader's convenience, we reproduce the statement of Theorem~\autoref{thmC} here.

\begin{thm} \label{thm:equiZstab}
	Let $A$ be a separable, simple, nuclear, $\mathcal{Z}$-stable \uca\ with 
	non-empty trace space.
	Let $G$ be a countable, discrete, amenable group,  
	and let $\alpha\colon G\to\Aut(A)$ be an action.
	Suppose that $\partial_eT(A)$ is compact, that $\dim(\partial_eT(A))<\I$,
that the orbits of the induced action of $G$ on $\partial_eT(A)$ are
finite with uniformly bounded cardinality, and that the orbit space $\partial_eT(A)/G$
is Hausdorff.
	Then $\alpha$ is cocycle conjugate to
	$\alpha\otimes\id_{\mathcal{Z}}$.
\end{thm}
\begin{proof}
By \autoref{prop:Gammabyhand}, for every $d \in \N$ there exists a unital homomorphism
\[
\varphi\colon M_d \to (A^\U \cap A')^{\alpha^\U}.
\]
The conclusion follows by the implication
\eqref{gamma.cpou:3} $\Rightarrow$ \eqref{gamma.cpou:4}
in \autoref{thm:mainCPoU}
\end{proof}

We record the following immediate 
combination of \autoref{thm:equiZstab} and 
\autoref{thm:mainRokDim}.
\begin{cor} \label{cor:last}
Let $A$ be a separable, simple, nuclear, $\mathcal{Z}$-stable \uca,
let $G$ be a countable, discrete, amenable group,  
and let $\alpha\colon G\to\Aut(A)$ be an action.
Suppose that $\partial_eT(A)$ is compact, that $\dim(\partial_eT(A))<\I$,
that the orbits of the induced action of $G$ on $\partial_eT(A)$ are
finite with uniformly bounded cardinality, and that the orbit space $\partial_eT(A)/G$
is Hausdorff. Then the following are equivalent:
\be
\item $\alpha$ is strongly outer.
\item  $\alpha$ has the weak tracial Rokhlin property.
\item when $G$ is residually finite, $\alpha$ has finite 
Rokhlin dimension. 
\item when $G$ is residually finite, we have $\dimRok(\alpha)\leq 2$.  
\ee 
\end{cor}

The difference between the above corollary and
\autoref{thm:mainRokDim} is that here we do not 
assume that $\alpha$ absorbs $\id_{\mathcal{Z}}$;
since $\partial_eT(A)$ is compact finite dimensional
and $\partial_e T(A)/G$ is Hausdorff,
this assumption follows from \autoref{thm:equiZstab}.


\begin{thebibliography}{CETW19}

\bibitem[BBS{\etalchar{+}}19]{bbstww}
J.~Bosa, N.~P. Brown, Y.~Sato, A.~Tikuisis, S.~White, and W.~Winter,
  \emph{Covering dimension of {$C^*$}-algebras and 2-coloured classification},
  Mem. Amer. Math. Soc. \textbf{257} (2019), no.~1233, vii+97. 

\bibitem[BK96]{beckech}
H.~Becker and A.~S. Kechris, \emph{The descriptive set theory of {P}olish group
  actions}, London Mathematical Society Lecture Note Series, vol. 232,
  Cambridge University Press, Cambridge, 1996. 

\bibitem[Bla06]{blackadar}
B.~Blackadar, \emph{Operator algebras}, Encyclopaedia of Mathematical Sciences,
  vol. 122, Springer-Verlag, Berlin, 2006, Theory of $C^*$-algebras and von
  Neumann algebras, Operator Algebras and Non-commutative Geometry, III.

\bibitem[BO08]{brownozawa}
N.~P. Brown and N.~Ozawa, \emph{{$C^*$}-algebras and finite-dimensional
  approximations}, Graduate Studies in Mathematics, vol.~88, AMS, Providence, RI, 2008. 
  
 \bibitem[Cho74]{choi}
M. D. Choi, \emph{A Schwarz inequality for positive linear maps on $C^{\ast}$-algebras}, Illinois J. Math. \textbf{18} (1974), 565--574.

\bibitem[CE78]{choieffros}
M.~D. Choi and E.~G. Effros, \emph{Nuclear {$C^\ast$}-algebras and the
  approximation property}, Amer. J. Math. \textbf{100} (1978), no.~1, 61--79.

\bibitem[CET{\etalchar{+}}21]{cpou}
J.~Castillejos, S.~Evington, A.~Tikuisis, S.~White, and W.~Winter,
  \emph{Nuclear dimension of simple {$C^*$}-algebras}, Invent. Math. \textbf{224} (2021), no. 1, 245--290.

\bibitem[CETW19]{uniformGamma}
J.~Castillejos, S.~Evington, A.~Tikuisis, and S.~White, \emph{Uniform property
  {$\Gamma$}}, preprint arXiv:1912.04207 (2019).

\bibitem[CJK{\etalchar{+}}18]{tilings}
C.~T. Conley, S.~C. Jackson, D.~Kerr, A.~S. Marks, B.~Seward, and R.~D.
  Tucker-Drob, \emph{F{\o}lner tilings for actions of amenable groups}, Math.
  Ann. \textbf{371} (2018), no.~1-2, 663--683. 
 
 \bibitem[Con75]{connes}
A.~Connes, \emph{Outer conjugacy classes of automorphisms of factors}, Ann. Sci. École Norm. Sup. (4) \textbf{8} (1975), no. 3, 383--419.

\bibitem[Dad09]{dadarlat09}
M.~Dadarlat, \emph{Continuous fields of {$C^*$}-algebras over finite
  dimensional spaces}, Adv. Math. \textbf{222} (2009), no.~5, 1850--1881.

\bibitem[DT09]{Dadarlat-Toms}
M.~Dadarlat and A.~S.~Toms.
\emph{{$\mathcal{Z}$}-stability and infinite tensor powers of
{$C^*$}-algebras},
 Adv. Math. \textbf{220} (2009) no.~2, 341--366.

\bibitem[Dix69]{dix}
J.~Dixmier, \emph{Quelques propri\'{e}t\'{e}s des suites centrales dans les
  facteurs de type {${\rm II}\sb{1}$}}, Invent. Math. \textbf{7} (1969),
  215--225. 

\bibitem[ELPW10]{EchLucPhiWal_structure_2010}
S.~Echterhoff, W.~L\"{u}ck, N.~C. Phillips, and S.~Walters, \emph{The structure
  of crossed products of irrational rotation algebras by finite subgroups of
  {${\rm SL}_2(\mathbb Z)$}}, J. Reine Angew. Math. \textbf{639} (2010), 173--221.

\bibitem[FGT21]{ForGarTho_asymptotic_2021}
M.~Forough, E.~Gardella, and K.~Thomsen,
\emph{Asymptotic lifting for completely positive
maps}. Preprint
arXiv:2103.09176 (2021).

\bibitem[FHH{\etalchar{+}}]{ttransfer}
I.~Farah, B.~Hart, I.~Hirshberg, C.~Schafhauser, A.~Tikuisis, and A.~Vaccaro,
  \emph{A tracial transfer property}, in preparation.

\bibitem[FHL{\etalchar{+}}]{modeltheory}
I.~Farah, B.~Hart, M.~Lupini, L.~Robert, A.~Tikuisis, A.~Vignati, and
  W.~Winter, \emph{Model theory of {$C^\ast$}-algebras}, Memoirs AMS (to
  appear).
  
\bibitem[Gar17]{Gar_rokhlin_2017}
E.~Gardella,
\emph{Rokhlin dimension for compact group actions},
Indiana Univ. Math. J., \textbf{66} (2017), no.~2, 659--703.  

\bibitem[GH18]{GarHir_strongly_2018}
E.~Gardella and I.~Hirshberg,
\emph{Strongly outer actions of amenable groups on 
$\mathcal{Z}$-stable $C^*$-algebras}. Arxiv preprint math.OA/1811.00447, 2018.

\bibitem[GHS17]{GarHirSan_rokhlin_2017}
E.~Gardella, I.~Hirshberg, and L.~Santiago
\emph{Rokhlin dimension: duality, tracial properties, and crossed products},
Ergodic Theory Dynam. Syst. \textbf{41} (2021), no. 2, 408--460.

\bibitem[GKL19]{GarKalLup_rokhlin_19}
E.~Gardella, M.~Kalantar, M.~Lupini, 
\emph{Rokhlin dimension for compact quantum group actions}. J. Noncommut. Geom. \textbf{13} (2019), no. 2, 711--767.

\bibitem[GL18a]{GarLup_rigid_2018}  
E.~Gardella and M.~Lupini, \emph{Actions of rigid groups on UHF-algebras}. J. Funct. Anal. \textbf{275} (2018), 381--421.

\bibitem[GL18b]{GarLup_applications_2018}  
E.~Gardella and M.~Lupini, \emph{Applications of model theory to 
$C^*$-dynamics}, 
J. Funct. Anal. \textbf{275} (2018), 1341--1354.


\bibitem[GS16]{GarSan_equivariant_2016}  
E.~Gardella and L.~Santiago, \emph{Equivariant *-homomorphisms, Rokhlin constraints and UHF-absorption}, 
J. Funct. Anal. \textbf{270} (2016), 2543--2590.


\bibitem[HJ82]{hermjon1}
R.~H. Herman and J.~F. R. Vaughan, \emph{Period two automorphisms of ${\rm UHF}$ $C\sp{\ast} $-algebras}, J. Funct. Anal. \textbf{45} (1982), no. 2, 169--176.

\bibitem[HJ83]{hermjon2}
R.~H. Herman and J.~F. R. Vaughan, \emph{Models of finite group actions}, Math. Scand. \textbf{52} (1983), no. 2, 312--320.

\bibitem[HO84]{herocn}
R.~H. Herman and A. Ocneanu, \emph{Stability for integer actions on UHF $C\sp{\ast} $-algebras}, J. Funct. Anal. \textbf{59} (1984), no. 1, 132--144.

\bibitem[HO12]{HirOro_tracially_2013}
I.~Hirshberg and J.~Orovitz,
\emph{Tracially {$\mathcal{Z}$}-absorbing {$C^*$}-algebras},
J. Funct. Anal., \textbf{265} (2013), no.~5, 765--785.

\bibitem[HP15]{HirPhi_rokhlin_2015}
I.~Hirshberg and N.~C. Phillips,
\emph{Rokhlin dimension: obstructions and permanence properties},
Doc. Math., \textbf{20} (2015), 199--236.

\bibitem[HW07]{hw07} I.~Hirshberg and W.~Winter, \emph{Rokhlin actions and 
self-absorbing {$C^*$}-algebras}, Pacific J. Math. \textbf{233} (2007), no.~1, 
125-–143. 

\bibitem[HWZ15]{hwz}
I.~Hirshberg, W.~Winter, and J.~Zacharias, \emph{Rokhlin dimension and
  {$C^*$}-dynamics}, Comm. Math. Phys. \textbf{335} (2015), no.~2, 637--670.
  
\bibitem[HSWW17]{HirSzaWinWu_rokhlin_2017}
 I.~Hirshberg, G.~Szab{\'o}, W.~Winter and J.~Wu,
 \emph{Rokhlin {D}imension for {F}lows},
Comm. Math. Phys., \textbf{353} (2017), no.~ 1, 253--316.

\bibitem[JT84]{jt84}
V. F. R. Jones, M. Takesaki, \emph{Actions of compact abelian groups on semifinite injective factors}, Acta Math. \textbf{153} (1984), no. 3-4, 213--258.

\bibitem[Izu04a]{Izu_finiteI_2004}
M.~Izumi,
\emph{Finite group actions on {$C^*$}-algebras with the {R}ohlin property.
	{I}},
 Duke Math. J., \textbf{122} (2004), no.~2, 233--280.

\bibitem[Izu04b]{Izu_finiteII_2004}
M.~Izumi,
\emph{Finite group actions on {$C^*$}-algebras with the {R}ohlin property.
	{II}},
Adv. Math., \textbf{184} (2004), no.~1, 119--160.

\bibitem[JS99]{jiangsu}
X.~Jiang and H.~Su, \emph{On a simple unital projectionless $C^*$-algebra. Amer}, J. Math. \textbf{121} (1999), no. 2, 359--413.

\bibitem[Kec95]{kechris}
A.~S. Kechris, \emph{Classical descriptive set theory}, Graduate Texts in
  Mathematics, vol. 156, Springer-Verlag, New York, 1995. 

\bibitem[Kis81]{Kis_simplicity}
A.~Kishimoto, \emph{Outer automorphisms and reduced crossed products of 
simple $C^*$-algebras}, Comm. Math. Phys. \textbf{81} (1981), 429--435.
  
\bibitem[Kis95]{kishimoto}
A.~Kishimoto, \emph{The Rohlin property for automorphisms of UHF algebras}, J. Reine Angew. Math. \textbf{465} (1995), 183--196.

\bibitem[Kis96]{kishimoto:I}
A.~Kishimoto, \emph{The Rohlin property for shifts on UHF algebras and automorphisms of Cuntz algebras}, J. Funct. Anal. \textbf{140} (1996), no. 1, 100--123.

\bibitem[Kis98]{kishimoto2}
A.~Kishimoto, \emph{Automorphisms of ${\mathrm A}\mathbf T$ algebras with the Rohlin property}, J. Operator Theory \textbf{40} (1998), no. 2, 277--294.

\bibitem[Kir06]{kir}
E.~Kirchberg, \emph{Central sequences in {$C^*$}-algebras and strongly purely
  infinite algebras}, Operator {A}lgebras: {T}he {A}bel {S}ymposium 2004, Abel
  Symp., vol.~1, Springer, Berlin, 2006, pp.~175--231. 

\bibitem[KR14]{kirchror}
E.~Kirchberg and M.~R{\o}rdam, \emph{Central sequence {$C^*$}-algebras and
  tensorial absorption of the {J}iang-{S}u algebra}, J. Reine Angew. Math.
  \textbf{695} (2014), 175--214. 

\bibitem[KW04]{kirchwin}
E.~Kirchberg and W.~Winter
\emph{Covering dimension and quasidiagonality},
Internat. J. Math. \textbf{15} (2004), no.~1, 63-–85. 

\bibitem[Lan75]{lance}
C.~Lance, \emph{Direct integrals of left {H}ilbert algebras}, Math. Ann.
  \textbf{216} (1975), 11--28. 

\bibitem[Lia16]{liao}
H.~C. Liao, \emph{A {R}okhlin type theorem for simple {$C^*$}-algebras of
  finite nuclear dimension}, J. Funct. Anal. \textbf{270} (2016), no.~10,
  3675--3708. 

\bibitem[Lia17]{liaoZm}
H.~C. Liao, \emph{Rokhlin dimension of {$\mathbb{Z}^m$}-actions on simple
  {$C^*$}-algebras}, Internat. J. Math. \textbf{28} (2017), no.~7, 1750050.

\bibitem[Lon81]{longo} R.~Longo, \emph{A remark on crossed product of 
{$C^*$}-algebras}, J. London Math. Soc. (2) \textbf{23} (1981), no.~3, 
531-–533. 

\bibitem[Lor97]{loring}
T.~A. Loring, \emph{Lifting solutions to perturbing problems in
  {$C^*$}-algebras}, Fields Institute Monographs, vol.~8, American Mathematical
  Society, Providence, RI, 1997. 

\bibitem[MS12a]{matuisato:I}
H.~Matui and Y.~Sato, \emph{{$\mathcal{Z}$}-stability of crossed products by
  strongly outer actions}, Comm. Math. Phys. \textbf{314} (2012), no.~1,
  193--228. 

\bibitem[MS12b]{matuisato:strict}
H.~Matui and Y.~Sato, \emph{Strict comparison and {$\mathcal{Z}$}-absorption of 
nuclear
  {$C^*$}-algebras}, Acta Math. \textbf{209} (2012), no.~1, 179--196.

\bibitem[MS14]{matuisato:II}
H.~Matui and Y.~Sato, \emph{{$\mathcal{Z}$}-stability of crossed products by 
strongly outer
  actions {II}}, Amer. J. Math. \textbf{136} (2014), no.~6, 1441--1496.

\bibitem[MT16]{mastom}
T.~Masuda and R.~Tomatsu, \emph{Rohlin flows on von {N}eumann algebras}, Mem.
  Amer. Math. Soc. \textbf{244} (2016), no.~1153, ix+111. 

\bibitem[MvN43]{mvnIV}
F.~J. Murray and J.~von Neumann, \emph{On rings of operators. {IV}}, Ann. of
  Math. (2) \textbf{44} (1943), 716--808. 

\bibitem[Ocn85]{ocneanu}
A.~Ocneanu, \emph{Actions of discrete amenable groups on von {N}eumann
  algebras}, Lecture Notes in Mathematics, vol. 1138, Springer-Verlag, Berlin,
  1985. 

\bibitem[Oza13]{ozawa_approx}
N.~Ozawa, \emph{Dixmier approximation and symmetric amenability for {$
  C^\ast$}-algebras}, J. Math. Sci. Univ. Tokyo \textbf{20} (2013), no.~3,
  349--374. 

\bibitem[Pea75]{pears}
A.~R. Pears, \emph{Dimension theory of general spaces}, Cambridge University Press, Cambridge, England-New York-Melbourne, 1975. {\rm xii}+428 pp.

\bibitem[Phi11]{phil:tracial}
N.~C. Phillips, \emph{The tracial {R}okhlin property for actions of finite
  groups on {$C^\ast$}-algebras}, Amer. J. Math. \textbf{133} (2011), no.~3,
  581--636. 

\bibitem[R{\o}r04]{rord}
M.~R{\o}rdam, \emph{The stable and the real rank of {$\mathcal{Z}$}-absorbing
  {$C^*$}-algebras}, Internat. J. Math. \textbf{15} (2004), no.~10, 1065--1084.

\bibitem[RW10]{RorWin_algebra_2010}
M.~R{\o}rdam and W.~Winter, \emph{The {J}iang-{S}u algebra revisited}, J. Reine
  Angew. Math. \textbf{642} (2010), 129--155. 

\bibitem[Run02]{runde}
V.~Runde, \emph{Lectures on amenability}, Lecture Notes in Mathematics, 1774. Springer-Verlag, Berlin, 2002. xiv+296 pp. ISBN: 3-540-42852-6 


\bibitem[San15]{San_crossed_2015}
L.~Santiago, \emph{Crossed product by actions of finite groups with the Rokhlin property}, Int. J. Math. \textbf{26} (2015), no. 7, 31 pp.

\bibitem[Sat10]{satoRoh}
Y.~Sato, \emph{The {R}ohlin property for automorphisms of the {J}iang-{S}u
  algebra}, J. Funct. Anal. \textbf{259} (2010), no.~2, 453--476. 

\bibitem[Sat19]{sato19}
Y.~Sato, \emph{Actions of amenable groups and crossed products of
  {$\mathcal{Z}$}-absorbing {$ C^*$}-algebras}, Operator algebras and
  mathematical physics, Adv. Stud. Pure Math., vol.~80, Math. Soc. Japan,
  Tokyo, 2019, pp.~189--210. 
  
\bibitem[Sut88]{suth}
C. E. Sutherland, \emph{Classifying actions of groups on von Neumann algebras}, 
Miniconferences on harmonic analysis and operator algebras (Canberra, 1987), 
317--328, Proc. Centre Math. Anal. Austral. Nat. Univ., \textbf{16}, Austral. 
Nat. Univ., Canberra, 1988.

\bibitem[ST85]{suthtak}
C. E. Sutherland, M. Takesaki, \emph{Actions of discrete amenable groups and 
groupoids on von Neumann algebras}, Publ. Res. Inst. Math. Sci. \textbf{21} 
(1985), no. 6, 1087--1120.


\bibitem[Sza18]{szabo:strself}
G.~Szab\'{o}, \emph{Strongly self-absorbing {$C^*$}-dynamical systems}, Trans.
Amer. Math. Soc. \textbf{370} (2018), no.~1, 99--130. 

\bibitem[Sza21]{equiSI}
G.~Szab\'o, \emph{Equivariant property ({SI}) revisited},  Anal. PDE 
\textbf{14} (2021), no.~4, 1199–-1232.

\bibitem[SWZ19]{swz}
G.~Szab\'{o}, J.~Wu, and J.~Zacharias, \emph{Rokhlin dimension for actions of
  residually finite groups}, Ergodic Theory Dynam. Systems \textbf{39} (2019),
  no.~8, 2248--2304. 

\bibitem[Tak02]{take1}
M.~Takesaki, \emph{Theory of operator algebras. {I}}, Encyclopaedia of
  Mathematical Sciences, vol. 124, Springer-Verlag, Berlin, 2002, Reprint of
  the first (1979) edition, Operator Algebras and Non-commutative Geometry, 5.
  


\bibitem[Tak03]{take3}
M.~Takesaki, \emph{Theory of operator algebras. {III}}, Encyclopaedia of
  Mathematical Sciences, vol. 127, Springer-Verlag, Berlin, 2003, Operator
  Algebras and Non-commutative Geometry, 8. 
  

\bibitem[Wan18]{wang}
Q.~Wang, \emph{The tracial {R}okhlin property for actions of amenable groups on
  {$C^*$}-algebras}, Rocky Mountain J. Math. \textbf{48} (2018), no.~4,
  1307--1344. 

\bibitem[Wil07]{crpr}
D.~P. Williams, \emph{Crossed products of {$C^\ast$}-algebras}, Mathematical
  Surveys and Monographs \textbf{134} (2007).
  
\bibitem[Win18]{WinterICM} 
W.~Winter, 
\emph{Structure of nuclear {$C^*$}-algebras: from quasidiagonality to 
classification and back again}, Proceedings of the International Congress of 
Mathematicians—Rio de Janeiro 2018. Vol. III. Invited lectures, 1801–-1823. 


\bibitem[Wou21]{wou}
L.~Wouters,
  \emph{Equivariant $\mathcal{Z}$-stability for single automorphisms on simple $C^\ast$-algebras with tractable trace simplices}, preprint
  arXiv:2105.04469 (2021).

\end{thebibliography}

\newcommand{\etalchar}[1]{$^{#1}$}
\providecommand{\bysame}{\leavevmode\hbox to3em{\hrulefill}\thinspace}
\providecommand{\MR}{\relax\ifhmode\unskip\space\fi MR }
\providecommand{\MRhref}[2]{%
  \href{http://www.ams.org/mathscinet-getitem?mr=#1}{#2}
}
\providecommand{\href}[2]{#2}

\end{document}